\documentclass[10pt,a4paper]{article}

\usepackage{amsmath}
\usepackage{amsfonts}
\usepackage{amsthm}
\usepackage[utf8]{inputenc}
\usepackage[T1]{fontenc} 
\usepackage{mathabx}
\usepackage{color}
\usepackage{dsfont}
\usepackage{stmaryrd}
\usepackage{hyperref}
\usepackage{graphics}
\usepackage{graphicx,epstopdf}

\usepackage{esint}
\usepackage{enumerate}
\usepackage{psfrag}
\usepackage{tikz}
\usepackage{setspace}

\newtheorem{theorem}{Theorem}
\newtheorem{lemma}[theorem]{Lemma}
\newtheorem{corollary}[theorem]{Corollary}
\newtheorem{remark}[theorem]{Remark}
\newtheorem{proposition}[theorem]{Proposition}

\newcommand{\R}{\mathbb{R}}

\newcommand{\N}{\mathbb{N}}
\newcommand{\eps}{\varepsilon}
\newcommand{\etaa}{\eta}
\newcommand{\dis}{\displaystyle}

\newcommand{\abil}{a^\star}
\newcommand{\weight}{\chi}

\begin{document}

\title{Stable approximation\\ of the advection-diffusion equation\\ using the invariant measure}
\author{Claude Le Bris, Fr\'ed\'eric Legoll and Fran\c cois Madiot
\\
Ecole des Ponts and INRIA,
\\
77455 Marne-la-Vall\'ee, France 
\\
{\small \tt \{claude.le-bris,frederic.legoll,francois.madiot\}@enpc.fr}}

\date{\today}

\maketitle

\begin{abstract}
We consider an advection-diffusion equation that is both non-coercive and advection-dominated. We present a possible numerical approach, to our best knowledge new, and based on the invariant measure associated to the original equation. The approach has been summarized in~\cite{cras-madiot}. We show that the approach allows for an unconditionally well-posed finite element approximation. We provide a numerical analysis and a set of comprehensive numerical tests showing that the approach can be stable, as accurate as, and more robust than a classical stabilization approach. 
\end{abstract}


\section{Introduction and motivation}
Our purpose is to study the advection-diffusion equation 
\begin{equation}
-\Delta u+ b\cdot\nabla u= f \quad\text{in $\Omega$}, 
\label{pb_adv_diff_non_coercive_without_bc}
\end{equation}
more specifically in the regime where it is both non-coercive and possibly unstable. We present a new numerical strategy, based upon the utilization of the invariant measure associated to~\eqref{pb_adv_diff_non_coercive_without_bc}, namely the solution~$\sigma$ to the adjoint equation
\begin{equation}
-\text{div}(\nabla \sigma + b\sigma) = 0\quad\text{in $\Omega$},
\label{pb_adv_diff_non_coercive_without_bc-adjoint}
\end{equation}
supplied with suitable boundary conditions and normalization constraints.

Equation~\eqref{pb_adv_diff_non_coercive_without_bc} arises in a huge variety of contexts of the engineering sciences, either in its stationary form~\eqref{pb_adv_diff_non_coercive_without_bc}, or in its time-dependent form. It may also contain a reaction term~$c\,u$, with $c\geq 0$. The second order (diffusion) operator can be chosen more general than a pure Laplacian, in the form of a divergence operator $-\text{div}(A\,\nabla \cdot)$, with a suitable matrix-valued function $A$. All such situations proceed from straightforward applications of our discussions below, which, for brevity and clarity, we limit to the simple, stationary case~\eqref{pb_adv_diff_non_coercive_without_bc}.

A typical difficulty associated with equation~\eqref{pb_adv_diff_non_coercive_without_bc} is the possible lack of coercivity of the bilinear form, owing to the presence of the advection term $b \cdot \nabla u$. More severely, not only coercivity, but also stability may be affected by the advection term, when the latter is ``large'' in a certain sense. The equation is then said advection-dominated. Studies abound in the literature, that describe the theory necessary to prove well-posedness of that problem under those difficult circumstances. Similarly, the works presenting possible numerical discretization techniques specifically targeted to this context are countless. Our purpose here is to propose yet another way of addressing the difficulties mentioned above. The computational approach we present actually originates from the theoretical proof of well-posedness of the problem.

As is well-known, a classical proof of the well-posedness of~\eqref{pb_adv_diff_non_coercive_without_bc}, when it is not coercive, proceeds by the Fredholm alternative. On the other hand, well-posedness is also, using the Banach-Ne\v{c}as-Babu\v{s}ka Theorem, equivalent to a set of two conditions, namely the inf-sup condition
\begin{equation}
\label{eq:inf-sup}
\exists\alpha>0 \quad \hbox{\rm such that} \quad \inf_{w\in W} \ \sup_{v\in V} \ \frac{a(w,v)}{\|w\|_W \|v\|_V} \geq \alpha>0
\end{equation}
where here
\begin{equation}
\label{eq:def_bil_a}
a(w,v) = \int_\Omega \nabla w \cdot \nabla v + \int_\Omega (b \cdot \nabla w) \, v
\end{equation}
and~$V=W=H^1_0(\Omega)$, and the additional condition 
\begin{equation}
\label{eq:inf-sup_encore}
\forall v \in V, \ \ (\forall w \in W, \ a(w,v)=0) \Rightarrow v=0.
\end{equation}
We refer to, e.g.,~\cite[p.~85]{EG} for a comprehensive exposition of the theory and general references therein for the study and approximation of~\eqref{pb_adv_diff_non_coercive_without_bc}. We shall recall some basic facts in Section~\ref{ssec:inf-sup} of this article. 

In practice, two questions arise: to prove that the above conditions~\eqref{eq:inf-sup} and~\eqref{eq:inf-sup_encore} hold, thereby providing a proof of well-posedness that is independent from Fredholm theory, and to make the constant~$\alpha$ in~\eqref{eq:inf-sup} explicit. For this twofold purpose, the classical argument is to consider the invariant measure associated to~\eqref{pb_adv_diff_non_coercive_without_bc}, namely the solution~$\sigma$ to~\eqref{pb_adv_diff_non_coercive_without_bc-adjoint}, satisfying~$\sigma(x) \geq \underline{\sigma} > 0$ almost everywhere in~$\Omega$, $\displaystyle |\Omega|^{-1} \, \int_\Omega \sigma=1$, along with an adequate boundary condition (think of the natural Neumann boundary condition, but other boundary conditions might be considered, in particular because of the various boundary conditions~\eqref{pb_adv_diff_non_coercive_without_bc} itself may be supplied with). The existence and uniqueness of a suitable~$\sigma$ follows from the Fredholm theory. In short, \eqref{eq:inf-sup} is then obtained as follows. One multiplies~\eqref{pb_adv_diff_non_coercive_without_bc} by the product~$\sigma\,v$ and integrates over the domain $\Omega$:
\begin{equation}
\label{eq:integration}
\int_\Omega (-\Delta u + b \cdot \nabla u) \, \sigma v
=
\int_\Omega \sigma \, \nabla u \cdot \nabla v + \int_\Omega (\nabla\sigma+\sigma\,b) \, \cdot \, \nabla u \, v - \int_{\partial\Omega} (\nabla u \cdot n) \, \sigma \, v.
\end{equation}
The rightmost term cancels out when homogeneous Dirichlet boundary conditions are imposed on $u$ (and thus on the test function~$v$), which is the setting we adopt throughout this article. Considering~\eqref{pb_adv_diff_non_coercive_without_bc-adjoint}, this formally yields
\begin{equation}
\label{eq:calcul-adjoint}
a(u,\sigma\,u) = \int_\Omega \sigma \, |\nabla u|^2 \quad \text{for any $u \in H^1_0(\Omega)$},
\end{equation}
which readily implies~\eqref{eq:inf-sup}, as soon as $\sigma$ is positive and bounded away from zero. The classical approach is then to use a finite element discretization that also satisfies the inf-sup condition~\eqref{eq:inf-sup}, at least for a sufficiently small mesh size~$h$, and is therefore, ``by continuity'', also well-posed, using the same type of argument. 

The above observation on how the consideration of the invariant measure allows to transform the original problem~\eqref{pb_adv_diff_non_coercive_without_bc} into a coercive problem seems to not have been exploited computationally (except in the very specific case when $b$ is irrotational~\cite{brezzi}, for which $\sigma$ is then analytically known). This is our purpose to do so. Formally, \eqref{eq:integration} and~\eqref{eq:calcul-adjoint} suggest a Petrov-Galerkin formulation of the problem using test functions of the form $\sigma \, v$ instead of a classical Galerkin formulation. (Formally) equivalently, one may perform a Galerkin approximation of the modified equation
\begin{equation}
\label{eq:adv-diff-modifiee}
-\hbox{\rm div}(\sigma\,\nabla u) + (\nabla\sigma + \sigma \, b) \cdot \nabla u=\sigma \, f.
\end{equation}
The point is of course that the modified advection field
$$
B = \nabla\sigma+\sigma\,b
$$ 
is divergence-free because of~\eqref{pb_adv_diff_non_coercive_without_bc-adjoint}. Problem~\eqref{eq:adv-diff-modifiee}, complemented by homogeneous Dirichlet boundary conditions, is consequently coercive. And the numerical analysis of its finite element approximation is amenable to standard arguments. 

\medskip

The definite added value of the approach is that it provides an unconditionally well-posed approximation, irrespective of the discretization parameter --the meshsize-- adopted for approximating $u$, provided $\sigma$ itself is correctly approximated (which in particular implies that some positivity of $\sigma$ is preserved at the discrete level). This unconditional well-posedness may be most useful in problems where one can only afford a coarse approximation of $u$. Multiscale problems, where the Laplacian operator is replaced by $-\hbox{\rm div}(a(x/\varepsilon) \, \nabla\, \cdot)$, are prototypical examples of such a context. Problems such as inverse problems, or time-dependent problems (once semi-discretized in time using, say, an implicit Euler scheme), where the solution to the advection-diffusion equation is required repeatedly, are also problems of choice for the approach. A rather coarse approximation might be employed, while the additional computational workload to solve the adjoint equation~\eqref{pb_adv_diff_non_coercive_without_bc-adjoint} is required only once. A definite improvement of the total computational time may be observed. In addition, in the advection dominated context, the approach enjoys particular stability properties that lead to numerical results qualitatively comparable to those obtained with classical, state-of-the-art stabilization approaches~\cite{franca-frey,johnson,quarteroni1994numerical}. Our results show that the approach is accurate, robust and can be made effective in terms of computational cost. Applications to several other, more general contexts, may be envisioned.

On the theoretical level, one advantage of the approach is that we can establish (see Section~\ref{sec:discretization}) a complete numerical analysis. In contrast, and to the best of our knowledge, the added value of a classical \emph{stabilized} finite element approximation is not proven theoretically when the advection-diffusion equation is \emph{not} coercive. 

\medskip

Our article is articulated as follows. 
Section~\ref{sec:mathematical-setting} collects some preparatory material. We need to recall a few results, first on the Banach-Ne\v{c}as-Babu\v{s}ka theory and the inf-sup condition, and second on the invariant measures that may be associated to the problem. The former ones are very classical, and we briefly overview them in Section~\ref{ssec:inf-sup}. The latter ones are slightly less standard and are the purpose of Section~\ref{ssec:invariant}. Of course, the reader familiar with the theory may easily skip our recollection and directly proceed to Section~\ref{sec:discretization}, where we present the specific discretization we use and analyze theoretically our approximation strategy. Our main result is Theorem~\ref{prop_convergence_sigma_h_p1}. An ingredient of our numerical analysis is the adaptation to the case of the advection-diffusion equation with invariant measure~$\sigma$ of classical arguments from~\cite{BS} that allow for an error estimate in $W^{1,p}(\Omega)$, $p>2$, of a $\mathbb{P}^1$ finite element approximation of $\sigma$ (see Proposition~\ref{prop_brenner_scott_estimate} below). This extension is, in our opinion, nontrivial and has an interest on its own. We present its proof in Appendix~\ref{sec:appendix_BS}. Our final Section~\ref{sec:numerical-results} presents a comprehensive set of numerical tests that demonstrate the accuracy, stability and efficiency of the approach. 
 
\medskip
 
Our main conclusions are as follows. The approach based upon the precomputation of an approximation of~$\sigma$ solution to~\eqref{pb_adv_diff_non_coercive_without_bc-adjoint} and next the approximation of~$u$ as the solution to~\eqref{eq:adv-diff-modifiee} provides with an approximation that is (i) well-posed unconditionally in the meshsize and amenable to a precise numerical analysis of convergence, and (ii) as stable and accurate as a direct typical stabilized solution of~\eqref{pb_adv_diff_non_coercive_without_bc}. It is slightly more robust with respect to the mesh size used for the approximation of $u$. If the precomputation time for~$\sigma$ is not accounted for, the approach has an equal computational cost. And if it is, then roughly ten solutions of the advection-diffusion equation are necessary to make the approach profitable. In many of the contexts we mentioned above, this is clearly the case. 
 
\medskip
 
In short, the approach presented here does not provide spectacular results but constitutes an interesting, certified and efficient alternative to more established approaches. It can be shown to outperform them in certain situations when only a coarse mesh can be afforded and/or when the advection-diffusion equation needs to be solved repeatedly.
 
\medskip

The present work complements an earlier publication~\cite{cras-madiot} which already summarized the approach. It provides the numerical analysis of the approach and extensive numerical tests to assess its performance. 

%
%
%
%
%

\section{Mathematical setting and theoretical results}
\label{sec:mathematical-setting}

We consider equation~\eqref{pb_adv_diff_non_coercive_without_bc} on a domain $\Omega$ and for an advection field $b$ that \emph{at least} satisfy, throughout the article, the following two conditions:
\begin{equation}
\label{hyp_H0_debut}
\left\{
\begin{aligned}
&\text{$\Omega$ is an open bounded domain of $\R^d$, $d\geq 2$};
\\
&b\in(L^\infty(\Omega))^d.
\end{aligned}
\right.
\end{equation}
We additionnally assume, in this Section~\ref{sec:mathematical-setting}, that
\begin{equation}
\label{hyp_H0_fin}
\text{$\Omega$ is of class $\mathcal{C}^1$}.
\end{equation}
For some of our results of Section~\ref{sec:mathematical-setting}, we will have to make stronger assumptions (see in particular~\eqref{hyp_H_11} below).

The right-hand side $f$ of equation~\eqref{pb_adv_diff_non_coercive_without_bc} is assumed $H^{-1}(\Omega)$. Again, in some instances, we will need to assume a better regularity (typically $L^p(\Omega)$) on this function~$f$. 

The boundary conditions we supply~\eqref{pb_adv_diff_non_coercive_without_bc} with affect the boundary conditions we need to impose in the definition of the invariant measure(s) we introduce. For simplicity, our discussion assumes homogeneous Dirichlet boundary conditions
\begin{equation}
\label{boundary_cond_dirichlet_homog}
u=0 \quad\text{on $\partial\Omega$},
\end{equation}
although other boundary conditions may be considered. As the definition of the invariant measure $\sigma$ is essentially a matter of integration by parts (in the spirit of~\eqref{eq:integration} above), we leave to the reader the adaptation to other boundary conditions. We emphasize, however, that some of the arguments that follow might require some additional work. 

For the sake of consistency, we are now about to recall a set of basic results we need, on the inf-sup theory and on the invariant measure. The results are in particular interesting to motivate the specific discretization approach we introduce. We reiterate that the reader familiar with the classical theory may easily skip the sequel and directly proceed to Section~\ref{sec:discretization}. 

\subsection{Inf-sup theory}
\label{ssec:inf-sup}

The advection-diffusion equation~\eqref{pb_adv_diff_non_coercive_without_bc} supplied with the data we have just described and the boundary condition~\eqref{boundary_cond_dirichlet_homog} can be studied in the context of the Banach-Ne\v{c}as-Babu\v{s}ka theory. Defining $U=V=H^1_0(\Omega)$, 
\begin{align}
a(u,v)&=\int_\Omega \nabla u\cdot\nabla v + (b\cdot\nabla u) v,
\label{def_a}
\\
F(v)&=\left\langle f, v\right\rangle_{H^{-1}(\Omega),H^1_0(\Omega)}, 
\label{def_F}
\end{align}
it leads to a particular case of the general variational formulation:
\begin{equation}
\text{Find $u\in U$ such that, for all $v\in V$, \quad $a(u,v)=F(v)$},
\label{varf_coercivity_loss_continuous}
\end{equation}
where $U$ is a Banach space, $V$ is a reflexive Banach space, $a\in\mathcal{L}(U\times V;\R)$ and $F\in V'$. 
The well-posedness of~\eqref{pb_adv_diff_non_coercive_without_bc}--\eqref{boundary_cond_dirichlet_homog}, recast as~\eqref{def_a}--\eqref{def_F}--\eqref{varf_coercivity_loss_continuous}, is known to be equivalent to the following two conditions:
\begin{enumerate}[$(\textnormal{BNB}1)$]
\item There exists $\alpha>0$ such that $\displaystyle \inf_{u\in H^1_0(\Omega)} \ \sup_{v\in H^1_0(\Omega)} \ \frac{a(u,v)}{\|u\|_{H^1(\Omega)}\|v\|_{H^1(\Omega)}} \geq \alpha;
$ 
\item For all $v\in H^1_0(\Omega)$, $\displaystyle \left(\forall u\in H^1_0(\Omega), \ a(u,v)=0 \right) \Longrightarrow (v=0).
$	
\end{enumerate}

\medskip

\noindent
Introducing a function~$\sigma$ satisfying
\begin{enumerate}[(C1)]
\item $-\text{div }(\nabla \sigma + b\sigma)=0$ in $\Omega$,
\item $\displaystyle \inf_\Omega \sigma>0$,
\item for all $u\in H^1_0(\Omega)$, ($\sigma \, u \in H^1_0(\Omega)$ and) $\|\sigma \, u\|_{H^1(\Omega)} \leq C_\sigma \|u\|_{H^1(\Omega)}$,
\end{enumerate}
it is easy to see that the conditions~$(\textnormal{BNB}1)$ and~$(\textnormal{BNB}2)$ are satisfied in our case. Using~\textnormal{(C3)} and~\textnormal{(C1)}, a simple calculation indeed yields, for any $u\in H^1_0(\Omega)$,
\begin{align*}
a(u,\sigma u) 
&= 
\int_\Omega\sigma \nabla u \cdot\nabla u -\int_\Omega \text{div }(\nabla \sigma + b\sigma)\frac{u^2}{2}
\\
&=\int_\Omega\sigma \nabla u \cdot\nabla u 
\\
&\geq \left(\inf_\Omega \sigma\right)\|\nabla u\|_{L^2(\Omega)}^2
\\
&\geq C\,\left(\inf_\Omega \sigma\right) \|u\|_{H^1(\Omega)}^2.
\end{align*}
Using~\textnormal{(C2)} and~\textnormal{(C3)}, we obtain 
$$
\frac{a(u,\sigma u)}{\|u\|_{H^1(\Omega)}\|\sigma u\|_{H^1(\Omega)}} \geq C \frac{\left(\inf_\Omega \sigma\right)}{C_\sigma} > 0
$$
and thus the inf-sup inequality~$(\textnormal{BNB}1)$. The second condition~$(\textnormal{BNB}2)$ is a consequence of the maximum principle (see e.g.~\cite[Theorem 8.1]{gilbarg2001elliptic}): a function $v \in H^1_0(\Omega)$ that satisfies $a(u,v)=0$ for all $u\in H^1_0(\Omega)$ is a solution to $-\text{div} \, \left(\nabla v+ b v\right)= 0$ in $\Omega$ and therefore vanishes.

\medskip

The following, very classical proposition collects the properties established. 

\begin{proposition} 
\label{prop_inf_sup}
We assume~\eqref{hyp_H0_debut}--\eqref{hyp_H0_fin} and that there exists $\sigma\in H^1(\Omega)$ satisfying conditions~\textnormal{(C1)}, \textnormal{(C2)} and~\textnormal{(C3)}. Then~\eqref{pb_adv_diff_non_coercive_without_bc}--\eqref{boundary_cond_dirichlet_homog} is well-posed, that is, it has a unique solution in $H^1_0(\Omega)$, and the map $H^{-1}(\Omega)\ni f \mapsto u\in H^1_0(\Omega)$ is continuous. 
\end{proposition}

Simply observing that one may harmlessly multiply and divide by a function~$\sigma$ enjoying the properties~\textnormal{(C2)} and~\textnormal{(C3)}, and that the following formal integration by parts holds
\begin{align*}
\int_\Omega(-\Delta u + b\cdot\nabla u)\,\sigma v
&= -\int_{\partial\Omega}(\nabla u\cdot n)\, \sigma v+\int_\Omega\nabla u\cdot\nabla(\sigma v) +\int_\Omega (\sigma\,b\cdot\nabla u ) \,v
\\
&= -\int_{\partial\Omega}(\sigma\nabla u\cdot n )\, v+\int_\Omega(\sigma \nabla u)\cdot\nabla v + \int_\Omega((\nabla \sigma + b\sigma)\cdot\nabla u)\, v
\\
&= \int_{\Omega}(-\textnormal{div} (\sigma \nabla u))\,v+\int_\Omega((\nabla \sigma + b\sigma)\cdot\nabla u)\, v,
\end{align*}
one readily obtains the following result.

\begin{proposition}
\label{prop_equiv_modified_problem}
Under the assumptions of Proposition~\ref{prop_inf_sup}, \eqref{pb_adv_diff_non_coercive_without_bc}--\eqref{boundary_cond_dirichlet_homog} is equivalent to the problem
\begin{equation}
-\textnormal{div} (\sigma \nabla u) + (\nabla \sigma + b \sigma)\cdot\nabla u = \sigma f \quad\text{in $\Omega$}, \qquad u=0 \quad\text{on $\partial\Omega$},
\label{eq_modified_problem}
\end{equation}
which is therefore also well-posed.
\end{proposition}
Note that property~\textnormal{(C1)} is actually not required to show that~\eqref{pb_adv_diff_non_coercive_without_bc}--\eqref{boundary_cond_dirichlet_homog} and~\eqref{eq_modified_problem} are equivalent. This equivalence thus also holds for an approximation of the invariant measure, a fact we will use in the sequel.

\subsection{On the invariant measure}
\label{ssec:invariant}

We now turn to elements of theory regarding the invariant measure $\sigma$ solution to~\eqref{pb_adv_diff_non_coercive_without_bc-adjoint}. As mentioned above, \eqref{pb_adv_diff_non_coercive_without_bc-adjoint} does not completely characterize $\sigma$ since a boundary condition and possibly an additional normalization need to be supplied. Because of the homogeneous Dirichlet boundary condition~\eqref{boundary_cond_dirichlet_homog} we have imposed on $u$ for~\eqref{pb_adv_diff_non_coercive_without_bc}, it turns out that we have some flexibility on the boundary condition we may impose on~$\sigma$. In other situations, the integration by parts performed to employ the adjoint equation might require more stringent conditions on~$\sigma$ on the boundary. The adaptation is, as we said, left to the reader.

\medskip

We are going to consider two different choices for $\sigma$, respectively studied in the next two sections. 

\subsubsection{First choice of invariant measure}
\label{sssec:invariant1}

The first case is (formally) defined by
\begin{equation}
\left\{
\begin{aligned}
&-\textnormal{div }(\nabla \sigma+ b \sigma)= 0 \quad\text{in $\Omega$}, 
\\
&(\nabla \sigma+ b \sigma)\cdot n=0 \quad\text{on $\partial\Omega$},
\end{aligned}
\right.
\label{invariant_measure_droniou}
\end{equation}
with the normalization constraint 
\begin{equation}
\fint_\Omega \sigma:=|\Omega|^{-1} \, \int_\Omega \sigma =1,
\label{invariant_measure_droniou-int-1}
\end{equation}
and the property
\begin{equation}
\inf_\Omega\sigma\geq c>0.
\label{invariant_measure_droniou-pos}
\end{equation}
The existence (and uniqueness) of such a function $\sigma$, along with some additional regularity properties, is established below. In the subsequent sections, this function is denoted by~$\sigma_1$. 

\medskip

The classical results concerning this case are contained in the following Lemma. 

\begin{lemma}[Theorem~1.1 of~\cite{droniou2009noncoercive}]
\label{theorem_droniou}

In addition to~\eqref{hyp_H0_debut}--\eqref{hyp_H0_fin}, we assume that $\Omega$ is connected (for the uniqueness part of our statements). Then
\begin{enumerate}[i)]
\item There exists $\sigma\in H^1(\Omega)$, unique up to a multiplicative constant, solution to~\eqref{invariant_measure_droniou}. Up to a change of sign, we may require $\sigma>0$ a.e. on $\Omega$;
\item Let $g \in (H^1(\Omega))'$. The problem 
\begin{equation}
  \label{eq:pourquoi_moi}
  \left\{
  \begin{aligned}
    & \text{Find $v \in H^1(\Omega)$ such that, for any $\varphi \in H^1(\Omega)$},
    \\
    & \int_\Omega (\nabla \varphi)^T (\nabla v+ b v) = \langle g, \varphi \rangle_{(H^1(\Omega))',H^1(\Omega)},
  \end{aligned}
  \right.
\end{equation}
admits at least one solution if and only if 
$$
\left\langle g, 1\right\rangle_{(H^1(\Omega))',H^1(\Omega)}=0.
$$
In this case, the set of all solutions is $v+\R \sigma$. In addition, the application $g\mapsto v$ is a bounded linear map from 
$$
V_{\# 1} = \left\{ g\in(H^1(\Omega))', \quad \left\langle g, 1\right\rangle_{(H^1(\Omega))',H^1(\Omega)}=0 \right\}
$$
to 
$$
H^1_{\int=0}(\Omega) = \left\{ v\in H^1(\Omega), \quad \int_\Omega v=0 \right\};
$$
\item Let $f \in (H^1(\Omega))'$. The problem 
\begin{equation}
  \label{eq:pourquoi_moi2}
\left\{
\begin{aligned}
  & \text{Find $u \in H^1(\Omega)$ such that, for any $\varphi \in H^1(\Omega)$},
    \\
    & \int_\Omega \nabla \varphi \cdot \nabla u + \varphi b \cdot \nabla u = \langle f, \varphi \rangle_{(H^1(\Omega))',H^1(\Omega)},
\end{aligned}
\right.
\end{equation}
admits at least one solution (which is unique up to the addition of a constant) if and only if 
$$
\left\langle f, \sigma\right\rangle_{(H^1(\Omega))',H^1(\Omega)}=0.
$$
In addition, the application $f\mapsto u$ is a bounded linear map from 
$$ 
V_{\# \sigma} = \left\{f \in (H^1(\Omega))', \quad \left\langle f, \sigma\right\rangle_{(H^1(\Omega))',H^1(\Omega)}=0 \right\}
$$ 
to $H^1_{\int=0}(\Omega)$.
\end{enumerate}
\end{lemma}

\noindent
It is easily seen that all assertions are consequences of the Fredholm alternative. The positivity stated in (i) follows from the maximum principle. A bound from below on~$\sigma$ may be obtained using stronger assumptions. It is the purpose of the following lemma. 

\begin{lemma}[after Theorem~1 of~\cite{perthame1990perturbed}]
\label{theorem_perthame}

We assume~\eqref{hyp_H0_debut} and that the domain $\Omega$ is of class $\mathcal{C}^2$. Then there exists a unique solution $\sigma\in W^{1,p}(\Omega) \cap \mathcal{C}^0(\overline{\Omega})$, for all $1\leq p<+\infty$, to~\eqref{invariant_measure_droniou} that satisfies $\dis \max_{\overline{\Omega}} \sigma =1$. In addition, there exists $c>0$ so that
$$
\sigma\geq c\quad \text{in $\Omega$}.
$$
This solution $\sigma$ satisfies the conditions~\textnormal{(C1)}, \textnormal{(C2)} and~\textnormal{(C3)}.
\end{lemma}

We immediately remark that the solution, the existence and uniqueness of which is established in Lemma~\ref{theorem_perthame}, co\"\i ncides with one of the positive solutions dealt with in Lemma~\ref{theorem_droniou}. We note also that a renormalization of that solution can be performed to comply with the constraint~\eqref{invariant_measure_droniou-int-1}. 

\medskip

Before we are in position to state our main proposition regarding the measure defined in~\eqref{invariant_measure_droniou}--\eqref{invariant_measure_droniou-int-1}--\eqref{invariant_measure_droniou-pos}, we need to recall a technical lemma for the Neumann problem, and next to strengthen the assumptions~\eqref{hyp_H0_debut}--\eqref{hyp_H0_fin}. The technical lemma, which will be useful in our proof of Proposition~\ref{prop:sigma1} below, is the following. 

\begin{lemma}[Chapter~I, Theorem~1.10 of~\cite{girault1986finite}]
\label{theorem_girault}

Let $\Omega$ be an open, bounded and connected domain of $\R^d$ with a $\mathcal{C}^{1,1}$ boundary. Let $1 < p < \infty$ and let $u$ be a solution to
$$
\left\{
\begin{aligned}
&-\Delta u =f \quad \text{in $\Omega$}, 
\\
&\nabla u\cdot n = g \quad\text{on $\partial\Omega$},
\end{aligned}
\right.
$$
where $f\in L^p(\Omega)$ and $g\in W^{1-1/p,p}(\partial\Omega)$ satisfy the relation $\displaystyle \int_\Omega f+ \int_{\partial\Omega} g = 0$. Then $u\in W^{2,p}(\Omega)$ and there exists a constant $C$, depending on $p$ and $\Omega$ but independent from $f$, $g$ and $u$, such that 
$$
\left\| u - \fint_\Omega u\right\|_{W^{2,p}(\Omega)} \leq C \Big( \|f\|_{L^p(\Omega)} + \|g\|_{W^{1-1/p,p}(\partial\Omega)} \Big).
$$
\end{lemma}

\noindent In the sequel of this Section~\ref{sec:mathematical-setting}, we assume that
\begin{equation}
\left\{
\begin{aligned}
&\text{the domain $\Omega$ is connected, and of class $\mathcal{C}^2$};
\\
&\text{$b$ is Lipschitz-continuous on $\overline{\Omega}$}. 
\end{aligned}
\right.
\label{hyp_H_11}
\end{equation}
We have the following result:

\begin{proposition}
\label{prop:sigma1}

Under assumptions~\eqref{hyp_H0_debut}--\eqref{hyp_H_11}, there exists a unique solution $\sigma\in H^1(\Omega)$ to~\eqref{invariant_measure_droniou} with the normalization~\eqref{invariant_measure_droniou-int-1}. In addition, it satisfies~\eqref{invariant_measure_droniou-pos}. Furthermore, for all $1 < p < +\infty$, we have $\sigma\in W^{2,p}(\Omega)\cap\mathcal{C}^1(\overline{\Omega})$ and the estimate
\begin{equation}
\label{eq:estimateW12}
\|\sigma\|_{W^{2,p}(\Omega)}\leq C\,(1+\,\|\sigma\|_{W^{1,p}(\Omega)}),
\end{equation}
where $C$ is a constant depending on $p$, $\Omega$ and 
$$
\|b\|_{\text{Lip}(\overline{\Omega})} = \sup_{x \in \overline{\Omega}} |b(x)| + \sup_{x\neq y \in \overline{\Omega}} \frac{|b(x)-b(y)|}{|x-y|}.
$$
\end{proposition}

\begin{proof}
Lemma~\ref{theorem_perthame} yields the existence and uniqueness of $\sigma \in W^{1,p}(\Omega) \cap \mathcal{C}^0(\overline{\Omega})$, $1\leq p<+\infty$, solution to~\eqref{invariant_measure_droniou}--\eqref{invariant_measure_droniou-int-1}, and states that~\eqref{invariant_measure_droniou-pos} holds. The point here is to prove that $\sigma\in W^{2,p}(\Omega)$ for all $1 < p < +\infty$, and that the estimate~\eqref{eq:estimateW12} holds true. Since 
$$
\left\{
\begin{aligned}
&-\Delta \sigma =\text{div }( b\sigma) \ \ \text{in $\Omega$}, \qquad \fint_\Omega \sigma =1, 
\\
&\nabla \sigma\cdot n = -(b\cdot n)\sigma \ \ \text{on $\partial\Omega$},
\end{aligned}
\right.
$$
we may apply Lemma~\ref{theorem_girault}. Using in particular that $b\in W^{1,\infty}(\Omega)$, we obtain that
$$
\|\text{div }(b\sigma)\|_{L^p(\Omega)} 
\leq 
\|b\|_{W^{1,\infty}(\Omega)} \|\sigma\|_{L^p(\Omega)} + \|b\|_{L^\infty(\Omega)}\|\nabla \sigma\|_{L^p(\Omega)}
\leq 
C \, \|\sigma\|_{W^{1,p}(\Omega)}
$$
using the Leibniz formula and the H\"older inequality, while
$$
\| (b\cdot n) \sigma\|_{W^{1-1/p,p}(\partial\Omega)} \leq C \|\sigma\|_{W^{1,p}(\Omega)} 
$$
is a consequence of the Sobolev trace theorems and the Lipschitz regularity of $b$ on the closed domain $\overline{\Omega}$. Lemma~\ref{theorem_girault} therefore implies $\sigma\in W^{2,p}(\Omega)$ and~\eqref{eq:estimateW12}. Since $p$ can be taken arbitrary large, we also have $\sigma \in \mathcal{C}^1(\overline{\Omega})$.
\end{proof}

For the numerical analysis of the approach performed in Section~\ref{sec:discretization}, we need the following extension of Proposition~\ref{prop:sigma1}, making precise the continuity of the solutions to the advection-diffusion equation and its adjoint equation with respect to their respective right-hand sides. 

\begin{proposition}
\label{prop_u_f_A_id}

Let $p$ be such that $1< p<+\infty$ if $d=2$ and $2d/(d+2)\leq p<+\infty$ otherwise. Assuming~\eqref{hyp_H0_debut}--\eqref{hyp_H_11} and letting $\sigma$ be the unique solution to~\eqref{invariant_measure_droniou}--\eqref{invariant_measure_droniou-int-1}, we have the following results:

\begin{enumerate}[i)]

\item for all $f\in L^p(\Omega)$ such that $\displaystyle \int_\Omega f=0$, there exists a unique $v\in H^1(\Omega)$ solution to
  \begin{equation}
    \label{eq:pourquoi_moi4}
\left\{
\begin{aligned}
&-\textnormal{div }(\nabla v + bv)=f \ \ \text{in $\Omega$}, \qquad \fint_\Omega v =1, 
\\
&(\nabla v + b v)\cdot n = 0 \ \ \text{on $\partial\Omega$}.
\end{aligned}
\right.
\end{equation}
In addition, $v \in W^{2,p}(\Omega)$ and we have the estimate
\begin{equation}
\| v-\sigma \|_{W^{2,p}(\Omega)}\leq C\|f\|_{L^p(\Omega)},
\label{estimate_v_g_A_id}
\end{equation}
where $C$ is a constant depending on $p$, $\Omega$ and $\displaystyle \|b\|_{\text{Lip}(\overline{\Omega})}$;

\item for all $f\in L^p(\Omega)$ such that $\displaystyle \int_\Omega f \, \sigma =0$, there exists a unique $u\in H^1(\Omega)$ solution to
$$
\left\{
\begin{aligned}
&-\Delta u + b\cdot \nabla u =f \ \ \text{in $\Omega$}, \qquad \fint_\Omega u =1,
\\
& \nabla u \cdot n = 0 \ \ \text{on $\partial\Omega$}.
\end{aligned}
\right.
$$
It satisfies $u \in W^{2,p}(\Omega)$ and we have the estimate
\begin{equation}
\| u-1 \|_{W^{2,p}(\Omega)}\leq C\|f\|_{L^p(\Omega)},
\label{estimate_u_f_A_id}
\end{equation}
where $C$ is a constant depending on $p$, $\Omega$ and $\|b\|_{L^\infty(\Omega)}$.
\end{enumerate}
\end{proposition}

\begin{remark}
When $d \geq 3$, the range $1 < p < 2d/(d+2)$ is not covered by the above proposition, and we will not need this case in the sequel. We remark that the existence of a solution to~\eqref{eq:pourquoi_moi4} for such $p$ could be shown by regularization. Let $f\in L^p(\Omega)$ of mean zero. Consider a sequence $f_n \in L^q(\Omega)$ with $q \geq 2d/(d+2)$, of mean zero and such that $\dis \lim_{n \to \infty} \| f_n - f \|_{L^p(\Omega)} = 0$. Problem~\eqref{eq:pourquoi_moi4} is then well-posed for such $f_n$. We are then left with passing to the limit $n \to \infty$ in~\eqref{eq:pourquoi_moi4}, which can be done using Lemma~\ref{theorem_girault}. The other assertions of the above proposition likewise hold.
\end{remark}

\begin{proof}
Introducing $w=v-\sigma$, the point for proving assertion (i) is to consider
\begin{equation}
\left\{
\begin{aligned}
&-\textnormal{div }(\nabla w + bw)=f \ \ \text{in $\Omega$}, \qquad \fint_\Omega w=0, 
\\
&(\nabla w + b w)\cdot n = 0 \ \ \text{on $\partial\Omega$}.
\end{aligned}
\right.
\label{pb_adjoint_neumann_homog_g}
\end{equation}
The right hand side $f$ belongs to $L^p(\Omega)$ for exponents $p$ that have been chosen so that, by the Sobolev embeddings, $f\in (H^1(\Omega))'$. Lemma~\ref{theorem_droniou} therefore shows the existence and uniqueness of $w\in H^1(\Omega)$, and the fact that
\begin{equation}
\label{ineq_27_pre}
\|w\|_{H^1(\Omega)} \leq C \| f \|_{(H^1(\Omega))'} \leq C \| f \|_{L^p(\Omega)}. 
\end{equation}
We next rewrite~\eqref{pb_adjoint_neumann_homog_g} as 
$$
\left\{
\begin{aligned}
&-\Delta w =\text{div }( bw) + f \ \ \text{in $\Omega$}, \qquad \fint_\Omega w =0, 
\\ 
&\nabla w\cdot n = -(b\cdot n)w \ \ \text{on $\partial\Omega$}.
\end{aligned}
\right.
$$
To apply Lemma~\ref{theorem_girault}, we distinguish two cases, whether $1 < p \leq 2$ or $p>2$. 

\medskip

Suppose first that $1 < p \leq 2$. H\"older inequalities and~\eqref{ineq_27_pre} show that 
\begin{multline*}
\|\text{div }(bw) + f\|_{L^p(\Omega)}
\leq 
\|b\|_{W^{1,\infty}(\Omega)} \|w\|_{W^{1,p}(\Omega)} + \|f\|_{L^p(\Omega)}
\\
\leq
C \|w\|_{H^1(\Omega)} + \|f\|_{L^p(\Omega)}
\leq
C \|f\|_{L^p(\Omega)}.
\end{multline*}
In addition, using again~\eqref{ineq_27_pre} and the fact that $b$ is Lipschitz regular on the closed domain $\overline{\Omega}$, we get
$$
\|(b\cdot n)w\|_{W^{1-1/p,p}(\partial\Omega)}
\leq 
C \, \|w\|_{W^{1,p}(\Omega)}
\leq
C \, \|w\|_{H^1(\Omega)}
\leq
C \|f\|_{L^p(\Omega)}. 
$$
Lemma~\ref{theorem_girault} therefore implies that $w \in W^{2,p}(\Omega)$ with
$$
\|w \|_{W^{2,p}(\Omega)} 
\leq 
C\left(\|\text{div }(bw) + f\|_{L^p(\Omega)} + \|(b\cdot n)w\|_{W^{1-1/p,p}(\partial\Omega)}\right)
\leq 
C \|f\|_{L^p(\Omega)}.
$$
This readily yields~\eqref{estimate_v_g_A_id}, in the case when $1 < p \leq 2$.

\medskip

We now turn to the case $p>2$. H\"older inequalities show that, for any $2 \leq q \leq p$, we have 
\begin{equation}
\|\text{div }(bw) + f\|_{L^q(\Omega)}
\leq 
\|b\|_{W^{1,\infty}(\Omega)} \|w\|_{W^{1,q}(\Omega)} + \|f\|_{L^q(\Omega)}.
\label{ineq_27}
\end{equation}
In addition, because $b$ is Lipschitz regular on the closed domain $\overline{\Omega}$, we get
\begin{equation}
\|(b\cdot n)w\|_{W^{1-1/q,q}(\partial\Omega)}
\leq 
C \, \|w\|_{W^{1,q}(\Omega)}
\label{ineq_28}
\end{equation}
for any $2 \leq q \leq p$.

Setting $q=2$ in~\eqref{ineq_27} and~\eqref{ineq_28} and using that $w \in H^1(\Omega)$, we are in position to use Lemma~\ref{theorem_girault}, which implies that $w\in H^2(\Omega)$ with
\begin{align}
\|w \|_{H^2(\Omega)} & \leq C\left(\|\text{div }(bw) + f\|_{L^2(\Omega)} + \|(b\cdot n)w\|_{W^{1-1/2,2}(\partial\Omega)}\right)
\nonumber
\\
&\leq C ( \| w \|_{H^1(\Omega)} + \|f\|_{L^2(\Omega)} )
\nonumber
\\
&\leq C \|f\|_{L^p(\Omega)}.
\label{eq:estimee11}
\end{align}
Using the Sobolev embeddings, we deduce that $w \in W^{1,q}(\Omega)$ for any $2 \leq q < \infty$ if $d=2$, and $w \in W^{1,q^\star}(\Omega)$ for $q^\star = 2d/(d-2)$ otherwise.

\medskip

If $d=2$, we deduce from~\eqref{eq:estimee11} that
\begin{equation}
\label{eq:estimee11_bis}
\| w \|_{W^{1,p}(\Omega)} \leq C_p \|w \|_{H^2(\Omega)} \leq C \|f\|_{L^p(\Omega)}.
\end{equation}
We set $q=p$ in~\eqref{ineq_27} and~\eqref{ineq_28} and use Lemma~\ref{theorem_girault}, which implies that $w\in W^{2,p}(\Omega)$ with
\begin{align*}
\|w \|_{W^{2,p}(\Omega)} & \leq C\left(\|\text{div }(bw) + f\|_{L^p(\Omega)} + \|(b\cdot n)w\|_{W^{1-1/p,p}(\partial\Omega)}\right)
\\
&\leq C ( \| w \|_{W^{1,p}(\Omega)} + \|f\|_{L^p(\Omega)} )
\\
&\leq C \|f\|_{L^p(\Omega)}
\end{align*}
where we have used~\eqref{eq:estimee11_bis}. We have thus proved~\eqref{estimate_v_g_A_id}, in the case when $p>2$ and $d=2$.

\medskip

If $d>2$, we deduce from~\eqref{eq:estimee11} that
\begin{equation}
\label{eq:estimee11_qua}
\| w \|_{W^{1,q^\star}(\Omega)} \leq C \|w \|_{H^2(\Omega)} \leq C \|f\|_{L^p(\Omega)}.
\end{equation}
If $q^\star \geq p$, we proceed as above and readily obtain~\eqref{estimate_v_g_A_id}. If $q^\star < p$, we set $q=q^\star$ in~\eqref{ineq_27} and~\eqref{ineq_28} and use Lemma~\ref{theorem_girault}, which implies that $w\in W^{2,q^\star}(\Omega)$ with
\begin{align*}
\|w \|_{W^{2,q^\star}(\Omega)} & \leq C\left(\|\text{div }(bw) + f\|_{L^{q^\star}(\Omega)} + \|(b\cdot n)w\|_{W^{1-1/q^\star,q^\star}(\partial\Omega)}\right)
\\
&\leq C ( \| w \|_{W^{1,q^\star}(\Omega)} + \|f\|_{L^{q^\star}(\Omega)} )
\\
&\leq C \|f\|_{L^p(\Omega)}
\end{align*}
where we have used~\eqref{eq:estimee11_qua}. Using again the Sobolev embeddings, we deduce that $w \in W^{1,q^{\star\star}}(\Omega)$ for $1/q^{\star\star} = 1/q^\star - 1/d = 1/2-2/d$ if $d>4$, and for any $q^{\star\star} \geq 2$ otherwise. Iterating the argument a sufficient number of times, we prove~\eqref{estimate_v_g_A_id} in the case $p>2$ and $d>2$. This concludes the proof of assertion (i).

\medskip

The proof of assertion (ii) proceeds similarly.
\end{proof}

\subsubsection{Second choice of invariant measure}
\label{sssec:invariant2}

In the case where $\text{div } b = 0$, Problem~\eqref{pb_adv_diff_non_coercive_without_bc} is coercive, and the introduction of an equivalent problem using the invariant measure seems unnecessary (our numerical results will however show that using~$\sigma_1$ solution to~\eqref{invariant_measure_droniou}--\eqref{invariant_measure_droniou-int-1} indeed shows itself useful). Put differently, one intuitive choice of invariant measure is then~$\sigma=1$. Since~$\sigma =1$ is not necessary solution to~\eqref{invariant_measure_droniou} in that case, we consider another choice of invariant measure. 

The invariant measure $\sigma$ that we now aim to use (and which we will denote by~$\sigma_2$ in the sequel of this article) is a solution to
\begin{equation}
\left\{
\begin{aligned}
&-\text{div }(\nabla \sigma + b \sigma) = 0 \quad \text{in $\Omega$},
\\
&(\nabla \sigma + b \sigma) \cdot n = b\cdot n - \fint_{\partial\Omega} b\cdot n \quad\text{on $\partial\Omega$},
\\
&\inf_\Omega \sigma >0.
\end{aligned}
\right.
\label{sigma_2positive}
\end{equation}
Two remarks are in order. First, we note that $\sigma$ needs not satisfy $\dis \fint_\Omega \sigma =1$ since, in essence, this normalization constraint does not affect~\eqref{eq:adv-diff-modifiee} nor \emph{a fortiori} the original problem. Second, it is evident that $\sigma$ solution to~\eqref{sigma_2positive} is constant if $\textnormal{div }\,b=0$, a property that has precisely motivated the consideration of this alternate invariant measure. 

Because of the constraint of positivity, we are unable to directly prove the existence of~$\sigma$ solution to~\eqref{sigma_2positive}. We therefore circumvent this theoretical difficulty by temporarily considering the same problem, but without the sign constraint and with the specific normalization~$\dis \fint_\Omega \sigma =1$ (see~\eqref{sig_helmholtz_hodge} below). In a second stage, we will modify the function (adding some term involving $\sigma_1$, see Corollary~\ref{prop_sigma_2} below) in order to obtain positivity (possibly at the price of losing the normalization). We already notice that, when discretizing the problems and solving them numerically, we will proceed similarly. Since the practical implementation of~\eqref{sigma_2positive}, involving a sign constraint, would be delicate, we will first approximate numerically the solution to~\eqref{sig_helmholtz_hodge} below and next combine it with the numerical approximation of the solution to~\eqref{invariant_measure_droniou}--\eqref{invariant_measure_droniou-int-1}--\eqref{invariant_measure_droniou-pos} to obtain an approximation to the solution to~\eqref{sigma_2positive}. This will be made precise in Section~\ref{sec:discretization}.

\medskip

Let us consider the problem
\begin{equation}
\left\{
\begin{aligned}
&-\text{div }(\nabla \sigma_2^0 + b \sigma_2^0) = 0 \ \ \text{in $\Omega$}, \qquad \fint_\Omega \sigma_2^0 =1, 
\\
&(\nabla \sigma_2^0 + b \sigma_2^0) \cdot n = b\cdot n -\fint_{\partial\Omega} b\cdot n \ \ \text{on $\partial\Omega$}.
\end{aligned}
\right.
\label{sig_helmholtz_hodge}
\end{equation}
We have the following proposition, the proof of which is similar to that of Proposition~\ref{prop:sigma1} and which we therefore skip (note that the well-posedness of~\eqref{sig_helmholtz_hodge} in $H^1(\Omega)$ is a direct consequence of Lemma~\ref{theorem_droniou}(ii)):

\begin{proposition}
\label{prop_sigma_two_A_id}

Under assumptions~\eqref{hyp_H0_debut}--\eqref{hyp_H_11}, there exists a unique $\sigma_2^0\in H^1(\Omega)$ solution to~\eqref{sig_helmholtz_hodge}. For any $1 < p < +\infty$, this solution satisfies $\sigma_2^0 \in W^{2,p}(\Omega) \cap \mathcal{C}^1(\overline{\Omega})$ and we have the following estimate:
$$
\|\sigma_2^0\|_{W^{2,p}(\Omega)}\leq C\left(1+\,\|\sigma_2^0\|_{W^{1,p}(\Omega)}+ \left\|b\cdot n -\fint_{\partial\Omega} b\cdot n\right\|_{W^{1-1/p,p}(\partial\Omega)}\right),
$$
where $C$ is a constant depending on $p$, $\Omega$ and $\|b\|_{\text{Lip}(\overline{\Omega})}$. In addition, the solution to~\eqref{sig_helmholtz_hodge} satisfies conditions~\textnormal{(C1)} and~\textnormal{(C3)}.
\end{proposition}

Likewise, the following proposition holds, with a proof that mimics that of Proposition~\ref{prop_u_f_A_id}:

\begin{proposition}
\label{prop_u_f_sigma_two_A_id}

Let $1< p<+\infty$ if $d=2$ and $2d/(d+2)\leq p<+\infty$ otherwise. Assume~\eqref{hyp_H0_debut}--\eqref{hyp_H_11}. For all~$f\in L^p(\Omega)$ such that $\dis \int_\Omega f = 0$, there exists a unique $v\in H^1(\Omega)$ solution to
$$
\left\{
\begin{aligned}
&-\textnormal{div }(\nabla v + bv)=f \ \ \text{in $\Omega$}, \qquad \fint_\Omega v =1, 
\\
&(\nabla v + b v)\cdot n = b\cdot n - \fint_{\partial\Omega}b\cdot n \ \ \text{on $\partial\Omega$}.
\end{aligned}
\right.
$$
This solution belongs to $W^{2,p}(\Omega)$ and satisfies
\begin{equation}
\label{eq:estimee-u-sigma}
\| v-\sigma_2^0 \|_{W^{2,p}(\Omega)}\leq C\|f\|_{L^p(\Omega)},
\end{equation}
where $\sigma_2^0$ is the solution to~\eqref{sig_helmholtz_hodge} and $C$ is a constant depending on $p$, $\Omega$ and $\|b\|_{\text{Lip}(\overline{\Omega})}$.
\end{proposition}

Of course, all what matters in the above estimation~\eqref{eq:estimee-u-sigma} is that $\displaystyle \fint_\Omega v = \fint_\Omega \sigma_2^0$ and not the actual value of that integral.

\medskip

We now eventually obtain a solution to~\eqref{sigma_2positive}, modifying $\sigma_2^0$ in a suitable manner. This is the purpose of our next result. 

\begin{corollary}
\label{prop_sigma_2}

Let $\sigma_2^0$ (resp. $\sigma_1$) be the solution to~\eqref{sig_helmholtz_hodge} (resp. to~\eqref{invariant_measure_droniou}--\eqref{invariant_measure_droniou-int-1}). Under assumptions~\eqref{hyp_H0_debut}--\eqref{hyp_H_11}, the set of solutions to~\eqref{sigma_2positive} reads as
$$
\left\{\sigma_2^0+\kappa \, \sigma_1, \quad \kappa\in\R \text{ such that } \inf_\Omega \ (\sigma_2^0+\kappa \, \sigma_1)>0\right\}.
$$
\end{corollary}

The proof of Corollary~\ref{prop_sigma_2} is immediate. Let $\sigma$ be a solution to~\eqref{sigma_2positive}. Then $\sigma -\sigma_2^0$ is a solution to~\eqref{invariant_measure_droniou}, and we are then in position to use Lemma~\ref{theorem_droniou}, noticing that the necessary value of $\kappa$ is $\dis \kappa=\fint_\Omega\sigma -1$. The converse inclusion is straightforward. 

\medskip

We finally define $\sigma_2$ solution to~\eqref{sigma_2positive} as 
\begin{equation}
\label{eq:def_sigma2_enfin}
\sigma_2 = \sigma_2^0+\kappa^\star \, \sigma_1,
\end{equation}
where 
$$
\kappa^\star = 1+ \inf \left\{ \kappa\in\R \text{ such that } \inf_\Omega \, (\sigma_2^0+\kappa \, \sigma_1)>0 \right\}.
$$ 
Of course this is an arbitrary choice. In practice, some suitable $\kappa$ (and thus $\sigma_2$) will be used. The numerical analysis will account for this. 
 
\section{Discretization and numerical analysis}
\label{sec:discretization}

Practically, we implement a finite element Galerkin approximation $u_H$ of the solution $u$ to the coercive equivalent modified problem~\eqref{eq_modified_problem}. Since, in most of the cases, the invariant measure $\sigma$ is not known analytically, we first seek a Galerkin approximation of~$\sigma$ (we will describe later in this article how this approximation is obtained, for each of the two cases $\sigma \equiv \sigma_1$ and $\sigma \equiv \sigma_2$). We denote by $H$ and $h$ the mesh sizes for these two approximations, respectively. We have in mind that $H \gg h$, in order to be as efficient as possible. This is made possible by the uniform well-posedness of the discrete problem in $u_H$ (see Proposition~\ref{prop_petrov_galerkin_sigma_h} below). In practice, we will observe that we may indeed choose~$H$ one order of magnitude larger, say, than~$h$.

\medskip

We begin with making precise the approximation $u_H$, for a given approximation $\sigma_h$ of $\sigma$. The discrete variational formulation reads as:
\begin{equation}
\text{Find $u_H \in U_H$ such that, for all $v_H\in U_H$, \quad $a_{\text{ss}}(\sigma_h;u_H,v_H) = F(\sigma_h v_H)$},
\label{varf_sigmah_symmetric}
\end{equation}
where $F$ is defined by~\eqref{def_F} and $a_{\text{ss}}$ is defined by~\eqref{eq:faible-uH} below.

We assume throughout this section that the discretization space $U_H$ in~\eqref{varf_sigmah_symmetric} is a subspace of $H^1_0(\Omega)$. In our actual implementation, the above formulation will be possibly slightly modified to account for a stabilization performed when computing $\sigma_h$. This will be made precise in the next section, in formulae~\eqref{eq:Bbarre}--\eqref{eq:formulation-modifiee-barre}. As will be mentioned there, this potential modification does not modify the numerical analysis we perform in the present section. 

In the left hand side of~\eqref{varf_sigmah_symmetric}, we have denoted
\begin{align}
a_{\text{ss}}(\sigma_h;u_H,v_H) &= \int_\Omega \sigma_h\nabla u_H\cdot\nabla v_H + B_h\cdot\frac{(\nabla u_H) v_H -(\nabla v_H)u_H }{2},
\label{eq:faible-uH}
\\
B_h &=\nabla \sigma_h +\sigma_h\,b.
\label{eq:Bh}
\end{align}
The classical skew-symmetric formulation of the advection part is adopted in order to ensure that $\displaystyle a_{\text{ss}}(\sigma_h;u_H,u_H)= \int_\Omega \sigma_h\,|\nabla u_H|^2$ and thus that the problem is coercive at the discrete level whenever $\sigma_h$ is positive and bounded away from zero. A simple application of standard arguments therefore shows the following well-posedness of the discretization. The point is, this well-posedness is uniform in the mesh size $H$, a property that is of major practical interest.

\begin{proposition}
\label{prop_petrov_galerkin_sigma_h}
Assume~\eqref{hyp_H0_debut}. Consider $\sigma_h$ an approximation of $\sigma\in H^1(\Omega)$ such that $\displaystyle \inf_\Omega \sigma_h >0$. Then $a_{\text{ss}}(\sigma_h;\cdot,\cdot)$ is coercive in $H^1_0(\Omega)$ and~\eqref{varf_sigmah_symmetric} is well-posed, uniformly in $H$.
\end{proposition}

We now proceed with the numerical analysis of~\eqref{varf_sigmah_symmetric}. 

\subsection{Numerical analysis in the case when the invariant measure is analytically known}

To begin with, we temporarily assume that we know $\sigma$ analytically, meaning we replace $B_h$ given by~\eqref{eq:Bh} by $B=\nabla\sigma+\sigma\,b$ in the second term of~\eqref{eq:faible-uH} (and we likewise replace $\sigma_h$ by $\sigma$ in the first term of~\eqref{eq:faible-uH}). Otherwise stated, we replace $\sigma_h$ by~$\sigma$ in~\eqref{varf_sigmah_symmetric}.

\begin{proposition}
\label{prop_convergence_sigma_exact}
Assume~\eqref{hyp_H0_debut}--\eqref{hyp_H_11} and that $\sigma_h \equiv \sigma$ in~\eqref{varf_sigmah_symmetric}. Let $u$ be the solution to~\eqref{eq_modified_problem} and $u_H$ be the solution to~\eqref{varf_sigmah_symmetric}. Then, for any $p>d$, we have the estimate
\begin{equation}
\|u-u_H\|_{H^1(\Omega)} \leq C \, \left[ \frac{\|\sigma\|_{L^\infty(\Omega)} + \|\nabla\sigma + \sigma\,b\|_{L^p(\Omega)}}{\inf_\Omega \sigma} \right] \inf_{v_H\in U_H} \|u-v_H\|_{H^1(\Omega)},
\label{eq:u-uH}
\end{equation}
with a constant $C$ that only depends on $\Omega$ and $p$.
\end{proposition}

Note that, in view of Lemma~\ref{theorem_perthame}, the assumptions~\eqref{hyp_H0_debut}--\eqref{hyp_H_11} imply that $\sigma_1$ satisfies the conditions~\textnormal{(C1)}, \textnormal{(C2)} and~\textnormal{(C3)}. Likewise, in view of~\eqref{eq:def_sigma2_enfin}, Lemma~\ref{theorem_perthame} and Proposition~\ref{prop_sigma_two_A_id}, $\sigma_2$ satisfies the conditions~\textnormal{(C1)}, \textnormal{(C2)} and~\textnormal{(C3)}. In particular, for both choices $\sigma \equiv \sigma_1$ and $\sigma \equiv \sigma_2$, we have $\dis \inf_\Omega \sigma > 0$. 

\begin{proof}
As $\text{div} \, B = 0$, we note that the problem
$$
\text{Find $u\in H^1_0(\Omega)$ such that, for all $v\in H^1_0(\Omega)$, \quad $a_{\text{ss}}(\sigma;u,v) = F(\sigma \, v)$}
$$
is a variational formulation of the modified problem~\eqref{eq_modified_problem}. Since $\sigma_h \equiv \sigma$, Problem~\eqref{varf_sigmah_symmetric} is the Galerkin approximation of~\eqref{eq_modified_problem} in $U_H$. We note that the bilinear form $a_{\text{ss}}(\sigma;\cdot,\cdot)$ is coercive, while, for all $u$ and $v$ in $H^1_0(\Omega)$, we have
\begin{align}
a_{\text{ss}}(\sigma;u,v) &\leq \|\sigma\|_{L^\infty(\Omega)}\|\nabla u\|_{L^2(\Omega)}\|\nabla v\|_{L^2(\Omega)} 
+ \|\nabla\sigma+\sigma\,b\|_{L^p(\Omega)}\|\nabla u\|_{L^2(\Omega)}\,\|v\|_{L^q(\Omega)}
\nonumber
\\
&\leq \left(\|\sigma\|_{L^\infty(\Omega)} + C_{p,\Omega} \, \|\nabla\sigma+\sigma\,b\|_{L^p(\Omega)}\right)\|u\|_{H^1(\Omega)}\|v\|_{H^1(\Omega)},
\label{eq:const_p_omega}
\end{align}
where $1/p+1/q=1/2$ and, for the Sobolev embedding to hold, $q<2d/(d-2)$ which amounts to $p>d$. Classical results of numerical analysis of coercive problems then allow to conclude, using the C\'ea lemma.
\end{proof}

The following corollary makes precise how the $L^p(\Omega)$ norm in the right hand side of~\eqref{eq:u-uH} may be bounded from above by the $H^1(\Omega)$ norm of $\sigma$, because of the particular properties of $\sigma$. When the discretized approximation~$\sigma_h$ is reinstated in place of $\sigma$, this part of the argument will become substantially more difficult. We will return to this later.
 
\begin{corollary}
In addition to the assumptions of Proposition~\ref{prop_convergence_sigma_exact}, we assume that the ambient dimension is $d=2$ or $3$. Then, we have
$$
\|u-u_H\|_{H^1(\Omega)}
\leq C \left( \frac{\|\sigma\|_{L^\infty(\Omega)} + \|\sigma\|_{H^1(\Omega)} + C_\sigma}{\inf_\Omega \sigma} \right) \inf_{v_H\in U_H} \|u-v_H\|_{H^1(\Omega)},
$$
where $C$ is a constant independent of $H$ and 
$$
C_\sigma = 
\left\{
\begin{aligned} 
&0 \quad\text{if $\sigma \equiv \sigma_1$},
\\
&\left\|b\cdot n - \fint_{\partial\Omega}b\cdot n\right\|_{H^{1/2}(\partial\Omega)}\quad\text{if $\sigma \equiv \sigma_2$}.
\end{aligned}
\right.
$$
\end{corollary}

\begin{proof}
Using~\cite[Corollary 3.7]{girault1986finite}, we know that $B=\nabla\sigma+\sigma\,b$ satisfies
$$
\|B\|_{H^1(\Omega)} \leq C \left( \|B\|_{L^2(\Omega)} + \|\text{div} \, B \|_{L^2(\Omega)} + \|\text{curl}\,B\|_{L^2(\Omega)} + \|B\cdot n\|_{H^{1/2}(\partial\Omega)}\right).
$$
We notice that, on the one hand, $\text{div}\,B=0$, by definition of $\sigma$, while, on the other hand, $\text{curl}\,B=\text{curl}\,(\sigma\,b)$, thus, given the Lipschitz regularity of~$b$, $\|\text{curl}\,B\|_{L^2(\Omega)} \leq C\,\|\sigma\|_{H^1(\Omega)}$. We therefore obtain 
$$
\|B\|_{H^1(\Omega)} \leq C\left(\|\sigma\|_{H^1(\Omega)} +\|B\cdot n\|_{H^{1/2}(\partial\Omega)}\right). 
$$
Finally, because $d\leq 3$, we may find $p$ such that $d<p\leq 2d/(d-2)$, thus $\|B\|_{L^p(\Omega)}\leq C \|B\|_{H^1(\Omega)}$,
which proves the result.
\end{proof}

\subsection{Numerical analysis in the case when the invariant measure is numerically approximated}
\label{sec:ananu}

We now return to the case when the invariant measure is only approximated numerically. For simplicity, we restrict our attention to the case when the approximation space for $\sigma$ is a $\mathbb{P}^1$ finite element space associated to a polyhedral mesh of $\Omega$. In this case, $\Omega$ is thus a polygon, and it cannot be of class $\mathcal{C}^1$ or $\mathcal{C}^2$, as assumed previously in~\eqref{hyp_H0_fin} or~\eqref{hyp_H_11}. 

In the sequel of this Section~\ref{sec:ananu}, we assume, in addition to~\eqref{hyp_H0_debut}, that
\begin{equation}
\left\{
\begin{aligned}
&\text{the domain $\Omega$ is connected, convex and polyhedral};
\\
&\text{$b$ is Lipschitz-continuous on $\overline{\Omega}$};
\\
&\text{the ambient dimension satisfies $2 \leq d \leq 3$}.
\end{aligned}
\right.
\label{hyp_H_11_num}
\end{equation}
We also assume that
\begin{equation}
\text{The conclusions of Propositions~\ref{prop:sigma1}, \ref{prop_u_f_A_id}, \ref{prop_sigma_two_A_id} and~\ref{prop_u_f_sigma_two_A_id} hold.}
\label{hyp_H_22_num}
\end{equation}

\medskip

A few remarks are in order. 

First, we point out that~\eqref{hyp_H_22_num} is not a consequence of Section~\ref{sec:mathematical-setting}, as we now do not assume~\eqref{hyp_H_11}.

Second, \eqref{hyp_H_22_num} obviously holds in the case $b=0$. In that case, there exists a unique solution to~\eqref{invariant_measure_droniou}--\eqref{invariant_measure_droniou-int-1} (resp. to~\eqref{sig_helmholtz_hodge}) which is $\sigma_1 = 1$ (resp. $\sigma_2^0 = 1$).

Third, when $\|b\|_{\text{Lip}(\overline{\Omega})}$ is sufficiently small, then~\eqref{hyp_H_22_num} again holds. For the sake of brevity, we only sketch the proof for Proposition~\ref{prop:sigma1}. To construct the invariant measure, consider the following iterations: set $\sigma^0 = 1$, and define $\sigma^{m+1}$ as the unique solution in $H^1(\Omega)$ to the problem
$$
-\Delta \sigma^{m+1} = \textnormal{div }(b \sigma^m) \ \text{in $\Omega$}, 
\qquad
\nabla \sigma^{m+1} \cdot n = - b \sigma^m \cdot n \ \text{on $\partial\Omega$},
\qquad 
\fint_\Omega \sigma^{m+1} =1.
$$
It is easy to see that
$$
\| \nabla( \sigma^{m+1}- \sigma^m) \|_{L^2(\Omega)} \leq \| b \|_{L^\infty(\Omega)} \| \sigma^m- \sigma^{m-1} \|_{L^2(\Omega)}.
$$
Since the mean of $\sigma^m- \sigma^{m-1}$ vanishes, we can use the Poincar\'e-Wirtinger (PW) inequality, from which we deduce that
$$
\| \nabla( \sigma^{m+1}- \sigma^m) \|_{L^2(\Omega)} \leq C_{\rm PW} \| b \|_{L^\infty(\Omega)} \| \nabla (\sigma^m- \sigma^{m-1}) \|_{L^2(\Omega)}.
$$
Assume that $b$ is such that $C_{\rm PW} \| b \|_{L^\infty(\Omega)} < 1$. Then $\sigma^m$ converges to some $\sigma^\star$ in $H^1(\Omega)$, which is a solution to~\eqref{invariant_measure_droniou}--\eqref{invariant_measure_droniou-int-1}. The uniqueness of such a solution is easily obtained, again as a consequence of the Poincar\'e-Wirtinger inequality and of the fact that $C_{\rm PW} \| b \|_{L^\infty(\Omega)} < 1$. Furthermore, for any $v \in H^1(\Omega)$, we have
$$
\int_\Omega \nabla (\sigma^\star-1) \cdot \nabla v = - \int_\Omega (\sigma^\star-1) \, b \cdot \nabla v - \int_\Omega b \cdot \nabla v.
$$
Choosing $v = \sigma^\star-1$, we get
$$
\| \nabla( \sigma^\star - 1) \|_{L^2(\Omega)} \leq C_{\rm PW} \| b \|_{L^\infty(\Omega)} \| \nabla (\sigma^\star - 1 ) \|_{L^2(\Omega)} + \| b \|_{L^\infty(\Omega)},
$$
and thus $\dis \| \sigma^\star - 1 \|_{H^1(\Omega)} \leq \sqrt{1+C_{\rm PW}^2} \ \frac{\| b \|_{L^\infty(\Omega)}}{1 - C_{\rm PW} \| b \|_{L^\infty(\Omega)}}$.

We now prove that $\sigma^\star \in H^2(\Omega)$. Considering the Neumann problem
$$
-\Delta (\sigma^\star-1) = \textnormal{div }(b \sigma^\star) \quad\text{in $\Omega$}, \qquad \nabla (\sigma^\star-1) \cdot n = - b \sigma^\star \cdot n \quad\text{on $\partial\Omega$},
$$
we observe, as in the proof of Proposition~\ref{prop:sigma1} and using the regularity of $b$, that
$$
\|\textnormal{div }(b\sigma^\star)\|_{L^2(\Omega)} \leq C \, \| b \|_{W^{1,\infty}(\Omega)} \, \|\sigma^\star\|_{H^1(\Omega)}
$$
while
$$
\| (b\cdot n) \sigma^\star\|_{H^{1/2}(\partial\Omega)} \leq C \|b\|_{\text{Lip}(\overline{\Omega})} \, \|\sigma^\star\|_{H^1(\Omega)},
$$
where $C$ is independent of $b$. Thanks to the assumption~\eqref{hyp_H_11_num} on $\Omega$, we are in position to use~\cite[Theorem 3.12]{EG}, which implies that $\sigma^\star-1 \in H^2(\Omega)$ and
\begin{eqnarray*}
\| \sigma^\star-1 \|_{H^2(\Omega)} 
&\leq& 
C \left( \|\textnormal{div }(b\sigma^\star)\|_{L^2(\Omega)} + \| (b\cdot n) \sigma^\star\|_{H^{1/2}(\partial\Omega)} \right) 
\\
&\leq & C \left( \| b \|_{W^{1,\infty}(\Omega)} + \|b\|_{\text{Lip}(\overline{\Omega})} \right) \|\sigma^\star\|_{H^1(\Omega)}
\\
&\leq & C \|b\|_{\text{Lip}(\overline{\Omega})} \left( 1 + \frac{\| b \|_{L^\infty(\Omega)}}{1 - C_{\rm PW} \| b \|_{L^\infty(\Omega)}} \right).
\end{eqnarray*}
Similar estimates in $W^{2,p}(\Omega)$ can be shown using~\cite[Remark 3.13(ii)]{EG}. 

To show that Proposition~\ref{prop:sigma1} holds, we are now left with showing~\eqref{invariant_measure_droniou-pos}. Thanks to the Sobolev injections when $2 \leq d \leq 3$, we have $\| \sigma^\star-1 \|_{C^0(\Omega)} \leq C \| \sigma^\star-1 \|_{H^2(\Omega)}$. Thus, when $b$ is sufficiently small, then $\| \sigma^\star-1 \|_{C^0(\Omega)}$ is small as well and~\eqref{invariant_measure_droniou-pos} holds. We can thus conclude that Proposition~\ref{prop:sigma1} holds.

\subsubsection{Preliminary estimate}

In what follows, we proceed under the assumptions~\eqref{hyp_H0_debut}--\eqref{hyp_H_11_num}--\eqref{hyp_H_22_num}. The analogous result to that of Proposition~\ref{prop_convergence_sigma_exact} is stated in the following. 

\begin{proposition}
\label{prop_convergence_sigma_h}

Assume~\eqref{hyp_H0_debut}--\eqref{hyp_H_11_num}--\eqref{hyp_H_22_num}. Consider $\sigma_h$ an approximation of $\sigma\in H^1(\Omega)$ such that $\displaystyle \inf_\Omega \sigma_h >0$. Let $u$ be the solution to~\eqref{pb_adv_diff_non_coercive_without_bc}--\eqref{boundary_cond_dirichlet_homog} (or equivalently~\eqref{eq_modified_problem}) and $u_H$ be the solution to~\eqref{varf_sigmah_symmetric}, for some $f\in L^2(\Omega)$. For any $p>d$, we have the estimate 
\begin{align}
\|u-u_H\|_{H^1(\Omega)} &\leq \frac{C}{\inf_\Omega \sigma_h} \|f\|_{L^2(\Omega)}\|\sigma-\sigma_h\|_{L^p(\Omega)} 
\nonumber
\\
& + C\, \inf_{v_H\in U_H} \Bigg[ \left( 1 +\frac{\vvvert \sigma \vvvert_p}{\inf_\Omega \sigma_h}\right)\|u-v_H\|_{H^1(\Omega)} + \frac{\vvvert \sigma-\sigma_h \vvvert_p}{\inf_\Omega \sigma_h}\|v_H\|_{H^1(\Omega)}\Bigg],
\label{ineq_12}
\end{align}
where $C$ only depends on $p$ and $\Omega$, and where we have used the notation
$$
\vvvert \sigma \vvvert_p = \|\sigma\|_{L^\infty(\Omega)}+\|\nabla \sigma + b \sigma\|_{L^p(\Omega)}.
$$
\end{proposition}

\begin{proof}
We note that $u \in H^1_0(\Omega)$ satisfies
$$
\forall v \in H^1_0(\Omega), \quad a_{\text{ss}}(\sigma;u,v) = F(\sigma v),
$$
while $u_H \in U_H$ satisfies
$$
\forall v_H\in U_H, \quad a_{\text{ss}}(\sigma_h;u_H,v_H) = F(\sigma_h v_H).
$$
Applying the first Strang Lemma (see~\cite[Lemma 2.27]{EG}), we have
\begin{align}
\|u-u_H\|_{H^1(\Omega)} 
&\leq 
\frac{1}{C_\Omega \inf_\Omega\sigma_h} \sup_{w_H\in U_H} \frac{\left|\int_\Omega f(\sigma-\sigma_h)w_H\right|}{\|w_H\|_{H^1(\Omega)}} 
\nonumber
\\
& + \inf_{v_H\in U_H} \Bigg[ \left(1+ \frac{\|\sigma\|_{L^\infty(\Omega)} + C_{p,\Omega} \, \|B\|_{L^p(\Omega)}}{C_\Omega \inf_\Omega \sigma_h}\right) \|u-v_H\|_{H^1(\Omega)} 
\nonumber
\\
&+ \frac{1}{C_\Omega\inf_\Omega \sigma_h} \sup_{w_H \in U_H} \frac{\left|a_{\text{ss}}(\sigma-\sigma_h;v_H,w_H)\right|}{\|w_H\|_{H^1(\Omega)}}\Bigg],
\label{strang_lemma}
\end{align}
where $B=\nabla \sigma + b \sigma$, $C_\Omega$ is the Poincar\'e constant of $\Omega$ (so that $C_\Omega\inf_\Omega\sigma_h$ is a coercivity constant of $a_{\text{ss}}(\sigma_h;\cdot,\cdot)$ on $U_H$) and $C_{p,\Omega}$ is the constant introduced in~\eqref{eq:const_p_omega}. When we take~$\sigma_h \equiv \sigma$, this estimation of course agrees with the estimation we have already established, independently, for Proposition~\ref{prop_convergence_sigma_exact}. 

For the first term of the right-hand side of~\eqref{strang_lemma}, we notice that, for all $w_H\in U_H$ and $p>d$ (thus $1/q=1/2-1/p<1/2-1/d$), we have 
\begin{align}
\left|\int_\Omega f(\sigma-\sigma_h)w_H\right|
&\leq \|f\|_{L^2(\Omega)}\|\sigma-\sigma_h\|_{L^p(\Omega)}\|w_H\|_{L^q(\Omega)}
\nonumber
\\
&\leq C_{p,\Omega} \|f\|_{L^2(\Omega)}\|\sigma-\sigma_h\|_{L^p(\Omega)}\|w_H\|_{H^1(\Omega)}.
\label{ineq_10}
\end{align}
The rightmost term of~\eqref{strang_lemma} is estimated similarly:
\begin{multline}
\left|a_{\text{ss}}(\sigma-\sigma_h;v_H,w_H)\right|
\\
\leq \left(\|\sigma-\sigma_h\|_{L^\infty(\Omega)} + C_{p,\Omega} \, \|B-B_h\|_{L^p(\Omega)}\right)\,\|v_H\|_{H^1(\Omega)}\,\|w_H\|_{H^1(\Omega)},
\label{ineq_11}
\end{multline}
where $B_h=\nabla \sigma_h + b \sigma_h$. Combining~\eqref{strang_lemma},~\eqref{ineq_10} and~\eqref{ineq_11} gives the desired estimate.
\end{proof}


\subsubsection{Estimation of $\sigma - \sigma_h$}

The estimation of $\nabla (\sigma-\sigma_h)$ in $L^p(\Omega)$, for some $p>d$, is the crucial ingredient we now need to proceed with the estimation of the right-hand side of~\eqref{ineq_12}. This estimation is the main purpose of Proposition~\ref{prop_brenner_scott_estimate} below. We emphasize that the result is not immediate and its proof instructive. Before stating this result, we first detail how $\sigma_h$ is defined and provide in Proposition~\ref{prop_brenner_scott_estimate_debut} below a classical error estimate on $\sigma-\sigma_h$ in $H^1(\Omega)$. 

\bigskip

We introduce the bilinear form
\begin{equation}
\label{eq:def_astar}
\abil(u,v) = \int_\Omega (\nabla u + bu) \cdot \nabla v,
\end{equation}
which is formally the adjoint of the bilinear form $a$ defined by~\eqref{eq:def_bil_a}, in the sense that $\abil(u,v) = a(v,u)$. We note that the invariant measure $\sigma_1$ solution to~\eqref{invariant_measure_droniou}--\eqref{invariant_measure_droniou-int-1}--\eqref{invariant_measure_droniou-pos} satisfies
$$
\forall v \in H^1(\Omega), \qquad \abil(\sigma_1,v) = 0
$$
while the invariant measure $\sigma_2^0$ solution to~\eqref{sig_helmholtz_hodge} satisfies
$$
\forall v \in H^1(\Omega), \qquad \abil(\sigma_2^0,v) = \int_{\partial \Omega} g \, v
$$
with $\dis g = b\cdot n -\fint_{\partial\Omega} b\cdot n$ on $\partial \Omega$.

\begin{proposition}
\label{prop_brenner_scott_estimate_debut}
We assume that~\eqref{hyp_H0_debut}--\eqref{hyp_H_11_num}--\eqref{hyp_H_22_num} hold. Let $\Sigma_h$ be the $\mathbb{P}^1$ approximation space associated to a regular quasi-uniform polyhedral mesh of $\Omega$ and
$$
V_h=\left\{u \in \Sigma_h, \quad \fint_\Omega u = 1 \right\}.
$$
Let $\sigma$ denote either the solution to~\eqref{invariant_measure_droniou}--\eqref{invariant_measure_droniou-int-1}--\eqref{invariant_measure_droniou-pos} (in which case we set $g=0$) or the solution to~\eqref{sig_helmholtz_hodge} (in which case we set $\dis g = b\cdot n -\fint_{\partial\Omega} b\cdot n$ on $\partial \Omega$). 

For $h$ sufficiently small, there exists a unique $\sigma_h \in V_h$ (which is the Galerkin approximation of~$\sigma$) solution to
\begin{equation}
\label{eq:sigma_h_bien_pose}
\forall v_h \in \Sigma_h, \qquad \abil(\sigma_h,v_h) = \int_{\partial \Omega} g \, v_h.
\end{equation}
Furthermore, we have, for $h$ sufficiently small,
\begin{equation}
\label{eq:error_sigma_h}
\| \sigma - \sigma_h \|_{H^1(\Omega)} \leq C h \|\sigma\|_{H^2(\Omega)}
\end{equation}
where $C$ is independent of $h$. 
\end{proposition}

The proof of Proposition~\ref{prop_brenner_scott_estimate_debut} is postponed until Appendix~\ref{sec:appendix_H1}. We now turn to the estimation of $\nabla (\sigma-\sigma_h)$ in $L^p(\Omega)$. 

\begin{proposition}
\label{prop_brenner_scott_estimate}
Under the assumptions of Proposition~\ref{prop_brenner_scott_estimate_debut}, for all~$2<p< +\infty$, the estimate
\begin{equation}
\label{eq:estim_BS}
\| \sigma - \sigma_h \|_{W^{1,p}(\Omega)} \leq C h \|\sigma\|_{W^{2,p}(\Omega)}
\end{equation}
holds for $h$ sufficiently small, where $C$ is independent of $h$.
\end{proposition}

Before giving the actual proof of Proposition~\ref{prop_brenner_scott_estimate}, we first discuss this result and describe various strategies to prove it.

We emphasize that, to the best of our knowledge, this result is not present in the literature. There exist many contributions establishing $W^{1,p}$ estimates between the solution of a linear PDE and its finite element approximation, for problems posed with homogeneous Dirichlet boundary conditions, or problems posed with Neumann boundary conditions and including a zero-order term. In~\cite{scott1976optimal}, the author considers the Neumann problem
\begin{equation}
\label{eq:pourquoi_moi3}
-\Delta v + v = f \ \ \text{in $\Omega$}, \qquad \nabla v \cdot n=0 \ \ \text{on $\partial\Omega$},
\end{equation}
while the Dirichlet problem
$$
-\Delta v = f \ \ \text{in $\Omega$}, \qquad v=0 \ \ \text{on $\partial\Omega$},
$$
is studied in~\cite{natterer,nitsche1977convergence,rannacher-scott}. A more general PDE (including an advection term and a zero-order term, but again with homogeneous Dirichlet boundary conditions) is considered in~\cite[Chap. 8]{BS}. All these problems are well-posed (under appropriate assumptions) for {\em any} sufficiently regular right-hand side. We note that the proofs contained in the contributions we have cited consider the problem of interest (e.g.~\eqref{eq:pourquoi_moi3} in~\cite{scott1976optimal}) for several right-hand sides, and not only the right-hand side $f$ originally considered. In contrast, Problem~\eqref{eq:pourquoi_moi} is well-posed only for right-hand sides satisfying some compatibility conditions (see Lemma~\ref{theorem_droniou}). This is one of the reasons why the proof of Proposition~\ref{prop_brenner_scott_estimate} is not immediate. 

Another contribution we wish to cite is~\cite[Theorem A.2 p. 101]{girault1986finite}. Taking some sufficiently regular functions $f$ and $g$ such that the compatibility condition $\dis \int_\Omega f + \int_{\partial\Omega} g = 0$ holds, the authors consider the Neumann problem
\begin{equation}
\label{eq:pure_neumann}
-\Delta v = f \ \ \text{in $\Omega$}, \qquad \nabla v \cdot n = g \ \ \text{on $\partial\Omega$}
\end{equation}
and state a $W^{1,p}$ estimate between $v$ and its finite element approximation $v_h$ (chosen such that $\dis \int_\Omega v_h = \int_\Omega v$): there exists $C$ independent of $h$ such that
\begin{equation}
\label{eq:b_small4}
\| v_h-v \|_{W^{1,p}(\Omega)} \leq C h\| v \|_{W^{2,p}(\Omega)}.
\end{equation}

\medskip

There are (at least) two ways to prove Proposition~\ref{prop_brenner_scott_estimate}. A first possibility is to assume that $\| b \|_{\text{Lip}(\overline{\Omega})}$ is small enough. Under this assumption (which is restrictive since we precisely aim in this article at considering non-coercive problems~\eqref{pb_adv_diff_non_coercive_without_bc} where $b$ is not small) and using~\eqref{eq:b_small4}, the proof of~\eqref{eq:estim_BS} is short. For the sake of brevity, we only consider the invariant measure $\sigma_1$ solution to~\eqref{invariant_measure_droniou}--\eqref{invariant_measure_droniou-int-1}--\eqref{invariant_measure_droniou-pos}. We introduce the sequence $\sigma^m \in H^1(\Omega)$ defined by
\begin{equation}
\label{eq:b_small1}
-\Delta \sigma^{m+1} = \textnormal{div }(b \sigma^m) \ \text{in $\Omega$}, 
\quad
\nabla \sigma^{m+1} \cdot n = - b \sigma^m \cdot n \ \text{on $\partial\Omega$},
\quad 
\fint_\Omega \sigma^{m+1} =1,
\end{equation}
with $\sigma^0 = 1$. Since $b$ is small enough, it turns out that $\sigma^m$ converges to $\sigma_1$, solution to~\eqref{invariant_measure_droniou}--\eqref{invariant_measure_droniou-int-1}. In addition, the above problem is of the type~\eqref{eq:pure_neumann}, so we will be in position to use~\eqref{eq:b_small4}.

Consider the sequence $\sigma^m_h \in \Sigma_h$ defined by
\begin{equation}
\label{eq:b_small2}
\forall v \in \Sigma_h, \qquad \int_\Omega \nabla \sigma^{m+1}_h \cdot \nabla v = - \int_\Omega \sigma^m_h b \cdot \nabla v, \qquad \fint_\Omega \sigma^{m+1}_h =1,
\end{equation}
with $\sigma^0_h = 1$, which converges to $\sigma_{1,h}$, solution to~\eqref{eq:sigma_h_bien_pose} with $g=0$. Using the result~\eqref{eq:b_small4} given in~\cite[Theorem A.2 p. 101]{girault1986finite}, one can eventually show that
\begin{equation}
\label{eq:b_small3}
\| \sigma^{m+1}_h-\sigma^{m+1} \|_{W^{1,p}(\Omega)}
\leq
C \| b \|_{\text{Lip}(\overline{\Omega})} \left( \| \sigma^m_h-\sigma^m \|_{W^{1,p}(\Omega)} + h \| \sigma^m \|_{W^{1,p}(\Omega)} \right)
\end{equation}
for some $C$ independent of $h$ and $b$. Note that the right-hand sides of~\eqref{eq:b_small1} and ~\eqref{eq:b_small2} are different, so the intermediate problem
$$
-\Delta \overline{\sigma}^{m+1} = \textnormal{div }(b \sigma^m_h) \ \text{in $\Omega$}, 
\quad
\nabla \overline{\sigma}^{m+1} \cdot n = - b \sigma^m_h \cdot n \ \text{on $\partial\Omega$},
\quad 
\fint_\Omega \overline{\sigma}^{m+1} =1,
$$
has to be introduced to prove~\eqref{eq:b_small3}. Passing to the limit $m \to \infty$ in~\eqref{eq:b_small3}, and using again that $\| b \|_{\text{Lip}(\overline{\Omega})}$ is sufficiently small, we obtain~\eqref{eq:estim_BS}.

\bigskip

A second possibility, which is the one we follow here, is based on considering the following problem, that we write in a compact form as $L_\etaa \, \sigma^{\etaa,f} = f$:
\begin{equation}
\left\{
\begin{aligned}
&-\textnormal{div }(\nabla \sigma^{\etaa,f} + b \sigma^{\etaa,f}) + \etaa \sigma^{\etaa,f} = f \quad\text{in $\Omega$}, 
\\
&(\nabla \sigma^{\etaa,f} + b \sigma^{\etaa,f})\cdot n=0 \quad\text{on $\partial\Omega$},
\end{aligned}
\right.
\label{eq:b_small6}
\end{equation}
for any $0 < \etaa \leq 1$. Problem~\eqref{eq:b_small6} is well-posed for any sufficiently regular function $f$ (in particular, there is no compatibility condition on $f$). Let $\sigma^{\etaa,f}_h \in \Sigma_h$ be the P1 finite element approximation of $\sigma^{\etaa,f}$. It is then possible to adapt the proof of~\cite[Chap. 8]{BS} to this case, and show that there exists a constant $C_\etaa$ independent of $h$ and $f$ (but a priori depending on $\etaa$) such that
\begin{equation}
\label{eq:b_small5}
\| \sigma^{\etaa,f}_h - \sigma^{\etaa,f} \|_{W^{1,p}(\Omega)} \leq C_\etaa \, h \, \| \sigma^{\etaa,f} \|_{W^{2,p}(\Omega)}.
\end{equation}
We now sketch the proof of~\eqref{eq:estim_BS}, in the case of the invariant measure $\sigma_1$ solution to~\eqref{invariant_measure_droniou}--\eqref{invariant_measure_droniou-int-1}--\eqref{invariant_measure_droniou-pos}. The proof is based on the introduction of iterations of the type~\eqref{eq:b_small1}, with the operator $L_\etaa$ of~\eqref{eq:b_small6} instead of the Laplacian operator:
$$
L_\eta \, \sigma^{m+1}_\etaa = \etaa \sigma^m_\etaa
$$
for some $\etaa$ sufficiently small. We refer to~\eqref{eq:maison3} below for details. The proof then proceeds as in the case $b$ small above, the pivotal estimate~\eqref{eq:b_small4} being replaced by~\eqref{eq:b_small5}. We emphasize that we take $\etaa$ sufficiently small, but we do not need to take the limit $\etaa \to 0$. 

\bigskip

We now proceed in details. For any $0 < \etaa \leq 1$, we introduce the bilinear form
\begin{equation}
\label{eq:def_astar_mu}
\abil_\etaa(u,v) = \int_\Omega (\nabla u + bu) \cdot \nabla v + \etaa \int_\Omega u \, v.
\end{equation}
We have the following result:

\begin{proposition}
\label{prop:BS_mu}
We assume that~\eqref{hyp_H0_debut}--\eqref{hyp_H_11_num}--\eqref{hyp_H_22_num} hold. Let $u \in H^1(\Omega)$ and $u_h \in \Sigma_h$ such that
\begin{equation}
\label{eq:galerkin_orth}
\forall v \in \Sigma_h, \quad \abil_\etaa(u-u_h,v) = 0.
\end{equation}
Let $2 \leq p < \infty$ and assume that $u \in W^{1,p}(\Omega)$. Then, there exists $C_\etaa$ and $h_0(\etaa)$, that both depend on $\etaa$, such that, for any $0 < h < h_0(\etaa)$, we have
\begin{equation}
\label{eq:main1}
\| \nabla u_h \|_{L^p(\Omega)} \leq C_\etaa \, \| \nabla u \|_{L^p(\Omega)}.
\end{equation}
Assume furthermore that $u \in W^{2,p}(\Omega)$. Then
\begin{equation}
\label{eq:main2}
\|\nabla (u-u_h)\|_{L^p(\Omega)} \leq C_\etaa \, h \, \|u\|_{W^{2,p}(\Omega)}.
\end{equation}
\end{proposition}

The proof of Proposition~\ref{prop:BS_mu} is postponed until Appendix~\ref{sec:appendix_BS}. It follows the arguments of~\cite[Chap. 8]{BS}. Most presumably, a similar result can be obtained when $1 < p < 2$, using duality arguments as in~\cite[Sec. 8.5]{BS}. We do not need such a result here, and therefore do not proceed in that direction.

\medskip

We are now in position to prove Proposition~\ref{prop_brenner_scott_estimate}. 

\begin{proof}[Proof of Proposition~\ref{prop_brenner_scott_estimate}]

We define the function $g$ on $\partial \Omega$ by $g \equiv 0$ in the case of the invariant measure $\sigma_1$ and $\dis g = g_2 := b \cdot n -\fint_{\partial\Omega} b\cdot n$ on $\partial \Omega$ in the case of the invariant measure $\sigma_2^0$. Let $2 < p < \infty$ and let $\etaa > 0$ be small enough in a sense made precise below. The proof falls in three steps.

\medskip

\noindent {\bf Step 1.} Consider the following iterations: set $\sigma^0_\etaa = 1$ and define $\sigma^{m+1}_\etaa$ as the unique solution to the problem
\begin{equation}
\label{eq:maison3}
\left\{
\begin{array}{c}
\text{Find $\sigma^{m+1}_\etaa \in H^1(\Omega)$ such that, for any $v \in H^1(\Omega)$,}
\\ \noalign{\vskip 3pt}
\dis \abil_\etaa(\sigma^{m+1}_\etaa,v) = \etaa \int_\Omega \sigma^m_\etaa v + \int_{\partial \Omega} g v.
\end{array}
\right.
\end{equation}
Lemma~\ref{lem:these} in Appendix~\ref{sec:appendix_BS} below ensures that the above problem is well-posed (the bilinear form $\abil_\etaa$ satisfying an inf-sup condition in $H^1(\Omega)$) and that, if $\etaa$ is sufficiently small, $\sigma^m_\etaa \in W^{2,p}(\Omega)$ for any $m$. Taking $v \equiv 1$ in~\eqref{eq:maison3}, we observe that $\dis \fint_\Omega \sigma^m_\etaa = 1$ for any $m$. Furthermore, we see that
$$
\forall v \in H^1(\Omega), \qquad \abil(\sigma^{m+1}_\etaa,v) = \etaa \int_\Omega (\sigma^m_\etaa - \sigma^{m+1}_\etaa) \, v + \int_{\partial \Omega} g v.
$$
Using Proposition~\ref{prop_u_f_A_id} in the case $g \equiv 0$ (resp. Proposition~\ref{prop_u_f_sigma_two_A_id} in the case $g = g_2$), we get
$$
\| \sigma^{m+1}_\etaa - \sigma \|_{W^{2,p}(\Omega)} \leq C \etaa \| \sigma^m_\etaa - \sigma^{m+1}_\etaa \|_{L^p(\Omega)}.
$$
Taking $\etaa$ sufficiently small, this implies that
$$
\| \sigma^{m+1}_\etaa - \sigma \|_{W^{2,p}(\Omega)} \leq C \etaa \| \sigma^m_\etaa - \sigma \|_{L^p(\Omega)},
$$
and hence that
\begin{equation}
\label{eq:maison8}
\lim_{m \to \infty} \| \sigma^m_\etaa - \sigma \|_{W^{2,p}(\Omega)} = 0.
\end{equation}

\medskip

\noindent {\bf Step 2.} We now consider the following iterations, at the discrete level: set $\sigma^0_h = 1$ and define $\sigma^{m+1}_{\etaa,h}$ as the unique solution to the problem:
\begin{equation}
\label{eq:maison4}
\left\{
\begin{array}{c}
\text{Find $\sigma^{m+1}_{\etaa,h} \in \Sigma_h$ such that, for any $v \in \Sigma_h$,}
\\ \noalign{\vskip 3pt}
\dis \abil_\etaa(\sigma^{m+1}_{\etaa,h},v) = \etaa \int_\Omega \sigma^m_{\etaa,h} v + \int_{\partial \Omega} g v.
\end{array}
\right.
\end{equation}
Theorem~\ref{th:inv_measure_mu_h} in Appendix~\ref{sec:appendix_BS} below ensures that the above problem is well-posed (the proof of Theorem~\ref{th:inv_measure_mu_h} is performed in the case $g \equiv 0$, and it carries over to the case $g=g_2$). Taking $v \equiv 1$ in~\eqref{eq:maison4}, we observe that $\dis \fint_\Omega \sigma^m_{\etaa,h} = 1$ for any $m$. Furthermore, we see that
$$
\forall v \in \Sigma_h, \qquad \abil(\sigma^{m+1}_{\etaa,h} - \sigma_h,v) = \etaa \int_\Omega (\sigma^m_{\etaa,h} - \sigma^{m+1}_{\etaa,h}) \, v.
$$
We show in Appendix~\ref{sec:appendix_H1} below (see~\eqref{eq:carnon3}) that $\abil$ satisfies an inf-sup property on functions in $\Sigma_h$ of vanishing mean, with a constant $\gamma$ independent of $h$. Since $\dis \fint_\Omega \sigma^m_{\etaa,h} = 1 = \fint_\Omega \sigma_h$ for any $m$, we get
$$
\| \sigma^{m+1}_{\etaa,h} - \sigma_h \|_{H^1(\Omega)} \leq C \etaa \| \sigma^m_{\etaa,h} - \sigma^{m+1}_{\etaa,h} \|_{L^2(\Omega)}.
$$
Taking $\etaa$ sufficiently small, this implies that
$$
\| \sigma^{m+1}_{\etaa,h} - \sigma_h \|_{H^1(\Omega)} \leq C \etaa \| \sigma^m_{\etaa,h} - \sigma_h \|_{L^2(\Omega)},
$$
and hence that $\dis \lim_{m \to \infty} \| \sigma^m_{\etaa,h} - \sigma_h \|_{H^1(\Omega)} = 0$. By equivalence of the norms in the finite dimensional space $\Sigma_h$, this implies that
\begin{equation}
\label{eq:maison9}
\lim_{m \to \infty} \| \sigma^m_{\etaa,h} - \sigma_h \|_{W^{1,p}(\Omega)} = 0.
\end{equation}

\medskip

\noindent {\bf Step 3.} We eventually introduce the following problem:
\begin{equation}
\label{eq:maison5}
\left\{
\begin{array}{c}
\text{Find $\overline{\sigma}^{m+1}_\etaa \in H^1(\Omega)$ such that, for any $v \in H^1(\Omega)$,}
\\ \noalign{\vskip 3pt}
\dis \abil_\etaa(\overline{\sigma}^{m+1}_\etaa,v) = \etaa \int_\Omega \sigma^m_{\etaa,h} v + \int_{\partial \Omega} g v.
\end{array}
\right.
\end{equation}
We observe that~\eqref{eq:maison5} is the continuous analogue of~\eqref{eq:maison4}, for the {\em same} right-hand side (in contrast, when going from~\eqref{eq:maison3} to~\eqref{eq:maison4}, we modify both the space in which we search the solution and the right-hand side of the equation). We observe that
$$
\forall v \in \Sigma_h, \qquad \abil_\etaa(\overline{\sigma}^{m+1}_\etaa-\sigma^{m+1}_{\etaa,h},v) = 0.
$$
We are thus in position to use Proposition~\ref{prop:BS_mu}, which states that, for any $h < h_0(\etaa)$, we have
\begin{equation}
\label{eq:inria1}
\| \overline{\sigma}^{m+1}_\etaa - \sigma^{m+1}_{\etaa,h} \|_{W^{1,p}(\Omega)} \leq C_\etaa \, h \, \| \overline{\sigma}^{m+1}_\etaa \|_{W^{2,p}(\Omega)}.
\end{equation}
We now estimate $\| \overline{\sigma}^{m+1}_\etaa \|_{W^{2,p}(\Omega)}$. We observe that
$$
\forall v \in H^1(\Omega), \qquad \abil(\overline{\sigma}^{m+1}_\etaa,v) = \etaa \int_\Omega (\sigma^m_{\etaa,h} - \overline{\sigma}^{m+1}_\etaa) \, v + \int_{\partial \Omega} g v.
$$
In addition, taking $v \equiv 1$ in~\eqref{eq:maison5}, we see that $\dis \fint_\Omega \overline{\sigma}^{m+1}_\etaa = \fint_\Omega \sigma^m_{\etaa,h} = 1$. Using Proposition~\ref{prop_u_f_A_id} in the case $g \equiv 0$ (resp. Proposition~\ref{prop_u_f_sigma_two_A_id} in the case $g = g_2$), we get
$$
\| \overline{\sigma}^{m+1}_\etaa - \sigma \|_{W^{2,p}(\Omega)} \leq C \etaa \| \sigma^m_{\etaa,h} - \overline{\sigma}^{m+1}_\etaa \|_{L^p(\Omega)}.
$$
Taking $\etaa$ sufficiently small, this implies that
$$
\| \overline{\sigma}^{m+1}_\etaa - \sigma \|_{W^{2,p}(\Omega)} \leq C \etaa \| \sigma^m_{\etaa,h} - \sigma \|_{L^p(\Omega)},
$$
and hence
$$
\| \overline{\sigma}^{m+1}_\etaa \|_{W^{2,p}(\Omega)} \leq C \etaa \| \sigma^m_{\etaa,h} \|_{L^p(\Omega)} + C \| \sigma \|_{W^{2,p}(\Omega)}.
$$
Inserting this estimate in~\eqref{eq:inria1}, we deduce that
\begin{equation}
\label{eq:maison6}
\| \overline{\sigma}^{m+1}_\etaa - \sigma^{m+1}_{\etaa,h} \|_{W^{1,p}(\Omega)} \leq C_\etaa \, h \, \left( \etaa \| \sigma^m_{\etaa,h} \|_{L^p(\Omega)} + \| \sigma \|_{W^{2,p}(\Omega)} \right).
\end{equation}

\medskip

We now compare $\overline{\sigma}^{m+1}_\etaa$ and $\sigma^{m+1}_\etaa$. We observe that
$$
\forall v \in H^1(\Omega), \qquad \abil_\etaa(\overline{\sigma}^{m+1}_\etaa - \sigma^{m+1}_\etaa,v) = \etaa \int_\Omega (\sigma^m_{\etaa,h} - \sigma^m_\etaa) \, v,
$$
hence
$$
\forall v \in H^1(\Omega), \qquad \abil(\overline{\sigma}^{m+1}_\etaa - \sigma^{m+1}_\etaa,v) = \etaa \int_\Omega (\sigma^m_{\etaa,h} - \sigma^m_\etaa - \overline{\sigma}^{m+1}_\etaa + \sigma^{m+1}_\etaa) \, v.
$$
Using Proposition~\ref{prop_u_f_A_id}, we get that
$$
\| \overline{\sigma}^{m+1}_\etaa - \sigma^{m+1}_\etaa \|_{W^{2,p}(\Omega)} \leq C \etaa \| \sigma^m_{\etaa,h} - \sigma^m_\etaa - \overline{\sigma}^{m+1}_\etaa + \sigma^{m+1}_\etaa \|_{L^p(\Omega)},
$$
which implies, for $\etaa$ sufficiently small, that
$$
\| \overline{\sigma}^{m+1}_\etaa - \sigma^{m+1}_\etaa \|_{W^{2,p}(\Omega)} \leq C \etaa \| \sigma^m_{\etaa,h} - \sigma^m_\etaa \|_{L^p(\Omega)},
$$
and thus
\begin{equation}
\label{eq:maison7}
\| \overline{\sigma}^{m+1}_\etaa - \sigma^{m+1}_\etaa \|_{W^{1,p}(\Omega)} \leq C \etaa \| \sigma^m_{\etaa,h} - \sigma^m_\etaa \|_{W^{1,p}(\Omega)}.
\end{equation}

\medskip

\noindent {\bf Step 4.} Collecting~\eqref{eq:maison6} and~\eqref{eq:maison7}, we get
$$
\| \sigma^{m+1}_\etaa - \sigma^{m+1}_{\etaa,h} \|_{W^{1,p}(\Omega)} \leq C_\etaa \, h \, \left( \etaa \| \sigma^m_{\etaa,h} \|_{L^p(\Omega)} + \| \sigma \|_{W^{2,p}(\Omega)} \right) + C \etaa \| \sigma^m_{\etaa,h} - \sigma^m_\etaa \|_{W^{1,p}(\Omega)}.
$$
Using~\eqref{eq:maison8} and~\eqref{eq:maison9}, we are in position to pass to the limit $m \to \infty$. We deduce that 
$$
\| \sigma - \sigma_h \|_{W^{1,p}(\Omega)} \leq C_\etaa \, h \, \left( \etaa \| \sigma_h \|_{L^p(\Omega)} + \| \sigma \|_{W^{2,p}(\Omega)} \right) + C \etaa \| \sigma_h - \sigma \|_{W^{1,p}(\Omega)},
$$
and thus, for $\etaa$ sufficiently small,
\begin{eqnarray*}
\| \sigma - \sigma_h \|_{W^{1,p}(\Omega)}
&\leq&
C_\etaa \, h \, \left( \etaa \| \sigma_h \|_{L^p(\Omega)} + \| \sigma \|_{W^{2,p}(\Omega)} \right)
\\
&\leq&
C_\etaa \, h \, \left( \| \sigma_h - \sigma \|_{L^p(\Omega)} + \| \sigma \|_{W^{2,p}(\Omega)} \right).
\end{eqnarray*}
Taking $h < h_0(\etaa)$ such that $C_\etaa \, h \leq 1/2$, we deduce~\eqref{eq:estim_BS}.
\end{proof}

\subsubsection{Main result: estimation of $u - u_H$}

Proposition~\ref{prop_brenner_scott_estimate} allows to deduce from Proposition~\ref{prop_convergence_sigma_h} the following Theorem~\ref{prop_convergence_sigma_h_p1}. To this end, we successively consider the case of our two invariant measures. For our first invariant measure~$\sigma_1$, solution to~\eqref{invariant_measure_droniou}--\eqref{invariant_measure_droniou-int-1}--\eqref{invariant_measure_droniou-pos}, we obviously consider its $\mathbb{P}^1$ finite element approximation $\sigma_{1,h}$. For our second invariant measure, the analysis is essentially similar. There is, however, an additional subtlety in the very definition of the measure and its approximation. One, basic but crucial, remark is that the solution $u$ to~\eqref{pb_adv_diff_non_coercive_without_bc}--\eqref{boundary_cond_dirichlet_homog} does not depend on the choice of~$\sigma$. More precisely, \eqref{pb_adv_diff_non_coercive_without_bc}--\eqref{boundary_cond_dirichlet_homog} is equivalent to~\eqref{eq_modified_problem} irrespectively of the choice of~$\sigma$. We use this flexibility for our numerical analysis:
\begin{itemize}
\item[(i)] we first approximate~$\sigma_1$ as above by $\sigma_{1,h}$, and assume that $\sigma_{1,h}>0$ on $\Omega$. In practice, we have always numerically observed this property for sufficiently small $h$ (this property actually often holds for $h$ not asymptotically small). In addition, from a theoretical viewpoint, this bound from below is a consequence of the assumptions~\eqref{hyp_H0_debut}--\eqref{hyp_H_11_num}--\eqref{hyp_H_22_num}, as explained in the proof of Theorem~\ref{prop_convergence_sigma_h_p1} below. 
\item[(ii)] we next approximate, again using $\mathbb{P}^1$ finite elements, $\sigma_2^0$ solution to~\eqref{sig_helmholtz_hodge} by $\sigma_{2,h}^0$. We perform both these approximations on the same regular mesh~${\mathcal T}_h$. We then define $\sigma_{2,h}=\sigma_{2,h}^0+\kappa_h \, \sigma_{1,h}$ where 
$$
\kappa_h = 1+ \inf \left\{ \overline{\kappa} \in [0,\infty) \ \text{such that} \ \sigma_{2,h}^0+ \overline{\kappa} \, \sigma_{1,h} > 0 \ \text{on $\Omega$} \right\}.
$$
\end{itemize}
Precisely since, as noticed above, $u$ does not depend on our choice of invariant measure, we correspondingly define~$\sigma_2=\sigma_2^0+\kappa_h \, \sigma_1$. The point of our analysis is then to estimate $\sigma_2-\sigma_{2,h}$. The detail is contained in the following proof. 

\begin{theorem}
\label{prop_convergence_sigma_h_p1}

Assume that the invariant measure ($\sigma_1$, as defined by~\eqref{invariant_measure_droniou}--\eqref{invariant_measure_droniou-int-1}--\eqref{invariant_measure_droniou-pos}, or $\sigma_2$ a solution to~\eqref{sigma_2positive}) is approximated as we have just described in items~(i) and (ii) above. Under the assumptions~\eqref{hyp_H0_debut}--\eqref{hyp_H_11_num}--\eqref{hyp_H_22_num}, we have, for $h$ sufficiently small,
\begin{equation}
\|u-u_H\|_{H^1(\Omega)} 
\leq C\,h\,\|f\|_{L^2(\Omega)}
+C\, \inf_{v_H\in U_H} \Bigg[ \|u-v_H\|_{H^1(\Omega)} + h \, \|v_H\|_{H^1(\Omega)}\Bigg]
\label{ineq_13}
\end{equation}
for a constant $C$ independent of~$h$.
%
%
\end{theorem}

\begin{proof} 
We first consider the case of our first invariant measure~$\sigma_1$, solution to~\eqref{invariant_measure_droniou}--\eqref{invariant_measure_droniou-int-1}--\eqref{invariant_measure_droniou-pos}, approximated by its $\mathbb{P}^1$ finite element approximation $\sigma_{1,h}$. We have shown in Proposition~\ref{prop_brenner_scott_estimate} that $\| \sigma_1 -\sigma_{1,h} \|_{W^{1,p}(\Omega)}=O(h)$ for any $2 < p < \infty$. Using the Sobolev injections, we obtain that $\| \sigma_1 -\sigma_{1,h} \|_{C^0(\Omega)}=O(h)$. Since $\dis \inf_\Omega \sigma_1 > 0$, we get that, for any $h$ sufficiently small, $\dis \inf_\Omega \sigma_{1,h} \geq c_{\min} > 0$ for some $c_{\min}$ independent of $h$. The estimate~\eqref{ineq_12} thus holds. In order to deduce~\eqref{ineq_13} from~\eqref{ineq_12}, we have to estimate both $\| \sigma_1-\sigma_{1,h} \|_{L^\infty(\Omega)}$ and $\|\nabla (\sigma_1 -\sigma_{1,h})\|_{L^p(\Omega)}$ by $O(h)$ terms. Since, for $p$ sufficiently large, the $L^\infty$ norm is controlled by the $W^{1,p}$ norm, we will conclude our proof if we show that $\| \sigma_1 -\sigma_{1,h} \|_{W^{1,p}(\Omega)}=O(h)$. This latter bound is precisely the purpose of Proposition~\ref{prop_brenner_scott_estimate}.

\medskip

As we said, the case of the second invariant measure is essentially similar with the suitable definition of $\sigma_2$ and $\sigma_{2,h}$. Note first that, when $h$ is sufficiently small, we have $\dis \inf_\Omega \sigma_{1,h} \geq c_{\min} > 0$ for some $c_{\min}$ independent of $h$, as explained above. We then observe that
$$
\sigma_{2,h} = \sigma_{2,h}^0 + \kappa_h \, \sigma_{1,h} = \sigma_{2,h}^0 + (1+\overline{\kappa}_h) \, \sigma_{1,h} \geq c_{\min} > 0,
$$
since, by definition of $\overline{\kappa}_h$, we have $\sigma_{2,h}^0 + \overline{\kappa}_h \, \sigma_{1,h} \geq 0$ on $\Omega$. The estimate~\eqref{ineq_12} thus again holds. 

In our argument to establish~\eqref{ineq_13}, we use $\| \sigma_2-\sigma_{2,h} \|_{L^\infty(\Omega)}$ and $\|\nabla (\sigma_2 -\sigma_{2,h})\|_{L^p(\Omega)}$ for those particular choices of~$\sigma_2$ and~$\sigma_{2,h}$. Obviously, 
$$
\sigma_2-\sigma_{2,h}=(\sigma_2^0-\sigma_{2,h}^0)+\kappa_h\,(\sigma_1-\sigma_{1,h}).
$$
The term $\sigma_1-\sigma_{1,h}$ has just been estimated above in the suitable norms, while the term $\sigma_2^0-\sigma_{2,h}^0$ is estimated similarly. Eventually, for $h$ sufficiently small, because of the convergence in $C^0(\Omega)$ of the $\mathbb{P}^1$ finite element approximations, we know that~$\dis \sup_\Omega |\sigma_{2,h}^0|$ is bounded uniformly in~$h$ while $\dis \inf_\Omega \sigma_{1,h} \geq c_{\min} > 0$ for some $c_{\min}$ independent of $h$. Thus~$\kappa_h$ is bounded uniformly in~$h$, when $h$ is sufficiently small. The triangle inequality allows to conclude our proof. 
\end{proof}

\section{Implementation details and numerical results}
\label{sec:numerical-results}

\subsection{Discretization of the invariant measure}
\label{ssec:discrete_invariant_measure}

The numerical approximation of~$\sigma_1$, solution to~\eqref{invariant_measure_droniou}--\eqref{invariant_measure_droniou-int-1}--\eqref{invariant_measure_droniou-pos} is, as we said in the previous section, obtained using a classical $\mathbb{P}^1$ finite element space associated to a uniform mesh of size $h$ that we denote $\mathcal{T}_h$. 

Problem~\eqref{invariant_measure_droniou}--\eqref{invariant_measure_droniou-int-1}--\eqref{invariant_measure_droniou-pos} involves two constraints: a normalization constraint and a sign constraint. We comply with the normalization constraint by implementing an iterative algorithm. With a view to satisfying the positivity constraint, we use the adjoint equation and stabilize the problem (see~\eqref{eq:stabilization-sigma1} below). The stiffness matrix of the formulation employed to compute $\sigma_{1,h}$ is the adjoint matrix to the matrix of the Douglas-Wang (DW) stabilized version of the approximation of the solution to the advection-diffusion equation. Since, in most situations, it is observed that the latter approximation preserves the maximum principle, it is intuitively expected that the same applies for the adjoint formulation. Nevertheless, we have no general theoretical argument that shows our formulation guarantees positivity of the solution $\sigma_{1,h}$ (except of course if $h$ is sufficiently small, as shown in the proof of Theorem~\ref{prop_convergence_sigma_h_p1}). 

For essentially all the practical computations we have performed, we observe that positivity is preserved, in the sense that $\displaystyle \int_K \sigma_{1,h}$ is positive for all $K\in\mathcal{T}_H$. This is all we need to proceed with a coercive bilinear form at the discrete level for the advection-diffusion equation~\eqref{eq_modified_problem}, since we will be using $\mathbb{P}^1$ finite elements for the approximation $u_H$ of $u$ (see Section~\ref{ssec:discrete_u} below). We mention that Droniou~\cite{droniou2011neumann} and Tobiska~\cite{tobiska2012finite} have proposed a discretization that gives a positive discrete solution, but we do not proceed this way.

\medskip

We implement the following iterative algorithm: $\sigma_{1,h}^{n+1}$ is defined as the solution to the problem
\begin{equation}
  \label{eq:titi}
  \left\{
\begin{array}{c}
\text{Find $\sigma_{1,h}^{n+1}\in\Sigma_h$ such that, for all $\varphi\in\Sigma_h$},
\\ \noalign{\vskip 3pt}
\dis \abil\left(\sigma_{1,h}^{n+1},\varphi\right)
+ \lambda\int_\Omega \sigma_{1,h}^{n+1} \, \varphi + a_{\text{stab, \text{DW}}}\left(\sigma_{1,h}^{n+1},\varphi\right) = \lambda\int_\Omega \sigma_{1,h}^n \, \varphi,
\end{array}
\right.
\end{equation}
where $\lambda$ is a positive parameter, $\Sigma_h$ is the $\mathbb{P}^1$ finite element space associated to $\mathcal{T}_h$, $\abil$ is defined by~\eqref{eq:def_astar} and
\begin{equation}
\label{eq:stabilization-sigma1}
a_{\text{stab},\text{DW}}(\sigma,\varphi)=\sum_{K\in\mathcal{T}_h}\int_K \tau^\star \, \left[ \text{div}(\nabla \sigma + b \sigma) \right] \, ( -\Delta\varphi + b\cdot\nabla\varphi),
\end{equation}
with 
\begin{equation}
\tau^\star(x) = \frac{h}{2|b(x)|}\left(\coth(\text{Pe}^\star_K(x))-\frac{1}{\text{Pe}^\star_K(x)}\right),
\qquad
\text{Pe}^\star_K(x) = \frac{|b(x)|h}{2}.
\label{def_tau_star}
\end{equation}
Note that the stabilization term~\eqref{eq:stabilization-sigma1} is the standard DW stabilization term for the invariant measure equation~\eqref{invariant_measure_droniou}. The formulation is thus strongly consistent. 

We set $\lambda =10^{-3}$ in our numerical tests. The iterations are initialized with an approximation of the invariant measure of the potential part of $b$, namely~$I_{\Sigma_h}(e^{-\psi_H})$, where $I_{\Sigma_h}$ is the nodal interpolation operator in $\Sigma_h$ and $\psi_H\in U_H$ satisfies, for any $v_H\in U_H$, $\displaystyle \int_{\Omega}\nabla \psi_H\cdot\nabla v_H=\int_\Omega b\cdot\nabla v_H$. The stopping criterion we use is $\displaystyle\left\|1-\frac{\sigma_{1,h}^{n+1}}{\sigma_{1,h}^n}\right\|_{L^1(\Omega)}< 10^{-3}$. Temporarily ignoring the stabilization term $a_{\text{stab},\text{DW}}$ in~\eqref{eq:titi}, we formally see that using this stopping criterion aims at enforcing that $\dis \frac{\text{div}(\nabla \sigma_{1,h}^{n+1} + b \sigma_{1,h}^{n+1})}{\sigma_{1,h}^{n+1}}$ is small when we stop the iterations~\eqref{eq:titi}. This is a better criterion than enforcing that $\text{div}(\nabla \sigma_{1,h}^{n+1} + b \sigma_{1,h}^{n+1})$ is small, as $\sigma_1$ may vary a lot over the domain $\Omega$.

\bigskip

We similarly obtain an approximation $\sigma_{2,h}^0\in\Sigma_h$ of $\sigma_2^0$ solution to~\eqref{sig_helmholtz_hodge}, considering iterations where~$(\sigma_2^0)_h^{n+1}$ is the solution to the problem:
$$
\left\{
\begin{array}{c}
\text{Find $(\sigma_2^0)_h^{n+1}\in\Sigma_h$ such that, for all $\varphi\in\Sigma_h$},
\\ \noalign{\vskip 3pt}
\dis \abil\left((\sigma_2^0)_h^{n+1},\varphi\right)
+ \lambda\int_\Omega (\sigma_2^0)_h^{n+1} \, \varphi = \lambda\int_\Omega (\sigma_2^0)_h^n \, \varphi + \int_{\partial\Omega}\left(b\cdot n-\frac{1}{|\partial\Omega|} \int_{\partial\Omega}b\cdot n\right)\varphi.
\end{array}
\right.
$$
The iterations are initialized using~$(\sigma_2^0)_h^0=1$. The same stopping criterion is adopted as for $\sigma_1$. Notice that, in that case, we need not account for the positivity, which will be obtained by the combination $\sigma_{2,h}=(\sigma_2^0)_h+\kappa_h \, \sigma_{1,h}$ described above, so no stabilization of the formulation is employed. 

\subsection{Discretization of $u$}
\label{ssec:discrete_u}

A natural way to define the $(\mathbb{P}^1,\sigma_h\mathbb{P}^1)$ method would be to consider the following variational formulation (see~\eqref{varf_sigmah_symmetric}):
\begin{equation}
\text{Find $u_H \in \mathbb{P}^1(\mathcal{T}_H)$ s.t., for all $v_H\in \mathbb{P}^1(\mathcal{T}_H)$, \quad $a_{\text{ss}}(\sigma_h;u_H,v_H) = F(\sigma_h v_H)$},
\label{eq:form-var-uH-1}
\end{equation}
where we recall (see~\eqref{eq:faible-uH}--\eqref{eq:Bh}) that
\begin{align}
\label{eq:form-var-uH-2}
a_{\text{ss}}(\sigma_h;u_H,v_H) &= \int_\Omega \sigma_h\nabla u_H\cdot\nabla v_H + B_h\cdot\frac{(\nabla u_H) v_H -(\nabla v_H)u_H }{2},
\\
B_h &=\nabla \sigma_h +\sigma_h\,b.
\nonumber
\end{align}
In the case of the invariant measure $\sigma_1$, we need to add an extra term related to the stabilized discretization used for $\sigma_{1,h}$. The reason is the following. Formally, the variational formulation~\eqref{eq:form-var-uH-1}--\eqref{eq:form-var-uH-2} corresponds to the approximation
\begin{equation}
\label{eq:titi2}
-\hbox{\rm div}(\sigma_h\,\nabla u) + B_h \cdot \nabla u+ \frac{1}{2} \,(\hbox{\rm div}\,B_h)\,u=\sigma_h \, f
\end{equation}
of equation~\eqref{eq_modified_problem}. The zero order term in~$\hbox{\rm div}\,B_h$ (originating from the skew-symmetric formulation that ensures coercivity at the discrete level) affects the accuracy. Now, the stabilization~\eqref{eq:stabilization-sigma1} introduced in the variational formulation for~$\sigma_{1,h}$ amounts to modifying $B_h=(B_1)_h=\nabla \sigma_{1,h}+\sigma_{1,h}\,b$ into
\begin{equation}
(\overline{B}_1)_h = (B_1)_h + \left(\sum_{K\in\mathcal{T}_h}\tau^\star\text{div}((B_1)_h)\mathds{1}_K\right)b 
\label{eq:Bbarre}
\end{equation}
with $\tau^\star$ defined by~\eqref{def_tau_star}. More precisely, at convergence (i.e. when $n \to \infty$), the formulation~\eqref{eq:titi} amounts to requesting that, for any $\varphi\in\Sigma_h$, 
$$
\int_\Omega (\overline{B}_1)_h \cdot \nabla \varphi = 0
$$
rather than $\dis \int_\Omega (B_1)_h \cdot \nabla \varphi = 0$, which is the standard discretization of~\eqref{invariant_measure_droniou}. Formally, the quantity the divergence of which is zero is not $(B_1)_h$, but $(\overline{B}_1)_h$. In view of the last term of the left-hand side of~\eqref{eq:titi2}, and with the aim of obtaining the best possible accuracy, we thus need to modify $(B_1)_h$ into $(\overline{B}_1)_h$ in~\eqref{eq:form-var-uH-2}.

In order to be consistent, we therefore define
\begin{equation}
a_{\text{ss}}(\sigma_{1,h};u_H,v_H)= \int_\Omega \sigma_{1,h}\nabla u_H\cdot\nabla v_H + (\overline{B}_1)_h\cdot\frac{(\nabla u_H) v_H -(\nabla v_H)u_H }{2}
\label{eq:formulation-modifiee-barre}
\end{equation}
instead of~\eqref{eq:form-var-uH-2}. The problem is, by construction, coercive, and may be analyzed by the standard tools of numerical analysis we have used in the previous section. We readily note that the replacement of $(B_1)_h$ by $(\overline{B}_1)_h$ does not affect this analysis. Indeed, on any $K\in\mathcal{T}_h$,
$$
\hbox{\rm div}\left[(B_1)_h\right]=\hbox{\rm div}(\nabla \sigma_{1,h})+\hbox{\rm div}(\sigma_{1,h}\,b),
$$
where the first term vanishes for $\mathbb{P}^1$ finite elements. Thus, for any $p > d$,
$$
\left\| (\overline{B}_1)_h - (B_1)_h \right\|_{L^p(\Omega)}
\leq 
\left\| \tau^\star \, b \, \hbox{\rm div}(\sigma_{1,h}\,b) \right\|_{L^p(\Omega)}
\leq
\mathcal{C} h \, \| b \|_{W^{1,\infty}(\Omega)} \, \left\| \sigma_{1,h} \right\|_{W^{1,p}(\Omega)}
$$
where $\mathcal{C}$ is a universal constant such that $|\coth(y) - y^{-1} | \leq \mathcal{C}$ for any $y \in \R$. The factor $\left\| \sigma_{1,h} \right\|_{W^{1,p}(\Omega)}$ can be bounded from above independently of $h$ as a consequence of Proposition~\ref{prop_brenner_scott_estimate}. Using arguments similar to those used in the proofs of Proposition~\ref{prop_convergence_sigma_h} and Theorem~\ref{prop_convergence_sigma_h_p1}, we thus see that~\eqref{ineq_13} again holds when using the bilinear form~\eqref{eq:formulation-modifiee-barre} instead of~\eqref{eq:form-var-uH-2}.

\medskip

In the case of the invariant measure $\sigma_2$, we also need to add an extra term to~\eqref{eq:form-var-uH-1}--\eqref{eq:form-var-uH-2}. Recall that $\sigma_{2,h}=(\sigma_2^0)_h+\kappa_h \, \sigma_{1,h}$, where no stabilization is employed to compute $(\sigma_2^0)_h$, in contrast to $\sigma_{1,h}$. Formally, and again in view of the last term of the left-hand side of~\eqref{eq:titi2}, it seems advantageous to work with $(B_2^0)_h +\kappa_h \, (\overline{B}_1)_h$ rather than $(B_2^0)_h +\kappa_h \, (B_1)_h$, as we expect that the divergence of the former is smaller than that of the latter. In order to be consistent, we therefore define
\begin{equation}
a_{\text{ss}}(\sigma_{2,h};u_H,v_H)= \int_\Omega \sigma_{2,h}\nabla u_H\cdot\nabla v_H + (\overline{B}_2)_h\cdot\frac{(\nabla u_H) v_H -(\nabla v_H)u_H }{2}
\label{eq:formulation-modifiee-barre2}
\end{equation}
with $(\overline{B}_2)_h = (B_2^0)_h +\kappa_h \, (\overline{B}_1)_h$ instead of~\eqref{eq:form-var-uH-2} (recall that $(\overline{B}_1)_h$ is defined by~\eqref{eq:Bbarre} and that $(B_2^0)_h = \nabla (\sigma_2^0)_h + (\sigma_2^0)_h \, b$). As in the case of the invariant measure $\sigma_1$, the problem is, by construction, coercive, and may be analyzed by the standard tools of numerical analysis we have used in the previous section.

\medskip

In addition to the above practical and theoretical considerations, we also need to possibly modify the formulation when the problem is advection-domina\-ted. It turns out that we only need to use such a stabilized formulation when working with the invariant measure $\sigma_2$. In that case, we use a GLS type method and define the $(\mathbb{P}^1,\sigma_{2,h}\mathbb{P}^1)$-GLS method by the following variational formulation:
\begin{equation}
  \label{eq:stab-global}
  \left\{
\begin{array}{c}
\text{Find $u_H \in \mathbb{P}^1(\mathcal{T}_H)$ such that, for all $v_H\in \mathbb{P}^1(\mathcal{T}_H)$},
\\
a_{\text{ss}}(\sigma_{2,h};u_H,v_H) + a_{\text{stab}}(u_H,v_H) = F(\sigma_{2,h} v_H) + F_{\text{stab}}(v_H),
\end{array}
\right.
\end{equation}
where $a_{\text{ss}}$ is defined by~\eqref{eq:formulation-modifiee-barre2}, and
\begin{align}
\label{eq:a_stab}
a_{\text{stab}}(u_H,v_H)&=\sum_{K\in\mathcal{T}_H} \int_K \tau \, (\sigma_{2,h} b\cdot\nabla u_H) \, (\sigma_{2,h} b\cdot \nabla v_H),
\\
\nonumber
F_{\text{stab}}(v_H)&=\sum_{K\in\mathcal{T}_H}\int_K \tau \, (\sigma_{2,h} f) \, (\sigma_{2,h} b\cdot \nabla v_H),
\end{align}
with 
$$
\tau(x) = \frac{H}{2|(B_2)_h(x)|}\left(\coth(\text{Pe}_K(x))-\frac{1}{\text{Pe}_K(x)}\right),
\qquad
\text{Pe}_K(x) = \frac{|(B_2)_h(x)|H}{2\sigma_{2,h}(x)},
$$
where $(B_2)_h = (B_2^0)_h +\kappa_h \, (B_1)_h = \nabla \sigma_{2,h} + \sigma_{2,h} \, b$. Denoting $\mathcal{L}_h v = -\textnormal{div} (\sigma_h \nabla v) + B_h \cdot \nabla v$ the operator approximating that of~\eqref{eq_modified_problem} when only $\sigma_h$ is available, we indeed see that, on any $K \in \mathcal{T}_H$, we have
$$
\mathcal{L}_h u_H = \sigma_h b\cdot\nabla u_H
$$
as a consequence of the fact that $u_H$ is a $\mathbb{P}^1$ function. The term~\eqref{eq:a_stab} is indeed a GLS-type stabilization term, in the sense that it reads $\dis a_{\text{stab}}(u_H,v_H) = \sum_{K\in\mathcal{T}_H} \int_K \tau \, (\mathcal{L}_h u_H) \, (\mathcal{L}_h v_H)$.

\begin{remark}
We could have defined the function $\tau$ above using $(\overline{B}_2)_h$ instead of $(B_2)_h$. This leads to essentially identical numerical results.
\end{remark}

\subsection{Irrotational case}
\label{sec:section_irrotational_case}

For all our numerical tests throughout this article, we work on the unit square $\Omega=(0,1)^2$ and choose the right-hand side~$f=1$. All computations are performed on a Intel$\textregistered$ Xeon$\textregistered$ Processor E5-2667 v2. We use the FreeFem++ software~\cite{freefem}. 

\medskip

We assume in this Section~\ref{sec:section_irrotational_case} that the velocity field $b$ is irrotational: $b=\nabla \phi$. In that case, we know that $\displaystyle \sigma_1=\left(\fint_\Omega e^{-{\phi}}\right)^{-1}e^{-{\phi}}$ and that $\nabla\sigma_1 + \sigma_1\,b =0$ in $\Omega$. Specifically here, the velocity field $b$ is taken of the form
\begin{multline*}
b = \big(b_x^0,b_y^0\big)^T + \lambda_1 \big( \cos(2\pi x)\sin(2\pi y),\sin(2\pi x)\cos(2\pi y) \big)^T
\\
+\lambda_2 \big(\cos^2(2\pi x),0\big)^T +\lambda_3(y,x)^T,
\end{multline*} 
where $b^0_x$, $b^0_y$, $\lambda_1$, $\lambda_2$, $\lambda_3>0$. We take $b_x^0 = b_y^0 = 64$. The parameters $\lambda_1$, $\lambda_2$ and $\lambda_3$ are given in Table~\ref{pbs_inf_sup_d_2} for the four test cases (i) through~(iv) we consider. The last column of Table~\ref{pbs_inf_sup_d_2} shows that the problem is not coercive in the tests (ii) to (iv).

\medskip

\begin{table}[htbp]
\centering{
\begin{tabular}{c|ccc|c}
& $\lambda_1$ & $\lambda_2$ & $\lambda_3$ & $\dis \inf_{v_H\in U_H} \frac{a(v_H,v_H)}{\|v_H\|^2_{L^2(\Omega)}}$
\\ \noalign{\vskip 3pt}
\hline
Test (i) & 0 & 0 & 0 & 19.93 \\ 
Test (ii) & 0 & 50.34 & 0 & $-45.05$ \\ 
Test (iii) & 0 & 50.34 & 30 & $-45.05$ \\ 
Test (iv) & 20 & 50.34 & 0 & $-95.21$ 
\end{tabular}}
\caption{Definition of the parameters for the four discrete problems (i)-(iv)}
\label{pbs_inf_sup_d_2}
\end{table}

\medskip

Tables~\ref{table_10} through~\ref{table_13} show the relative error
\begin{equation}
\label{eq:def_erreur}
\text{err} = \frac{\| \nabla (u_H-u_{\rm ref}) \|_{L^2(\Omega\setminus\Omega_{\text{layer}})}}{\| \nabla u_{\rm ref} \|_{L^2(\Omega)}}
\end{equation}
for various numerical solutions $u_H$. In the convection-dominated regime, the solution presents a boundary layer of approximate width 
$$
\delta_{\text{layer}} = \frac{2}{\|b\|_{L^\infty(\Omega)}} \log \frac{\|b\|_{L^\infty(\Omega)}}{2}. 
$$
For the convection fields we consider, we set the boundary layer region (see Figure~\ref{non_coercive_omega_layer}) as 
$$
\Omega_{\text{layer}} = 
\big( (0,1)\times(1-\delta_{\text{layer}},1) \big) \cup \big( (1-\delta_{\text{layer}},1) \times(0,1) \big) \cup \big( (0,1) \times (0,\delta_{\text{layer}}) \big)
$$ 
and we only measure the accuracy of $u_H$ outside this layer. We have also assessed the accuracy in $L^2(\Omega)$ norm and obtained similar qualitative conclusions. The reference solution $u_{\rm ref}$ is computed using a $\mathbb{P}^1$ approach with a tiny mesh size. 

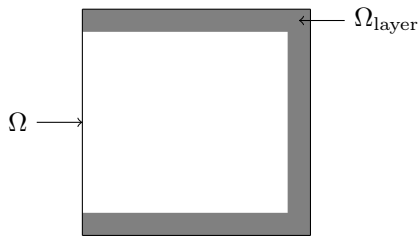
\begin{figure}[htbp]
\begin{center}
\begin{tikzpicture}[scale=3]
\fill [gray] (0.9,0.1) -- (0.,0.1)  -- (0,0) -- (1,0) -- (1,1) -- (0,1) -- (0,0.9) -- (0.9,0.9) -- cycle;
\draw (0,0) -- (1,0);
\draw (1,0) -- (1,1);
\draw (1,1) -- (0,1);
\draw (0,1) -- (0,0);
\draw [->] (-0.2,0.5) -- (0,0.5);
\draw (-0.2,0.5) node[left] {$\Omega$};
\draw [<-] (0.95,0.95) -- (1.15,0.95);
\draw (1.15,0.95) node[right] {$\Omega_{\text{layer}}$};
\end{tikzpicture}
\end{center}
\caption{The domain $\Omega_{\text{layer}}$ coloured in grey}
\label{non_coercive_omega_layer}
\end{figure} 

\medskip

For all approaches, we fix the mesh size $H=1/16$ for the approximation~$u_H$ of~$u$. We compare six approaches: the classical $\mathbb{P}^1$ finite element approximation (which may be unstable in the advection-dominated regime), its stabilized Galerkin least-square variant~$\mathbb{P}^1$-GLS, and our four approaches~$(\mathbb{P}^1 , \sigma_1\mathbb{P}^1)$, $(\mathbb{P}^1 , \sigma_{1,h}\mathbb{P}^1)$, $(\mathbb{P}^1 , \sigma_{2,h}\mathbb{P}^1)$ and $(\mathbb{P}^1 , \sigma_{2,h}\mathbb{P}^1)$-GLS, respectively using the exact value of~$\sigma_1$, its approximation~$\sigma_{1,h}$, and the approximation~$ \sigma_{2,h}$ of $\sigma_2$, either classical (as in~\eqref{eq:form-var-uH-1} with $a_{\text{ss}}$ defined by~\eqref{eq:formulation-modifiee-barre} or~\eqref{eq:formulation-modifiee-barre2}) or stabilized (as in~\eqref{eq:stab-global}). Note that the parameter $h$ is not used in the three first approaches.


\medskip

The comparison between $\mathbb{P}^1$ and~$\mathbb{P}^1$-GLS is used as an empirical measure of the instability of the problem. Likewise, comparing~$(\mathbb{P}^1 , \sigma_{2,h}\mathbb{P}^1)$ and $(\mathbb{P}^1 , \sigma_{2,h}\mathbb{P}^1)$-GLS allows to see the potential added value of a stabilization of the problem when using $\sigma_{2,h}$ as an approximation of the invariant measure. The approach~$(\mathbb{P}^1 , \sigma_1\mathbb{P}^1)$ using the exact value of the invariant measure is of course the most accurate one, and performs equally well as (and often better than)~$\mathbb{P}^1$-GLS. When we forbid ourselves to use that exact value of the invariant measure, $(\mathbb{P}^1 , \sigma_{1,h}\mathbb{P}^1)$ is the best method to use, and it does not require stabilization (see the third column of Tables~\ref{table_10} through~\ref{table_13}). Note yet that, if one has to work with $h=H$, then $(\mathbb{P}^1 , \sigma_{2,h}\mathbb{P}^1)$-GLS is the best method (see the second column of Tables~\ref{table_10} through~\ref{table_13}), providing results the accuracy of which is around 8\%. We also note the following fact. The two rightmost columns show tests that use a mesh to approximate~$\sigma$ that is not a subset of the mesh used to compute~$u$. In that case, the approach $(\mathbb{P}^1 , \sigma_{1,h}\mathbb{P}^1)$ deteriorates. In the present state of our understanding, we are unable to explain this phenomenon. We therefore advocate to employ~$(\mathbb{P}^1 , \sigma_{1,h}\mathbb{P}^1)$ with meshes that are a subset of one another, or, otherwise, to switch to $(\mathbb{P}^1 , \sigma_{2,h}\mathbb{P}^1)$-GLS.

\begin{remark}
\label{rem:h_grand}
For the Test (iv) reported on in Table~\ref{table_13}, and in the particular case $h=H$, it turns out that $\dis \int_K \sigma_{1,h}$ is not positive for all $K \in \mathcal{T}_H$ (it is positive for all the other values of $h$ considered, and for all the computations reported on in Tables~\ref{table_10} through~\ref{table_12}). In that case, it is thus not possible to use the $(\mathbb{P}^1 , \sigma_{1,h}\mathbb{P}^1)$ method. However, it turns out, still for that value of $h$, that there exists $\kappa_h$ such that $\dis \int_K \left( \sigma_{2,h}^0 + \kappa_h \, \sigma_{1,h} \right)$ is positive for all $K \in \mathcal{T}_H$. The $(\mathbb{P}^1 , \sigma_{2,h}\mathbb{P}^1)$-GLS approach can thus be used. 
\end{remark}

\begin{table}[htbp]
\centering{
\begin{tabular}{c|cccc}
Error~\eqref{eq:def_erreur} & ($h=H$) & ($h=H/5$) & ($h=1/150$) & ($h=1/230$) \\
\hline
$\mathbb{P}^1$ & 0.191 \\
$\mathbb{P}^1$-GLS & 0.0328\\
$(\mathbb{P}^1 , (\sigma_1)\mathbb{P}^1)$ & 0.0187 \\
$(\mathbb{P}^1 , \sigma_{1,h}\mathbb{P}^1)$ & 0.313 & 0.0208 & 0.127 & 0.0818 \\
$(\mathbb{P}^1 , \sigma_{2,h}\mathbb{P}^1)$ & 0.191 & 0.191 & 0.191 & 0.191 \\
$(\mathbb{P}^1 , \sigma_{2,h}\mathbb{P}^1)$-GLS & 0.0328 & 0.0328 & 0.0326 & 0.0327
\end{tabular}}
\caption{Test (i)}
\label{table_10}
\end{table}

\begin{table}[htbp]
\centering{
\begin{tabular}{c|cccc}
Error~\eqref{eq:def_erreur} & ($h=H$) & ($h=H/7$)& ($h=1/150$)& ($h=1/230$) \\
\hline
$\mathbb{P}^1$ & 0.479 \\
$\mathbb{P}^1$-GLS & 0.0551 \\
$(\mathbb{P}^1 , (\sigma_1)\mathbb{P}^1)$ & 0.0199 \\
$(\mathbb{P}^1 , \sigma_{1,h}\mathbb{P}^1)$ & 0.385 & 0.0218 & 0.139 & 0.102 \\
$(\mathbb{P}^1 , \sigma_{2,h}\mathbb{P}^1)$ & 0.614 & 0.398 & 0.362 & 0.377 \\
$(\mathbb{P}^1 , \sigma_{2,h}\mathbb{P}^1)$-GLS & 0.0827 & 0.0532 & 0.0511 & 0.0520
\end{tabular}}
\caption{Test (ii)}
\label{table_11}
\end{table}

\begin{table}[htbp]
\centering{
\begin{tabular}{c|cccc}
Error~\eqref{eq:def_erreur}& ($h=H$)& ($h=H/9$)& ($h=1/150$)& ($h=1/230$) \\
\hline
$\mathbb{P}^1$ & 0.536 \\
$\mathbb{P}^1$-GLS & 0.0411 \\
$(\mathbb{P}^1 , (\sigma_1)\mathbb{P}^1)$ & 0.0302 \\
$(\mathbb{P}^1 , \sigma_{1,h}\mathbb{P}^1)$ & 0.453& 0.0250 & 0.153 & 0.126 \\
$(\mathbb{P}^1 , \sigma_{2,h}\mathbb{P}^1)$ & 0.862 & 0.461 & 0.412 & 0.429 \\
$(\mathbb{P}^1 , \sigma_{2,h}\mathbb{P}^1)$-GLS & 0.0784 & 0.0420 & 0.0397 & 0.0405
\end{tabular}}
\caption{Test (iii)}
\label{table_12}
\end{table}

\begin{table}[htbp]
\centering{
\begin{tabular}{c|cccc}
Error~\eqref{eq:def_erreur} & ($h=H$) & ($h=H/7$) & ($h=1/150$) & ($h=1/230$) \\
\hline
$\mathbb{P}^1$ & 0.468 \\
$\mathbb{P}^1$-GLS & 0.0573\\
$(\mathbb{P}^1 , (\sigma_1)\mathbb{P}^1)$ & 0.0250 \\
$(\mathbb{P}^1 , \sigma_{1,h}\mathbb{P}^1)$ & - & 0.0266 & 0.153 & 0.111 \\
$(\mathbb{P}^1 , \sigma_{2,h}\mathbb{P}^1)$ & 0.612 & 0.401 & 0.379 & 0.388 \\
$(\mathbb{P}^1 , \sigma_{2,h}\mathbb{P}^1)$-GLS & 0.0894 & 0.0550 & 0.0535 & 0.0544
\end{tabular}}
\caption{Test (iv)}
\label{table_13}
\end{table}


\subsection{General case}

We now consider the general, not necessarily irrotational case. This time, $b$ reads as
\begin{multline*}
b = \big(b_x^0,b_y^0\big)^T + \lambda_1 \big(\cos(2\pi x)\sin(2\pi y),\sin(2\pi x)\cos(2\pi y) \big)^T
\\
+\lambda_2 \big(\cos^2(2\pi x),0\big)^T +\lambda_3(y,x)^T+\lambda_4(y,-x)^T,
\end{multline*}
where $b^0_x = b^0_y = 64$ and where $\lambda_1$, $\lambda_2$, $\lambda_3$, $\lambda_4>0$. We study the three examples defined in Table~\ref{pbs_inf_sup_d_2_general}.

\medskip

\begin{table}[htbp]
\centering{
\begin{tabular}{c|cccc|c}
& $\lambda_1$ & $\lambda_2$ &$\lambda_3$ & $\lambda_4$& $\dis \inf_{v_H\in U_H} \frac{a(v_H,v_H)}{\|v_H\|^2_{L^2(\Omega)}}$
\\
\hline \noalign{\vskip 3pt}
Test (v) & 0 & 50.34  & 0 & 64 & $-45.05$ \\ 
Test (vi) & 20 & 50.34 & 0 & 64 & $-95.21$ \\ 
Test (vii) & 0 & 50.34 & 30 & 64 & $-45.05$ 
\end{tabular}}
\caption{Definition of the parameters for the three discrete problems (v)-(vii)}
\label{pbs_inf_sup_d_2_general}
\end{table}

Tables~\ref{table_15} through~\ref{table_17} show our results, for $H=1/16$ as in the previous section. The approaches evaluated are identical to those of Tables~\ref{table_10} through~\ref{table_13}, with the notable exception of the approach $(\mathbb{P}^1 , \sigma_1\mathbb{P}^1)$ since now the exact invariant measure $\sigma_1$ is unknown. The conditions in which we perform our tests are identical. The results confirm our conclusions of the previous section. 

\begin{remark}
For the largest value of $h$ used in Tables~\ref{table_15} through~\ref{table_17}, it turns out that $\dis \int_K \sigma_{1,h}$ is not positive for all $K \in \mathcal{T}_H$ (see also Remark~\ref{rem:h_grand}). But there still exists $\kappa_h$ such that $\dis \int_K \left( \sigma_{2,h}^0 + \kappa_h \, \sigma_{1,h} \right)$ is positive for all $K \in \mathcal{T}_H$, which allows us to use the $(\mathbb{P}^1 , \sigma_{2,h}\mathbb{P}^1)$-GLS approach. 
\end{remark}

\begin{table}[htbp]
\centering{
\begin{tabular}{c|cccc}
Error~\eqref{eq:def_erreur} & ($h=1/17$) & ($h=H/7$) & ($h=1/150$) & ($h=1/230$) \\
\hline
$\mathbb{P}^1$ & 0.568 \\
$\mathbb{P}^1$-GLS & 0.0704 \\
$(\mathbb{P}^1 , \sigma_{1,h}\mathbb{P}^1)$ & - & 0.0390 & 0.146 & 0.0981 \\
$(\mathbb{P}^1 , \sigma_{2,h}\mathbb{P}^1)$ & 0.657 & 0.515 & 0.462 & 0.480 \\
$(\mathbb{P}^1 , \sigma_{2,h}\mathbb{P}^1)$-GLS & 0.117 & 0.0672 & 0.0658 & 0.0660
\end{tabular}}
\caption{Test (v)}
\label{table_15}
\end{table}

\begin{table}[htbp]
\centering{
\begin{tabular}{c|cccc}
Error~\eqref{eq:def_erreur}& ($h=H$)& ($h=H/7$)& ($h=1/150$)& ($h=1/230$) \\
\hline
$\mathbb{P}^1$ & 0.620 \\
$\mathbb{P}^1$-GLS & 0.0807 \\
$(\mathbb{P}^1 , \sigma_{1,h}\mathbb{P}^1)$ & - & 0.0549 & 0.151 & 0.105 \\
$(\mathbb{P}^1 , \sigma_{2,h}\mathbb{P}^1)$ & 0.641 & 0.522 & 0.482 & 0.495 \\
$(\mathbb{P}^1 , \sigma_{2,h}\mathbb{P}^1)$-GLS & 0.134 & 0.0772 & 0.0764 & 0.0768
\end{tabular}}
\caption{Test (vi)}
\label{table_16}
\end{table}

\begin{table}[h!]
\centering{
\begin{tabular}{c|cccc}
Error~\eqref{eq:def_erreur} & ($h=1/17$) & ($h=H/9$) & ($h=1/150$) & ($h=1/230$) \\
\hline
$\mathbb{P}^1$ & 0.636 \\
$\mathbb{P}^1$-GLS & 0.0606 \\
$(\mathbb{P}^1 , \sigma_{1,h}\mathbb{P}^1)$ & - & 0.0285 & 0.159 & 0.116 \\
$(\mathbb{P}^1 , \sigma_{2,h}\mathbb{P}^1)$ & 0.648 & 0.550 & 0.489 & 0.509 \\
$(\mathbb{P}^1 , \sigma_{2,h}\mathbb{P}^1)$-GLS & 0.112 & 0.0588 & 0.0571 & 0.0573
\end{tabular}}
\caption{Test (vii)}
\label{table_17}
\end{table}


\subsection{Computational cost and efficiency}

We now evaluate the computational cost of the most accurate of our approaches, namely the $(\mathbb{P}^1 , \sigma_{1,h}\mathbb{P}^1)$ method, given the results of the tests performed in the previous sections, as compared to the classical $\mathbb{P}^1$-GLS method. 

\medskip

As can be easily seen upon considering some particular situations where $\sigma_1$ is known analytically (some of these cases are considered in Section~\ref{sec:section_irrotational_case} above), the contrast of $\sigma_1$ over the domain, say measured by the ratio~$\displaystyle\frac{\sup_\Omega\sigma_1}{\inf_\Omega\sigma_1}$, may be huge, especially in the advection-dominated regime. Therefore, the stiffness matrix involved in the solution procedure for the modified equation~\eqref{eq_modified_problem} is often ill-conditioned. We therefore use, in our tests, a direct solver (from the UMFPACK library) for the linear algebraic systems. An alternate, equally effective approach is to use an iterative inversion algorithm together with a diagonal preconditioner. We have indeed tested such an approach in other tests not reproduced here, obtaining similar conclusions. In particular, the diagonal preconditionner, although simple, turns out to be very effective in diminishing the number of iterations. 

\subsubsection{Fixed cost} 

We compare the $(\mathbb{P}^1 , \sigma_{1,h}\mathbb{P}^1)$ method and the $\mathbb{P}^1$-GLS method at fixed cost. 
Tables~\ref{table_27} through~\ref{table_28} show the accuracy of the two methods for the tests~(v)-(vi)-(vii). Similar results have been obtained for our tests~(i) through (iv). We observe that the $\mathbb{P}^1$-GLS is definitely more accurate than the $(\mathbb{P}^1 , \sigma_{1,h}\mathbb{P}^1)$ method. However, as already mentioned and as will be confirmed in the next tests, the latter approach is more adequate in a multiquery context, where several resolutions of the advection-diffusion equation~\eqref{pb_adv_diff_non_coercive_without_bc}--\eqref{boundary_cond_dirichlet_homog} are to be performed. 

\begin{table}[htbp]
\centering{
\begin{tabular}{c|cc}
	& cost & Error~\eqref{eq:def_erreur} \\
\hline
$\mathbb{P}^1$-GLS ($H=1/122$) & 4.76 & 0.00293 \\
$(\mathbb{P}^1 , \sigma_{1,h}\mathbb{P}^1)$ ($H=1/16$, $h=H/7$) & 4.79 & 0.0390
\end{tabular}}
\caption{Test (v)}
\label{table_27}
\end{table}

\begin{table}[htbp]
\centering{
\begin{tabular}{c|cc}
	& cost & Error~\eqref{eq:def_erreur} \\
\hline
$\mathbb{P}^1$-GLS ($H=1/127$) & 6.42 & 0.00485 \\
$(\mathbb{P}^1 , \sigma_{1,h}\mathbb{P}^1)$ ($H=1/16$, $h=H/7$) & 6.30 & 0.0549
\end{tabular}}
\caption{Test (vi)}
\label{table_30}
\end{table}

\begin{table}[h!]
\centering{
\begin{tabular}{c|cc}
	& cost & Error~\eqref{eq:def_erreur} \\
\hline
$\mathbb{P}^1$-GLS ($H=1/144$) & 7.43 &	0.00143 \\
$(\mathbb{P}^1 , \sigma_{1,h}\mathbb{P}^1)$ ($H=1/16$, $h=H/9$) & 7.09 & 0.0285
\end{tabular}}
\caption{Test (vii)}
\label{table_28}
\end{table}

\subsubsection{Fixed meshsize $h$}

We fix the meshsize $h=1/2048$. In order to measure the cost of the methods in a multiquery context, we distinguish, in the computational cost of the $(\mathbb{P}^1 , \sigma_{1,h}\mathbb{P}^1)$ method, (i) the offline cost, which comprises the assembling phase of the stiffness matrix, which itself involves the pre-computation of $\sigma_{1,h}$, and (ii) the online cost equal to the resolution time for the modified advection-diffusion equation. 

Tables~\ref{table_34} through~\ref{table_36} show our results for the $(\mathbb{P}^1 , \sigma_{1,h}\mathbb{P}^1)$ method and the $\mathbb{P}^1$-GLS method for the tests~(v)-(vi)-(vii). Again, similar results we do not show and which lead to similar conclusions have been obtained for our tests~(i) through (iv). The two columns on the left of each table show the relative accuracy obtained for different mesh sizes~$H$ (employed, we recall, for the approximation of the advection-diffusion equation). The two columns on the right allow to compare the online cost of the $(\mathbb{P}^1 , \sigma_{1,h}\mathbb{P}^1)$ method, as defined above, with the total cost of the the $\mathbb{P}^1$-GLS method. The specific function used to measure the CPU time is clock\_gettime() with the clock CLOCK\_PROCESS\_CPUTIME\_ID. 

\medskip

The main two conclusions are, on the one hand, that the $(\mathbb{P}^1 , \sigma_{1,h}\mathbb{P}^1)$ method is more robust and allow for larger mesh sizes than the $\mathbb{P}^1$-GLS method, and, on the other hand, that the two approaches essentially share the same cost, if we assume that $\sigma_{1,h}$ has been precomputed. Other tests, not reported on here, show that roughly ten solutions of the advection-diffusion equation are necessary to make the approach profitable if we take into account the cost to compute $\sigma_{1,h}$. 

\begin{table}[htbp]
\centering{
\begin{tabular}{ccc|cc}
 & Error~\eqref{eq:def_erreur} & & Online cost & \\ 
$1/H$	& $(\mathbb{P}^1 , \sigma_{1,h}\mathbb{P}^1)$ & $\mathbb{P}^1$-GLS& $(\mathbb{P}^1 , \sigma_{1,h}\mathbb{P}^1)$ & $\mathbb{P}^1$-GLS \\
\hline
16	& 0.0291 & 0.0704 & 0.00107 & 0.000921 \\ 
24	& 0.0190 & 0.0508 & 0.00177 & 0.00171 \\ 
28	& 0.0173 & 0.0235 & 0.00250 & 0.00229 \\ 
32	& 0.0135 & 0.0187 & 0.00312 & 0.00294 \\ 
64	& 0.00626 & 0.00608 & 0.0138 & 0.0137 
\end{tabular}}
\caption{Test (v)}
\label{table_34}
\end{table}

\begin{table}[htbp]
\centering{
\begin{tabular}{ccc|cc}
 & Error~\eqref{eq:def_erreur} & & Online cost & \\ 
$1/H$	& $(\mathbb{P}^1 , \sigma_{1,h}\mathbb{P}^1)$ & $\mathbb{P}^1$-GLS& $(\mathbb{P}^1 , \sigma_{1,h}\mathbb{P}^1)$ & $\mathbb{P}^1$-GLS \\
\hline
16	& 0.0475 & 0.0807 & 0.00102 & 0.000804 \\
24	& 0.0308 & 0.0615 & 0.00157 & 0.00166 \\
28	& 0.0275 & 0.0440 & 0.00238 & 0.00232 \\
32	& 0.0226 & 0.0258 & 0.00298 & 0.00301 \\
64	& 0.0105 & 0.0102 & 0.0143 & 0.0139 
\end{tabular}}
\caption{Test (vi)}
\label{table_35}
\end{table}

\begin{table}[h!]
\centering{
\begin{tabular}{ccc|cc}
 & Error~\eqref{eq:def_erreur} & & Online cost & \\ 
$1/H$	& $(\mathbb{P}^1 , \sigma_{1,h}\mathbb{P}^1)$ & $\mathbb{P}^1$-GLS& $(\mathbb{P}^1 , \sigma_{1,h}\mathbb{P}^1)$ & $\mathbb{P}^1$-GLS \\
\hline
16	& 0.0219 & 0.0606 & 0.000976 & 0.000807 \\
24	& 0.0139 & 0.0437 & 0.00186 & 0.00171 \\
28	& 0.0119 & 0.0189 & 0.00256 & 0.00229 \\
32	& 0.00946 & 0.0146 & 0.00332 & 0.0296 \\
64	& 0.00401 & 0.00388 & 0.0156 & 0.0138 
\end{tabular}}
\caption{Test (vii)}
\label{table_36}
\bigskip
\end{table}

\section*{Acknowledgments}

The work of the authors is partially supported by the ONR under grant N00014-15-1-2777 and by the EOARD under grant FA8655-13-1-3061. Stimulating discussions with Y.~Achdou and O.~Pironneau are gratefully acknowledged. We warmly thank A. Lozinski for his remarks on a draft version of this article.


\appendix

\section{Proof of Proposition~\ref{prop_brenner_scott_estimate_debut}}
\label{sec:appendix_H1}

Proposition~\ref{prop_brenner_scott_estimate_debut} is shown as a consequence of a more general result, namely Theorem~\ref{th:main_debut} below. We introduce the space
$$
V=\left\{u \in H^1(\Omega), \quad \fint_\Omega u = 1 \right\}
$$
and recall (see~\eqref{eq:def_astar}) that the bilinear form $\abil$ is defined by
$$
\abil(u,v) = \int_\Omega (\nabla u + bu) \cdot \nabla v.
$$
Let $f \in L^2(\Omega)$ and $g$ be Lipschitz-continuous on $\partial \Omega$ such that $\dis \int_\Omega f + \int_{\partial \Omega} g = 0$. Consider the problem
\begin{equation}
\label{eq:pb_disc}
\text{Find $u \in V$ such that, for any $v \in H^1(\Omega)$, \quad $\abil(u,v) = \int_\Omega f v + \int_{\partial \Omega} g v$}.
\end{equation}
Under assumptions~\eqref{hyp_H0_debut}--\eqref{hyp_H_11_num}--\eqref{hyp_H_22_num}, we have shown above (in Proposition~\ref{prop_u_f_A_id} for the specific case $g=0$ and in Proposition~\ref{prop_u_f_sigma_two_A_id} for the case $\dis g = b\cdot n -\fint_{\partial\Omega} b\cdot n$, where we actually did not use the specific expression of $g$) that problem~\eqref{eq:pb_disc} is well-posed, and that its unique solution belongs to $H^2(\Omega)$.

We here consider the Galerkin discretization of~\eqref{eq:pb_disc}. Consider a mesh of $\Omega$ made of elements $T \in \mathcal{T}_h$. Following Proposition~\ref{prop_brenner_scott_estimate_debut}, we take $\Sigma_h \subset H^1(\Omega)$ the associated finite dimensional space made of continuous piecewise affine functions and introduce
$$
V_h = \left\{ u \in \Sigma_h, \quad \fint_\Omega u = 1 \right\} \subset V.
$$

\begin{theorem}
\label{th:main_debut}
We assume that~\eqref{hyp_H0_debut}--\eqref{hyp_H_11_num}--\eqref{hyp_H_22_num} hold. Let $u$ denote the solution to~\eqref{eq:pb_disc}. For $h$ sufficiently small, there exists a unique $u_h \in V_h$ solution to
\begin{equation}
\label{eq:carnon1}
\forall v_h \in \Sigma_h, \qquad \abil(u_h,v_h) = \int_\Omega f v_h + \int_{\partial \Omega} g \, v_h.
\end{equation}
Furthermore, we have, for $h$ sufficiently small,
\begin{equation}
\label{eq:carnon1_bis}
\| u - u_h \|_{H^1(\Omega)} \leq C h \| u \|_{H^2(\Omega)}
\end{equation}
where $C$ is independent of $h$. 
\end{theorem}

Theorem~\ref{th:main_debut} obviously implies Proposition~\ref{prop_brenner_scott_estimate_debut}. Consider indeed the invariant measure $\sigma_1 \in H^1(\Omega)$ solution to~\eqref{invariant_measure_droniou}--\eqref{invariant_measure_droniou-int-1}. It is the solution to~\eqref{eq:pb_disc} with $f=g=0$. Likewise, the invariant measure $\sigma_2^0 \in H^1(\Omega)$ solution to~\eqref{sig_helmholtz_hodge} is the solution to~\eqref{eq:pb_disc} with $f=0$ and $\dis g = b\cdot n -\fint_{\partial\Omega} b\cdot n$. Theorem~\ref{th:main_debut} then implies that~\eqref{eq:sigma_h_bien_pose} is well-posed and that the error estimate~\eqref{eq:error_sigma_h} holds. 

\begin{proof}[Proof of Theorem~\ref{th:main_debut}]

The proof falls in two steps. 

\medskip

\noindent {\bf Step 1: well-posedness of~\eqref{eq:carnon1}.}
Let $\lambda > 0$. The bilinear form $\dis a_{\rm coer}(u,v) = \int_\Omega \nabla v \cdot \nabla u + \lambda \int_\Omega v \, u$ is coercive in $H^1(\Omega)$, while the bilinear form $\dis a_{\rm comp}(u,v) = \int_\Omega b u \cdot \nabla v$ can be represented by a compact operator $\dis T \in \mathcal{L} \left( H^1(\Omega), \left( H^1(\Omega) \right)' \right)$ as $a_{\rm comp}(u,v) = \langle Tu,v \rangle$. Consequently (see the proof of~\cite[Theorem 4.2.9]{livre-sauter}),
\begin{equation}
\label{eq:these8}
\begin{array}{c}
\text{when $h$ is sufficiently small, the bilinear form}
\\
\text{$\abil_\lambda(u,v) = a_{\rm coer}(u,v) + a_{\rm comp}(u,v)$ satisfies an inf-sup condition on $\Sigma_h$.}
\end{array}
\end{equation}
Using~\cite[Prop. 2.21]{EG}, we thus see that the problem
$$
\text{Find $u_h \in \Sigma_h$ such that, for all $v_h \in \Sigma_h$, \quad $\abil_\lambda(u_h,v_h) = \int_\Omega \overline{f} v_h + \int_{\partial \Omega} g \, v_h$},
$$
is well-posed for any $\overline{f} \in L^2(\Omega)$. 

We now consider the iterations
\begin{equation}
\label{eq:carnon2}
\left\{
\begin{array}{c}
\text{Find $u_h^{n+1} \in \Sigma_h$ such that, for all $v_h \in \Sigma_h$},
\\ \noalign{\vskip 3pt}
\dis \abil(u_h^{n+1},v_h) + \lambda \int_\Omega u_h^{n+1} \, v_h 
=
\abil_\lambda(u_h^{n+1},v_h) 
= 
\lambda \int_\Omega u_h^n \, v_h + \int_\Omega f v_h + \int_{\partial \Omega} g \, v_h,
\end{array}
\right.
\hspace{-13mm}
\end{equation}
with the initial condition $u_h^0 = |\Omega|^{-1}$ (or any function in $\Sigma_h$ of mean equal to 1). Thanks to the above argument, these problems are well-posed and define a sequence $u_h^n \in \Sigma_h \subset H^1(\Omega)$. Furthermore, taking $v_h=1$ as test function, we see that all the functions $u_h^n$ share the same mean, and due to the choice of $u_h^0$, we get $u_h^n \in V_h$ for any $n$.

We next prove that the sequence $\left\{ u_h^n \right\}_{n \in \N}$ converges to a solution to~\eqref{eq:carnon1}. We recall that $\dis H^1_{\int=0}(\Omega) = \left\{ v\in H^1(\Omega), \quad \int_\Omega v=0 \right\}$. We infer from~\eqref{eq:carnon2} that, for any $v_h \in \Sigma_h \cap H^1_{\int=0}(\Omega)$,
$$
\abil(u_h^{n+1} - u_h^n,v_h) + \lambda \int_\Omega (u_h^{n+1}-u_h^n) \, v_h 
=
\lambda \int_\Omega (u_h^n - u_h^{n-1}) \, v_h, 
$$
from which we deduce that
\begin{equation}
\label{eq:titit}
\sup_{v_h \in \Sigma_h \cap H^1_{\int=0}(\Omega) } \frac{\abil(u_h^{n+1} - u_h^n,v_h)}{\| v_h \|_{H^1(\Omega)}} - \lambda \| u_h^{n+1} - u_h^n \|_{H^1(\Omega)} 
\leq
\lambda \| u_h^n - u_h^{n-1} \|_{H^1(\Omega)}.
\hspace{-3mm}
\end{equation}
Using the same arguments as above (this time for $\lambda=0$ and on $\Sigma_h \cap H^1_{\int=0}(\Omega)$), we have that the bilinear form $\abil$ satisfies an inf-sup condition on $\Sigma_h \cap H^1_{\int=0}(\Omega)$ with a constant $\gamma>0$ independent of $h$:
\begin{equation}
\label{eq:carnon3}
\inf_{w_h \in \Sigma_h \cap H^1_{\int=0}(\Omega)} \ \sup_{v_h \in \Sigma_h \cap H^1_{\int=0}(\Omega)} \ \frac{\abil(w_h,v_h)}{\| w_h \|_{H^1(\Omega)} \ \| v_h \|_{H^1(\Omega)}} \geq \gamma.
\end{equation}
We thus infer from~\eqref{eq:titit} that
$$
(\gamma - \lambda) \| u_h^{n+1} - u_h^n \|_{H^1(\Omega)} 
\leq
\lambda \| u_h^n - u_h^{n-1} \|_{H^1(\Omega)}.
$$
Taking $\lambda$ sufficiently small (so that $0 < \lambda/(\gamma - \lambda) < 1$), we obtain that the sequence $\left\{ u_h^n \right\}_{n \in \N}$ converges in $H^1(\Omega)$ to some $u_h^\infty \in V_h$. Passing to the limit $n \to \infty$ in~\eqref{eq:carnon2}, we get that $u_h^\infty$ is a solution to~\eqref{eq:carnon1}.

\medskip

We now prove that~\eqref{eq:carnon1} has a unique solution. Consider two solutions $u_h$ and $\overline{u}_h$ to~\eqref{eq:carnon1}. Then $u_h - \overline{u}_h \in \Sigma_h \cap H^1_{\int=0}(\Omega)$ and satisfies $\abil(u_h - \overline{u}_h,v_h) = 0$ for any $v_h \in \Sigma_h \cap H^1_{\int=0}(\Omega)$. We deduce from~\eqref{eq:carnon3} that $u_h - \overline{u}_h=0$. 

\medskip

\noindent {\bf Step 2: estimate~\eqref{eq:carnon1_bis}.} 
Introducing the interpolant $I_h u \in \Sigma_h$, we deduce from~\eqref{eq:carnon3} that
$$
\gamma \| I_h u - u_h - c \|_{H^1(\Omega)} \leq \sup_{w_h \in \Sigma_h \cap H^1_{\int=0}(\Omega)} \ \frac{\abil(I_h u - u_h - c,w_h)}{\| w_h \|_{H^1(\Omega)}},
$$
where $\dis c = \fint_\Omega (I_h u - u_h)$. Using~\eqref{eq:pb_disc} and~\eqref{eq:carnon1}, we deduce from the above estimate that
\begin{eqnarray}
\gamma \| I_h u - u_h - c \|_{H^1(\Omega)} 
&\leq& 
\sup_{w_h \in \Sigma_h \cap H^1_{\int=0}(\Omega)} \ \frac{\abil(I_h u - u - c,w_h)}{\| w_h \|_{H^1(\Omega)}}
\nonumber
\\
&\leq&
\left (1 + \| b \|_{L^\infty(\Omega)} \right) \| I_h u - u - c \|_{H^1(\Omega)}.
\label{eq:carnon4}
\end{eqnarray}
We next write that
\begin{eqnarray*}
\| u - u_h \|_{H^1(\Omega)} 
& \leq &
\| u - I_h u + c \|_{H^1(\Omega)} + \| I_h u - u_h - c \|_{H^1(\Omega)} 
\\
& \leq &
C \| u - I_h u + c \|_{H^1(\Omega)} \quad \text{[using~\eqref{eq:carnon4}]}
\\
& \leq &
C \| \nabla (u - I_h u) \|_{L^2(\Omega)},
\end{eqnarray*}
where, in the last line, we have used the Poincar\'e-Wirtinger inequality, as a consequence of the fact that $\dis c = \fint_\Omega (I_h u - u_h) = \fint_\Omega (I_h u - u)$. We then conclude using the approximation result $\| u - I_h u \|_{H^1(\Omega)} \leq C h \| u \|_{H^2(\Omega)}$. This yields~\eqref{eq:carnon1_bis}.
\end{proof}

\section{Proof of Proposition~\ref{prop:BS_mu}}
\label{sec:appendix_BS}

In order to prove Proposition~\ref{prop:BS_mu}, we follow and adapt the arguments of~\cite[Chap. 8]{BS}. The proof relies on several technical results, the proof of which are given in the subsequent appendices~\ref{sec:proof_lem4} and~\ref{sec:proofs}.

We fix some $0 < \etaa \leq 1$, and we recall (see~\eqref{eq:def_astar_mu}) that the bilinear form $\abil_\etaa$ is defined by
$$
\abil_\etaa(u,v) = \int_\Omega (\nabla u + bu) \cdot \nabla v + \etaa \int_\Omega u \, v.
$$
We also define $a_\etaa(u,v) = \abil_\etaa(v,u)$.

\subsection{Elliptic regularity results}

Let $q$ be such that $1< q <+\infty$ if $d=2$ and $2d/(d+2) \leq q <+\infty$ otherwise, and let $f \in L^q(\Omega)$. We consider the problem
\begin{equation}
\label{eq:pb}
\text{Find $u \in H^1(\Omega)$ such that, for any $v \in H^1(\Omega)$, \quad $\abil_\etaa(u,v) = \int_\Omega f v$},
\end{equation}
for which we have the following result.

\begin{lemma}
\label{lem:these}
We work under the assumptions of Proposition~\ref{prop:BS_mu} and assume that $0 < \etaa \leq 1$ and $f \in L^q(\Omega)$ with $q$ chosen as above. Then Problem~\eqref{eq:pb} has a unique solution $u \in H^1(\Omega)$. In addition, if $\etaa$ is sufficiently small, then $u \in W^{2,q}(\Omega)$ and it satisfies
\begin{equation}
\label{eq:regul}
\left\| u - \sigma_1 \fint_\Omega u \right\|_{W^{2,q}(\Omega)}
\leq
C \| f \|_{L^q(\Omega)}
\end{equation}
for some $C$ independent of $\etaa$ and $f$, where $\sigma_1$ is the invariant measure defined by~\eqref{invariant_measure_droniou}--\eqref{invariant_measure_droniou-int-1}--\eqref{invariant_measure_droniou-pos}.
\end{lemma}
We only prove this result in dimension $2 \leq d \leq 3$ (see the assumptions of Proposition~\ref{prop:BS_mu}), but it certainly holds for larger dimensions.

\begin{proof}
The proof falls in two steps.

\medskip

\noindent
\textbf{Step 1: existence and uniqueness of a solution.} 
We first show that the bilinear form $\abil_\etaa$ satisfies the~$(\textnormal{BNB}1)$ condition on $H^1(\Omega)$. Using the invariant measure $\sigma_1$ defined by~\eqref{invariant_measure_droniou}--\eqref{invariant_measure_droniou-int-1}--\eqref{invariant_measure_droniou-pos}, a simple computation indeed yields that, for any $v \in H^1(\Omega)$,
\begin{align}
a_\etaa(v,\sigma_1 v) 
&= 
\int_\Omega \sigma_1 |\nabla v|^2 + \frac{1}{2} \int_\Omega (\nabla \sigma_1 + b \sigma_1) \cdot \nabla (v^2) + \etaa \int_\Omega \sigma_1 \, v^2
\nonumber
\\
&=\int_\Omega \sigma_1 |\nabla v|^2 + \etaa \int_\Omega \sigma_1 \, v^2
\label{eq:these3}
\\
&\geq (\inf \sigma_1) \min(1,\etaa) \, \| v \|_{H^1(\Omega)}^2
\nonumber
\\
&\geq c \min(1,\etaa) \, \| v \|_{H^1(\Omega)} \| \sigma_1 v \|_{H^1(\Omega)},
\label{eq:these}
\end{align}
for some $c>0$ independent of $\etaa$. For any $w \in H^1(\Omega)$, we set $v = \sigma_1^{-1} w$, which belongs to $H^1(\Omega)$, and thus have
\begin{multline}
\abil_\etaa(w,\sigma_1^{-1} w)
=
a_\etaa(\sigma_1^{-1} w,w)
\geq
c \min(1,\etaa) \, \| v \|_{H^1(\Omega)} \| \sigma_1 v \|_{H^1(\Omega)}
\\
=
c \min(1,\etaa) \, \| w \|_{H^1(\Omega)} \| \sigma_1^{-1} w \|_{H^1(\Omega)}.
\label{eq:these6}
\end{multline}
We thus deduce the~$(\textnormal{BNB}1)$ condition.

The bilinear form $\abil_\etaa$ also satisfies the~$(\textnormal{BNB}2)$ condition on $H^1(\Omega)$. Indeed, if $v \in H^1(\Omega)$ is such that $\abil_\etaa(u,v)=0$ for any $u \in H^1(\Omega)$, then we have $\abil_\etaa(\sigma_1 v,v)=0$, and we have thus found a function $w \in H^1(\Omega)$ (namely $w = \sigma_1 v$) such that $\abil_\etaa(w,\sigma_1^{-1} w)=0$. The estimate~\eqref{eq:these6} shows that $w$ (and hence $v$) vanishes, which implies the~$(\textnormal{BNB}2)$ condition on $H^1(\Omega)$.

We have chosen $f \in L^q(\Omega)$ with an exponent $q$ such that $f \in (H^1(\Omega))'$. We thus obtain that Problem~\eqref{eq:pb} is well-posed.

\medskip

\noindent
\textbf{Step 2: $W^{2,q}$ estimate.} Introduce $\dis \overline{u} = u + \sigma_1 - \sigma_1 \fint_\Omega u$, which satisfies
$$
\left\{
\begin{aligned}
&-\textnormal{div }(\nabla \overline{u} + b \overline{u}) = F \ \ \text{in $\Omega$}, \qquad \fint_\Omega \overline{u} =1, 
\\
&(\nabla \overline{u} + b \overline{u})\cdot n = 0 \ \ \text{on $\partial\Omega$},
\end{aligned}
\right.
$$
with
$$
F = f - \etaa u = f - \etaa \overline{u} + \etaa \sigma_1 - \etaa \sigma_1 \fint_\Omega u.
$$
Using $v \equiv 1$ as test function in~\eqref{eq:pb}, we see that $\dis \etaa \fint_\Omega u = \fint_\Omega f$. We hence get that
$$
F = f - \etaa (\overline{u} - \sigma_1) - \sigma_1 \fint_\Omega f.
$$
Using that $\sigma_1 \in W^{2,s}(\Omega)$ for any $1 < s < \infty$, we deduce that, for any $1 < s \leq q$, 
\begin{equation}
\label{eq:these7}
\| F \|_{L^s(\Omega)}
\leq
C \|f\|_{L^s(\Omega)} + \etaa \| \overline{u} - \sigma_1 \|_{L^s(\Omega)},
\end{equation}
where we only know, at this stage, that $\overline{u} \in H^1(\Omega)$. To proceed and obtain a $W^{2,q}$ estimate, we distinguish two cases, whether $d=2$ or $d=3$.

\medskip

Suppose first that $d=2$. Using the continuous injection $H^1(\Omega) \subset L^q(\Omega)$, we deduce from~\eqref{eq:these7} (written with $s=q$) that $F \in L^q(\Omega)$. We are thus in position to apply Proposition~\ref{prop_u_f_A_id}, which yields (see~\eqref{estimate_v_g_A_id}) that there exists some $C$ independent of $\etaa$ such that
$$
\left\| \overline{u} - \sigma_1 \right\|_{W^{2,q}(\Omega)}
\leq
C \| F \|_{L^q(\Omega)}
\leq
C \|f\|_{L^q(\Omega)} + C \etaa \| \overline{u} - \sigma_1 \|_{L^q(\Omega)}.
$$
For $\etaa$ sufficiently small, this implies~\eqref{eq:regul}. 

\medskip

Suppose now that $d=3$. If $q \leq 6$, we proceed as above, using the continuous injection $H^1(\Omega) \subset L^q(\Omega)$. We now turn to the case $q>6$. Using~\eqref{eq:these7} for $s=6$, we deduce that $F \in L^6(\Omega)$. Applying Proposition~\ref{prop_u_f_A_id}, we obtain that
$$
\left\| \overline{u} - \sigma_1 \right\|_{W^{2,6}(\Omega)}
\leq
C \|f\|_{L^q(\Omega)} + C \etaa \| \overline{u} - \sigma_1 \|_{L^6(\Omega)},
$$
which implies that $\left\| \overline{u} - \sigma_1 \right\|_{W^{2,6}(\Omega)} \leq C \|f\|_{L^q(\Omega)}$. Using the continuous injection $W^{2,6}(\Omega) \subset L^\infty(\Omega)$, we deduce that  $\left\| \overline{u} - \sigma_1 \right\|_{L^\infty(\Omega)} \leq C \|f\|_{L^q(\Omega)}$. The estimate~\eqref{eq:these7}, written with $s=q$, now yields 
$$
\| F \|_{L^q(\Omega)}
\leq
C \|f\|_{L^q(\Omega)} + \etaa \| \overline{u} - \sigma_1 \|_{L^q(\Omega)}
\leq
C \|f\|_{L^q(\Omega)},
$$
from which, applying again Proposition~\ref{prop_u_f_A_id}, we infer~\eqref{eq:regul}.
\end{proof}

\medskip

Likewise, for any $f \in L^2(\Omega)$, we consider the problem
\begin{equation}
\label{eq:pb_adjoint}
\text{Find $u \in H^1(\Omega)$ s.t., for any $v \in H^1(\Omega)$, \quad $\abil_\etaa(v,u) = a_\etaa(u,v) = \int_\Omega f v$},
\end{equation}
for which we have the following result.

\begin{lemma}
\label{lem:these_adjoint}
We work under the assumptions of Proposition~\ref{prop:BS_mu} and assume that $0 < \etaa \leq 1$ and $f \in L^2(\Omega)$. Then Problem~\eqref{eq:pb_adjoint} has a unique solution $u \in H^1(\Omega)$. In addition, $u \in H^2(\Omega)$ and it satisfies
\begin{equation}
\label{eq:regul_adjoint}
\left\| u - \fint_\Omega u \right\|_{H^2(\Omega)}
\leq
C \| f \|_{L^2(\Omega)}
\end{equation}
for some $C$ independent of $\etaa$ and $f$.
\end{lemma}

A similar result certainly holds for $f \in L^q(\Omega)$ with $q$ chosen such that $f \in (H^1(\Omega))'$, yielding a control on $\dis \left\| u - \fint_\Omega u \right\|_{W^{2,q}(\Omega)}$. We will however not need such a result and therefore do not pursue in that direction. As for Lemma~\ref{lem:these}, we only prove Lemma~\ref{lem:these_adjoint} in dimension $2 \leq d \leq 3$ (see the assumptions of Proposition~\ref{prop:BS_mu}), but it certainly holds for larger dimensions.

\begin{proof}
The proof falls in three steps.

\medskip

\noindent
\textbf{Step 1: existence and uniqueness of a solution.} 
The bilinear form $a_\etaa$ satisfies the~$(\textnormal{BNB}1)$ condition on $H^1(\Omega)$, as a direct consequence of~\eqref{eq:these}. It also satisfies the~$(\textnormal{BNB}2)$ condition on $H^1(\Omega)$. Indeed, if $v \in H^1(\Omega)$ is such that $a_\etaa(u,v)=0$ for any $u \in H^1(\Omega)$, then we have $a_\etaa(\sigma_1^{-1} v,v)=0$, and we have thus found a function $w \in H^1(\Omega)$ (namely $w = \sigma_1^{-1} v$) such that $a_\etaa(w,\sigma_1 w)=0$. The estimate~\eqref{eq:these} shows that $w$ (and hence $v$) vanishes, which implies the~$(\textnormal{BNB}2)$ condition on $H^1(\Omega)$. We thus obtain that Problem~\eqref{eq:pb_adjoint} is well-posed.

\medskip

\noindent
\textbf{Step 2: $H^1$ estimate.} 
We claim that the solution $u$ to~\eqref{eq:pb_adjoint} satisfies
\begin{equation}
\label{eq:these2}
\left\| u - \fint_\Omega u \right\|_{H^1(\Omega)}
\leq
C \| f \|_{L^2(\Omega)}
\end{equation}
for some $C$ independent of $\etaa$ and $f$. Consider indeed $\dis \overline{u} = u - \fint_\Omega u$. Using the Poincar\'e-Wirtinger inequality and~\eqref{eq:these3} for the function $\overline{u}$, we have
\begin{align*}
\left\| \overline{u} \right\|_{H^1(\Omega)}^2
& \leq
C \left\| \nabla \overline{u} \right\|_{L^2(\Omega)}^2
\\
& \leq
C a_\etaa(\overline{u},\sigma_1 \overline{u})
\\
& =
C a_\etaa(u,\sigma_1 \overline{u}) - C a_\etaa(1,\sigma_1 \overline{u}) \fint_\Omega u
\\
& =
C \int_\Omega f \, \sigma_1 \overline{u} - C \etaa \int_\Omega \sigma_1 \overline{u} \fint_\Omega u
\\
& \leq
C \| \sigma_1 \|_{L^\infty(\Omega)} \| f \|_{L^2(\Omega)} \| \overline{u} \|_{H^1(\Omega)} + C \etaa \| \sigma_1 \|_{L^\infty(\Omega)} \| \overline{u} \|_{H^1(\Omega)} \| u \|_{L^2(\Omega)}.
\end{align*}
We hence deduce that
\begin{equation}
\label{eq:these4}
\left\| \overline{u} \right\|_{H^1(\Omega)}
\leq
C \left( \| f \|_{L^2(\Omega)} + \etaa \| u \|_{L^2(\Omega)} \right).
\end{equation}
We now write~\eqref{eq:these3} for the function $u$, from which we deduce that
$$
\| \nabla u \|^2_{L^2(\Omega)} + \etaa \| u \|^2_{L^2(\Omega)}
\leq
C a_\etaa(u,\sigma_1 u)
=
C \int_\Omega f \, \sigma_1 u
\leq
C \| f \|_{L^2(\Omega)} \| u \|_{L^2(\Omega)},
$$
which implies that
\begin{equation}
\label{eq:these5}
\etaa \| u \|_{L^2(\Omega)} \leq C \| f \|_{L^2(\Omega)}.
\end{equation}
Inserting this estimate in~\eqref{eq:these4}, we obtain the claimed estimate~\eqref{eq:these2}.

\medskip

\noindent
\textbf{Step 3: $H^2$ estimate.} We proceed as in the proof of Proposition~\ref{prop_u_f_A_id}. Introducing again $\dis \overline{u} = u - \fint_\Omega u$, we observe that
$$
\left\{
\begin{aligned}
&-\Delta \overline{u} = F \ \ \text{in $\Omega$}, \qquad \fint_\Omega \overline{u} = 0, 
\\ 
&\nabla \overline{u} \cdot n = 0 \ \ \text{on $\partial\Omega$},
\end{aligned}
\right.
$$
with $\dis F = f - b \nabla \overline{u} - \etaa \overline{u} - \etaa \fint_\Omega u$. We compute that
$$
\| F \|_{L^2(\Omega)} \leq \| f \|_{L^2(\Omega)} + \| b \|_{L^\infty(\Omega)} \| \overline{u} \|_{H^1(\Omega)} + \etaa \| \overline{u} \|_{L^2(\Omega)} + \etaa \| u \|_{L^2(\Omega)}
$$
and we deduce, using~\eqref{eq:these2} and~\eqref{eq:these5}, that $\| F \|_{L^2(\Omega)} \leq C \| f \|_{L^2(\Omega)}$. We are then in position to use~\cite[Theorem 3.12]{EG}, which implies that $\overline{u} \in H^2(\Omega)$ with
$$
\| \overline{u} \|_{H^2(\Omega)} \leq C \| F \|_{L^2(\Omega)} \leq C \| f \|_{L^2(\Omega)}.
$$
This concludes the proof of~\eqref{eq:regul_adjoint}.
\end{proof}

\subsection{Discretized problems}

We now consider the discretization of Problems~\eqref{eq:pb} and~\eqref{eq:pb_adjoint}. Let $f \in L^2(\Omega)$ and let $\Sigma_h \subset H^1(\Omega)$ be the $\mathbb{P}^1$ approximation space associated to a regular quasi-uniform polyhedral mesh of $\Omega$.

\begin{theorem}[Discretization of Problem~\eqref{eq:pb}]
\label{th:inv_measure_mu_h}
We assume that~\eqref{hyp_H0_debut}--\eqref{hyp_H_11_num}--\eqref{hyp_H_22_num} hold, and that $0 < \etaa \leq 1$. 
Then, there exists $h_0$ independent of $\etaa$ such that, for sufficiently small $\etaa$ and any $h \leq h_0$, there exists a unique $u_h \in \Sigma_h$ solution to
\begin{equation}
\label{eq:carnon1_mu_h}
\forall v_h \in \Sigma_h, \qquad \abil_\etaa(u_h,v_h) = \int_\Omega f v_h.
\end{equation}
\end{theorem}

\begin{proof}
The proof follows the lines of that of Theorem~\ref{th:main_debut}. Let $\lambda > 0$. We consider the iterations
\begin{equation}
\label{eq:carnon2_mu_h}
\left\{
\begin{array}{c}
\text{Find $u_h^{n+1} \in \Sigma_h$ such that, for all $v_h \in \Sigma_h$},
\\ \noalign{\vskip 3pt}
\dis \abil(u_h^{n+1},v_h) + \lambda \int_\Omega u_h^{n+1} \, v_h 
=
\abil_\lambda(u_h^{n+1},v_h) 
= 
(\lambda - \etaa) \int_\Omega u_h^n \, v_h + \int_\Omega f v_h,
\end{array}
\right.
\hspace{-13mm}
\end{equation}
with the initial condition $\dis u_h^0 = \etaa^{-1} \fint_\Omega f$ (or any function in $\Sigma_h$ such that $\dis \etaa \fint_\Omega u_h^0 = \fint_\Omega f$). In view of~\eqref{eq:these8}, these problems are well-posed for any $h \leq h_0$, where $h_0$ is independent of $\etaa$. They thus define a sequence $u_h^n \in \Sigma_h \subset H^1(\Omega)$. Furthermore, taking $v_h=1$ as test function, we see that
$$
\lambda \int_\Omega u_h^{n+1}
= 
(\lambda - \etaa) \int_\Omega u_h^n + \int_\Omega f.
$$
Our choice of $u_h^0$ implies that all the functions $u_h^n$ share the same mean.

We next prove that the sequence $\left\{ u_h^n \right\}_{n \in \N}$ converges to a solution to~\eqref{eq:carnon1_mu_h}. We recall that $\dis H^1_{\int=0}(\Omega) = \left\{ v\in H^1(\Omega), \quad \int_\Omega v=0 \right\}$. We infer from~\eqref{eq:carnon2_mu_h} that, for any $v_h \in \Sigma_h \cap H^1_{\int=0}(\Omega)$,
$$
\abil(u_h^{n+1} - u_h^n,v_h) + \lambda \int_\Omega (u_h^{n+1}-u_h^n) \, v_h 
=
(\lambda - \etaa) \int_\Omega (u_h^n - u_h^{n-1}) \, v_h, 
$$
from which we deduce that
\begin{equation}
\label{eq:titit_mu_h}
\sup_{v_h \in \Sigma_h \cap H^1_{\int=0}(\Omega) } \frac{\abil(u_h^{n+1} - u_h^n,v_h)}{\| v_h \|_{H^1(\Omega)}} - \lambda \| u_h^{n+1} - u_h^n \|_{H^1(\Omega)} 
\leq
(\lambda-\etaa) \| u_h^n - u_h^{n-1} \|_{H^1(\Omega)}.
\end{equation}
Using~\eqref{eq:carnon3}, we infer from~\eqref{eq:titit_mu_h} that
$$
(\gamma - \lambda) \| u_h^{n+1} - u_h^n \|_{H^1(\Omega)} 
\leq
(\lambda-\etaa) \| u_h^n - u_h^{n-1} \|_{H^1(\Omega)}.
$$
Taking $\lambda$ sufficiently small (so that $0 < \lambda/(\gamma - \lambda) < 1$) and $\etaa < \lambda$, we obtain that the sequence $\left\{ u_h^n \right\}_{n \in \N}$ converges in $H^1(\Omega)$ to some $u_h^\infty \in \Sigma_h$. Passing to the limit $n \to \infty$ in~\eqref{eq:carnon2_mu_h}, we get that $u_h^\infty$ is a solution to~\eqref{eq:carnon1_mu_h}.

\medskip

We now prove that~\eqref{eq:carnon1_mu_h} has a unique solution. Consider two solutions $u_h$ and $\overline{u}_h$ to~\eqref{eq:carnon1_mu_h}. Then $u_h - \overline{u}_h \in \Sigma_h \cap H^1_{\int=0}(\Omega)$ and satisfies $\dis \abil(u_h - \overline{u}_h,v_h) = - \etaa \int_\Omega (u_h - \overline{u}_h) \, v_h$ for any $v_h \in \Sigma_h$. We deduce from~\eqref{eq:carnon3} that
$$
\gamma \| u_h - \overline{u}_h \|_{H^1(\Omega)} \leq \etaa \| u_h - \overline{u}_h \|_{H^1(\Omega)},
$$
which implies, whenever $\etaa < \gamma$, that $u_h = \overline{u}_h$. 
\end{proof}

\begin{theorem}[Discretization of Problem~\eqref{eq:pb_adjoint}]
\label{th:adjoint_mu_h}
We assume that~\eqref{hyp_H0_debut}--\eqref{hyp_H_11_num}--\eqref{hyp_H_22_num} hold, and that $0 < \etaa \leq 1$. 
Then, there exists $h_0$ independent of $\etaa$ such that, for sufficiently small $\etaa$ and any $h \leq h_0$, there exists a unique $u_h \in \Sigma_h$ solution to
\begin{equation}
\label{eq:carnon1_adjoint_mu_h}
\forall v_h \in \Sigma_h, \qquad a_\etaa(u_h,v_h) = \int_\Omega f v_h.
\end{equation}
\end{theorem}

\begin{proof}
The proof follows the lines of that of Theorem~\ref{th:inv_measure_mu_h}. We consider the iterations
\begin{equation}
\label{eq:carnon2_mumu_h}
\left\{
\begin{array}{c}
\text{Find $u_h^{n+1} \in \Sigma_h$ such that, for all $v_h \in \Sigma_h$},
\\ \noalign{\vskip 3pt}
\dis a(u_h^{n+1},v_h) + \lambda \int_\Omega u_h^{n+1} \, v_h 
=
a_\lambda(u_h^{n+1},v_h) 
= 
(\lambda - \etaa) \int_\Omega u_h^n \, v_h + \int_\Omega f v_h,
\end{array}
\right.
\hspace{-13mm}
\end{equation}
with an initial condition $u_h^0$ such that $\dis \etaa \fint_\Omega u_h^0 \, \sigma_{1,h} = \fint_\Omega f \, \sigma_{1,h}$ (where $\sigma_{1,h}$ satisfies~\eqref{eq:sigma_h_bien_pose} with $g \equiv 0$). We have shown in Proposition~\ref{prop_brenner_scott_estimate_debut} that $\| \sigma_{1,h} - \sigma_1 \|_{H^1(\Omega)} \leq C h$, and we have shown in Lemma~\ref{theorem_perthame} that $\sigma_1$ is positive and bounded away from 0. We hence have $\dis \fint_\Omega \sigma_{1,h} > 0$ when $h$ is sufficiently small, and it is thus possible to pick $u_h^0$ as a constant function.

The problems~\eqref{eq:carnon2_mumu_h} are well-posed for any $h \leq h_0$, where $h_0$ is independent of $\etaa$. Consider indeed a basis $(\varphi_i)_{1 \leq i \leq I}$ of $\Sigma_h$. Since $\abil_\lambda$ satisfies the inf-sup condition~\eqref{eq:these8} on $\Sigma_h$ as soon as $h \leq h_0$, we know that the matrix $K$, defined by $K_{ij} = \abil_\lambda(\varphi_j,\varphi_i)$ for any $1 \leq i,j \leq I$, is invertible. The matrix $K^T$ is therefore invertible. This implies that~\eqref{eq:carnon2_mumu_h} is indeed well-posed for any $h \leq h_0$, and thus defines a sequence $u_h^n \in \Sigma_h \subset H^1(\Omega)$. Furthermore, taking $v_h=\sigma_{1,h}$ as test function, we see that
$$
\lambda \int_\Omega u_h^{n+1} \, \sigma_{1,h}
= 
(\lambda - \etaa) \int_\Omega u_h^n \, \sigma_{1,h} + \int_\Omega f \, \sigma_{1,h}.
$$
Our choice of $u_h^0$ implies that, for any $n$, we have
\begin{equation}
\label{eq:maison}
\etaa \fint_\Omega u_h^n \, \sigma_{1,h} = \fint_\Omega f \, \sigma_{1,h}.
\end{equation}

Let $\dis \overline{u}_h^n = u_h^n - \fint_\Omega u_h^n$. We infer from~\eqref{eq:carnon2_mumu_h} that, for any $v_h \in \Sigma_h$,
$$
a(\overline{u}_h^{n+1},v_h) + \lambda \int_\Omega u_h^{n+1} \, v_h 
=
(\lambda - \etaa) \int_\Omega u_h^n \, v_h + \int_\Omega f v_h.
$$
We recall that $\dis H^1_{\int=0}(\Omega) = \left\{ v\in H^1(\Omega), \quad \int_\Omega v=0 \right\}$. Taking now $v_h \in \Sigma_h \cap H^1_{\int=0}(\Omega)$, we get
$$
a(\overline{u}_h^{n+1},v_h) + \lambda \int_\Omega \overline{u}_h^{n+1} \, v_h 
=
(\lambda - \etaa) \int_\Omega \overline{u}_h^n \, v_h + \int_\Omega f v_h,
$$
hence
$$
a(\overline{u}_h^{n+1} - \overline{u}_h^n,v_h) + \lambda \int_\Omega (\overline{u}_h^{n+1}-\overline{u}_h^n) \, v_h 
=
(\lambda - \etaa) \int_\Omega (\overline{u}_h^n - \overline{u}_h^{n-1}) \, v_h, 
$$
from which we deduce that
\begin{equation}
\label{eq:titit_mumu_h}
\sup_{v_h \in \Sigma_h \cap H^1_{\int=0}(\Omega) } \frac{a(\overline{u}_h^{n+1} - \overline{u}_h^n,v_h)}{\| v_h \|_{H^1(\Omega)}} - \lambda \| \overline{u}_h^{n+1} - \overline{u}_h^n \|_{H^1(\Omega)} 
\leq
(\lambda-\etaa) \| \overline{u}_h^n - \overline{u}_h^{n-1} \|_{H^1(\Omega)}.
\end{equation}
The bilinear form $\abil$ satisfies an infsup condition on $\Sigma_h \cap H^1_{\int=0}(\Omega)$ for $h$ sufficiently small (see~\eqref{eq:carnon3}). Considering a basis $(\varphi_i)_{1 \leq i \leq I}$ of $\Sigma_h \cap H^1_{\int=0}(\Omega)$, we get that the matrix $K$ defined by $K_{ij} = \abil(\varphi_j,\varphi_i)$ for any $1 \leq i,j \leq I$, is invertible. The matrix $K^T$ is therefore invertible, which implies that the bilinear form $a$ also satisfies an infsup condition on $\Sigma_h \cap H^1_{\int=0}(\Omega)$ (for $h$ sufficiently small):
\begin{equation}
\label{eq:carnon3_adjoint}
\inf_{w_h \in \Sigma_h \cap H^1_{\int=0}(\Omega)} \ \sup_{v_h \in \Sigma_h \cap H^1_{\int=0}(\Omega)} \ \frac{a(w_h,v_h)}{\| w_h \|_{H^1(\Omega)} \ \| v_h \|_{H^1(\Omega)}} \geq \gamma_a.
\end{equation}
We then infer from~\eqref{eq:titit_mumu_h} that
$$
(\gamma_a - \lambda) \| \overline{u}_h^{n+1} - \overline{u}_h^n \|_{H^1(\Omega)} 
\leq
(\lambda-\etaa) \| \overline{u}_h^n - \overline{u}_h^{n-1} \|_{H^1(\Omega)}.
$$
Taking $\lambda$ sufficiently small (so that $0 < \lambda/(\gamma_a - \lambda) < 1$) and $\etaa < \lambda$, we obtain that the sequence $\left\{ \overline{u}_h^n \right\}_{n \in \N}$ converges in $H^1(\Omega)$ to some $\overline{u}_h^\infty \in \Sigma_h$. We also deduce from~\eqref{eq:maison} that
$$
\etaa \fint_\Omega \overline{u}_h^n \, \sigma_{1,h} + \etaa \fint_\Omega u_h^n \fint_\Omega \sigma_{1,h} = \fint_\Omega f \, \sigma_{1,h}.
$$
Since $\dis \fint_\Omega \sigma_{1,h} \neq 0$, we obtain that $\dis \fint_\Omega u_h^n$ converges to some $\ell$ satisfying
$$
\etaa \fint_\Omega \overline{u}_h^\infty \, \sigma_{1,h} + \etaa \, \ell \fint_\Omega \sigma_{1,h} = \fint_\Omega f \, \sigma_{1,h}.
$$
We thus get that the sequence $\left\{ u_h^n \right\}_{n \in \N}$ converges in $H^1(\Omega)$ to $u_h^\infty := \ell + \overline{u}_h^\infty \in \Sigma_h$.

Passing to the limit $n \to \infty$ in~\eqref{eq:carnon2_mumu_h}, we get that $u_h^\infty$ is a solution to~\eqref{eq:carnon1_adjoint_mu_h}.

\medskip

We now prove that~\eqref{eq:carnon1_adjoint_mu_h} has a unique solution. Consider two solutions $u_{1,h}$ and $u_{2,h}$ to~\eqref{eq:carnon1_adjoint_mu_h}. Introduce $\dis \overline{u}_{1,h} = u_{1,h} - \fint_\Omega u_{1,h}$ and likewise for $u_{2,h}$. Then $\overline{u}_{1,h} - \overline{u}_{2,h} \in \Sigma_h \cap H^1_{\int=0}(\Omega)$ and satisfies $\dis a(\overline{u}_{1,h} - \overline{u}_{2,h},v_h) = - \etaa \int_\Omega (\overline{u}_{1,h} - \overline{u}_{2,h}) \, v_h$ for any $v_h \in \Sigma_h \cap H^1_{\int=0}(\Omega)$. We deduce from~\eqref{eq:carnon3_adjoint} that
$$
\gamma_a \| \overline{u}_{1,h} - \overline{u}_{2,h} \|_{H^1(\Omega)} \leq \etaa \| \overline{u}_{1,h} - \overline{u}_{2,h} \|_{H^1(\Omega)},
$$
which implies, whenever $\etaa < \gamma_a$, that $\overline{u}_{1,h} = \overline{u}_{2,h}$. The functions $u_{1,h}$ and $u_{2,h}$ are thus equal up to the addition of a constant. Taking $v_h=\sigma_{1,h}$ as test function in~\eqref{eq:carnon1_adjoint_mu_h}, we see that
$$
\etaa \int_\Omega u_{1,h} \, \sigma_{1,h}
= 
\int_\Omega f \, \sigma_{1,h}
$$
and likewise for $u_{2,h}$, which implies that $u_{1,h}=u_{2,h}$. 
\end{proof}

\subsection{Weight function}

For some $\kappa \geq 1$, we set, for any $x$ and $z$ in $\Omega$,
\begin{equation}
\label{eq:def_sigma}
\weight_z(x) = \sqrt{|x-z|^2 + (\kappa h)^2}.
\end{equation}
It is easy to show that there exists $C$ (independent of $\kappa$ and $h$) such that
$$
\forall T \in \mathcal{T}_h, \quad \forall z \in \Omega, \quad \frac{\sup_{x \in T} \weight_z(x)}{\inf_{x \in T} \weight_z(x)} \leq C.
$$
Likewise, for any $\beta \in \N^d$ and $\Lambda \in \R$, there exists $C$ (independent of $\kappa$ and $h$) such that, for any $x$ and $z$ in $\Omega$,
\begin{equation}
\label{eq:bound_sigma2}
\left| \partial_\beta (\weight_z^\Lambda)(x) \right| \leq C \weight_z^{\Lambda - |\beta|}(x).
\end{equation}
The following estimate will be useful:
\begin{lemma}
For any real number $\theta > d$, we have, for any $x \in \Omega$,
\begin{equation}
\int_\Omega \weight_z^{-\theta}(x) \, dz
\leq 
C_d \frac{1}{(\kappa h)^{\theta-d}} \left( \frac{1}{d} + \frac{1}{\theta-d} \right),
\label{eq:util}
\end{equation}
where $C_d$ only depends on the dimension $d$.
\end{lemma}

\begin{proof}
Let $R$ be the diameter of $\Omega$, so that, for any $x \in \Omega$, we have $\Omega \subset B_R(x)$. We write
$$
\int_\Omega \weight_z^{-\theta}(x) dz
\leq
\int_{B_R(x)} \weight_z^{-\theta}(x) dz
=
C_d \int_0^R \frac{r^{d-1}}{(r^2 + (\kappa h)^2)^{\theta/2}} dr.
$$
We split the integral from $r=0$ to $r=\kappa h$ (for which we write that $r^2 + (\kappa h)^2 \geq (\kappa h)^2$) and from $r=\kappa h$ to $r=R$ (for which we write that $r^2 + (\kappa h)^2 \geq r^2$). A straightforward computation then leads to~\eqref{eq:util}.
\end{proof}

We recall some useful properties of $\Sigma_h$, the subset of $H^1(\Omega)$ of piecewise affine functions. First, for any $\Lambda \in \R$, there exists $C$ such that, for any $\psi \in H^1(\Omega)$ such that $\psi_{|T} \in H^2(T)$ for any $T \in \mathcal{T}_h$, there exists $I_h \psi \in \Sigma_h$ such that 
\begin{equation}
\label{eq:inter1}
\int_\Omega \weight_z^\Lambda (\psi - I_h \psi)^2 + h^2 \int_\Omega \weight_z^\Lambda |\nabla (\psi - I_h \psi)|^2 \leq C h^4 \sum_{T \in \mathcal{T}_h} \int_T \weight_z^\Lambda |\nabla^2 \psi|^2
\end{equation}
where $C$ is independent of $z$, $\kappa$, $h$ and $\psi$. Second, we have, for any $\Lambda \in \R$ and any $\psi_h \in \Sigma_h$, that
\begin{equation}
\label{eq:inter3}
\int_\Omega \weight_z^\Lambda |\nabla \psi_h|^2
\leq 
C h^{-2} \int_\Omega \weight_z^\Lambda \, \psi_h^2.
\end{equation}

\subsection{Numerical Green functions}

For any element $T \in \mathcal{T}_h$ of the mesh, we introduce a function $\overline{\delta}_T \in C^\infty_0(T)$ such that $\overline{\delta}_T \geq 0$ on $T$ and $\dis \int_T \overline{\delta}_T = 1$.

Let $z \in \Omega$ such that $z$ does not lie on an edge of the mesh. We call $K^z$ the element containing $z$, and set $\delta^z = \overline{\delta}_{K^z}$. For any function $P$ which is piecewise constant on $\mathcal{T}_h$, we thus have
\begin{equation}
\label{eq:dirac}
\int_\Omega \delta^z P
=
\int_{K^z} \delta^z P
=
P(z) \int_{K^z} \delta^z
=
P(z).
\end{equation}
It is possible to build $\delta^z$ such that it satisfies the following bounds:
$$
\forall k \in \N, \quad \| \nabla^k \delta^z \|_{L^\infty(\Omega)} \leq \frac{C_k}{h^{d+k}}.
$$
Note that $\delta^z$ depends on $h$ which is the diameter of $K^z$. 

\medskip

Let $\nu \in \R^d$ be a constant vector. Since $\delta^z \in C^\infty_0(K^z)$, we have that $\nu \cdot \nabla \delta^z \in L^2(\Omega)$. Problem~\eqref{eq:pb_adjoint} is thus well posed for the right-hand side $\nu \cdot \nabla \delta^z$ (see Lemma~\ref{lem:these_adjoint}). We hence define $g^z \in H^1(\Omega)$ such that
\begin{equation}
\label{eq:def_gz}
\forall v \in H^1(\Omega), \quad \abil_\etaa(v,g^z) = - \int_\Omega (\nu \cdot \nabla \delta^z) \ v.
\end{equation}
Likewise, we introduce $g^z_h \in \Sigma_h$ such that
\begin{equation}
\label{eq:def_gzh}
\forall v \in \Sigma_h, \quad \abil_\etaa(v,g^z_h) = - \int_\Omega (\nu \cdot \nabla \delta^z) \ v.
\end{equation}
In view of Theorem~\ref{th:adjoint_mu_h}, we know that there exists $h_0$ independent of $\etaa$ such that, for any $h < h_0$, the above problem is well-posed. 

\medskip

For any $\lambda > 0$, we define
\begin{equation}
\label{eq:def_M}
M_{h,\lambda} = \sup_{z \in \Omega, \ \text{$z$ not on edges}} \quad \sqrt{ \int_\Omega \weight_z^{d+\lambda} \Big( |g^z - g^z_h|^2 + |\nabla (g^z - g^z_h) |^2 \Big) }
\end{equation}
with $\weight_z$ defined by~\eqref{eq:def_sigma}. The following lemma will be most useful:

\begin{lemma}
\label{lem:bound_M}
We work under the assumptions of Proposition~\ref{prop:BS_mu}. Then there exists $h_0>0$, $\lambda > 0$, $\kappa \geq 1$ (possibly depending on $\etaa$) and $C_{\kappa,\lambda,\etaa}$ (possibly depending on $\kappa$, $\lambda$ and $\etaa$) such that, for any $h$ such that $0 < h \leq h_0$ and $\kappa h \leq 1$, we have
$$
M_{h,\lambda}^2 \leq C_{\kappa,\lambda,\etaa} \, h^\lambda.
$$
\end{lemma}
The proof of this lemma is postponed until Appendix~\ref{sec:proof_lem4}. The restriction $h \leq h_0$ comes from the fact that the existence of $g^z_h$ is only ensured for sufficiently small $h$. A careful inspection of the proof shows that one can take $\kappa = C/\etaa$ and $\dis C_{\kappa,\lambda,\etaa} = C \, \kappa^{d+\lambda} \, \etaa^{-2}$ for some $C$ independent of $\kappa$, $\etaa$ and $\lambda$.

\medskip

We proceed in the sequel of this Appendix~\ref{sec:appendix_BS} with the proof of Proposition~\ref{prop:BS_mu}.

\subsection{Proof of~\eqref{eq:main1}}

Let $u$ and $u_h$ (resp. in $H^1(\Omega)$ and $\Sigma_h$) satisfying the assumptions of Proposition~\ref{prop:BS_mu}. We write, for any fixed $z$ not lying on the mesh edges, that
\begin{eqnarray}
&& \nu \cdot \nabla u_h(z)
\nonumber
\\
&=& \int_\Omega \delta^z \ (\nu \cdot \nabla u_h) \qquad \text{[eq.~\eqref{eq:dirac}]}
\nonumber
\\
&=&
- \int_\Omega (\nu \cdot \nabla \delta^z) u_h \qquad \text{[int. by part and $\delta^z = 0$ on $\partial \Omega$]}
\nonumber
\\
&=&
\abil_\etaa(u_h,g^z_h) \qquad \text{[def.~\eqref{eq:def_gzh} of $g^z_h$ and $u_h \in \Sigma_h$]}
\nonumber
\\
&=&
\abil_\etaa(u,g^z_h) \qquad \text{[Assumption~\eqref{eq:galerkin_orth} and $g^z_h \in \Sigma_h$]}
\nonumber
\\
&=&
\abil_\etaa(u,g^z) + \abil_\etaa(u,g^z_h-g^z)
\nonumber
\\
&=&
- \int_\Omega (\nu \cdot \nabla \delta^z) u + \abil_\etaa(u,g^z_h-g^z) \qquad \text{[def.~\eqref{eq:def_gz} of $g^z$ and $u \in H^1(\Omega)$]}
\nonumber
\\
&=&
\int_\Omega \delta^z \, (\nu \cdot \nabla u)  + \abil_\etaa(u,g^z_h-g^z). \quad \text{[int. by part and $\delta^z = 0$ on $\partial \Omega$]}
\label{eq:split}
\end{eqnarray}
In Sections~\ref{sec:22} and~\ref{sec:22bis} below, we successively bound the two terms of the right-hand side of~\eqref{eq:split}. In Section~\ref{sec:22ter}, we conclude the proof of~\eqref{eq:main1}.

\subsubsection{Bound on the second term of the right hand side of~\eqref{eq:split}}
\label{sec:22}

In view of the assumptions of Proposition~\ref{prop:BS_mu}, we know that $u \in W^{1,p}(\Omega)$ for some $p \geq 2$. Let $1 < q \leq 2$ such that 
$$
1 = \frac{1}{p} + \frac{1}{q}.
$$
We know that $g^z_h-g^z \in H^1(\Omega) \subset W^{1,q}(\Omega)$. By H\"older inequality and since $0 < \etaa \leq 1$, we write that
\begin{align}
& |\abil_\etaa(u,g^z_h-g^z)|
\nonumber
\\
& \leq
\int_\Omega | \nabla u | \ | \nabla (g^z_h-g^z) | + \| b \|_{L^\infty} \int_\Omega | u | \ | \nabla (g^z_h-g^z) | + \int_\Omega | u | \ | g^z_h-g^z |
\nonumber
\\
& \leq 
\left( 1 + \| b \|_{L^\infty} \right) \left( \| \tau_z^{-1} \nabla u \|_{L^p(\Omega)} + \| \tau_z^{-1} u \|_{L^p(\Omega)} \right)
\nonumber
\\
& \qquad \times \left( \| \tau_z \nabla (g^z_h-g^z) \|_{L^q(\Omega)} + \| \tau_z (g^z_h-g^z) \|_{L^q(\Omega)} \right)
\label{eq:francois4}
\end{align}
where the function $\tau_z$ is defined by~\eqref{eq:def_tauz} below. Since $q \leq 2$, there exists $s>1$ such that $1 = 1/s + q/2$ (if $q=2$, we take $s=\infty$). Introducing real numbers $\alpha$ and $\beta$ such that $\alpha+\beta=1$, we write
$$
\left\| \tau_z \, \nabla (g^z_h-g^z) \right\|^q_{L^q(\Omega)} 
\leq
\left\| \tau_z^{\alpha q} \right\|_{L^s(\Omega)} \ \left\| \tau_z^{\beta q} \, |\nabla (g^z_h-g^z)|^q \right\|_{L^{2/q}(\Omega)}.
$$
We hence have
\begin{equation}
\left\| \tau_z \nabla (g^z_h-g^z) \right\|^2_{L^q(\Omega)} 
\leq
\left\| \tau_z^{\alpha q} \right\|^{2/q}_{L^s(\Omega)} \ \int_\Omega \tau_z^{2\beta} \, |\nabla (g^z_h-g^z)|^2.
\label{eq:francois2}
\end{equation}
Inspired by~\cite{rannacher-scott}, we take
\begin{equation}
\label{eq:def_tauz}
\tau_z = \weight_z^{(d+\lambda)/p}, \qquad 2 \beta = p,
\end{equation}
where $\lambda > 0$ and the parameter $\kappa \geq 1$ in the definition~\eqref{eq:def_sigma} of $\chi_z$ are defined in Lemma~\ref{lem:bound_M}. We hence have $\tau^{2\beta}_z = \weight_z^{d+\lambda}$. In view of the definition~\eqref{eq:def_M} of $M_{h,\lambda}$, we infer from~\eqref{eq:francois2} that
\begin{equation}
\label{eq:francois3}
\left\| \tau_z \nabla (g^z_h-g^z) \right\|^2_{L^q(\Omega)} \leq M^2_{h,\lambda} \, \left\| \tau_z^{\alpha q} \right\|^{2/q}_{L^s(\Omega)}.
\end{equation}
Our choice of $\beta$ implies that $\alpha = 1-\beta = 1 - p/2$ and $\alpha q s = -p$, hence $\dis \left\| \tau_z^{\alpha q} \right\|^s_{L^s(\Omega)} = \left\| \weight_z^{-(d+\lambda)} \right\|_{L^1(\Omega)}$, and therefore
$$
\left\| \tau_z^{\alpha q} \right\|^{2/q}_{L^s(\Omega)} 
= 
\left\| \weight_z^{-(d+\lambda)} \right\|^{2/(qs)}_{L^1(\Omega)}
=
\left\| \weight_z^{-(d+\lambda)} \right\|^{(p-2)/p}_{L^1(\Omega)}.
$$
If $s=\infty$ (which corresponds to the case $p=q=2$), the above estimate still holds, since $\alpha=0$ in that case. We thus get from~\eqref{eq:francois3} that
$$
\left\| \tau_z \nabla (g^z_h-g^z) \right\|_{L^q(\Omega)} \leq M_{h,\lambda} \ \left\| \weight_z^{-(d+\lambda)} \right\|^{(p-2)/(2p)}_{L^1(\Omega)}.
$$
We likewise have
$$
\left\| \tau_z (g^z_h-g^z) \right\|_{L^q(\Omega)} \leq M_{h,\lambda} \ \left\| \weight_z^{-(d+\lambda)} \right\|^{(p-2)/(2p)}_{L^1(\Omega)}.
$$
We then deduce from~\eqref{eq:francois4} that
\begin{multline}
|\abil_\etaa(u,g^z_h-g^z)|
\\
\leq
2 \left( 1+\| b \|_{L^\infty} \right) M_{h,\lambda} \ \left\| \weight_z^{-(d+\lambda)} \right\|^{(p-2)/(2p)}_{L^1(\Omega)} \ \left( \| \tau_z^{-1} \nabla u \|_{L^p(\Omega)} + \| \tau_z^{-1} u \|_{L^p(\Omega)} \right).
\label{eq:fm}
\end{multline}
Using~\eqref{eq:util} with $\theta=d+\lambda$, and noting that $\weight_z(x) = \weight_x(z)$, we get
$$
\left\| \weight_z^{-(d+\lambda)} \right\|_{L^1(\Omega)} \leq C_d \frac{1}{(\kappa h)^\lambda} \left( \frac{1}{d} + \frac{1}{\lambda} \right).
$$
Inserting this estimate in~\eqref{eq:fm}, using Lemma~\ref{lem:bound_M} and the fact that $\kappa \geq 1$, we obtain
$$
|\abil_\etaa(u,g^z_h-g^z)|
\leq
C_{\kappa,\lambda,\etaa} \, h^{\lambda/2} \ \left( \frac{1}{h^\lambda} \right)^{(p-2)/(2p)} \ \left( \| \tau_z^{-1} \nabla u \|_{L^p(\Omega)} + \| \tau_z^{-1} u \|_{L^p(\Omega)} \right)
$$
where $C_{\kappa,\lambda,\etaa}$ is independent of $h$, but depends on $\kappa$, $\lambda$ and $\etaa$. We denote it $C_\etaa$ in the sequel (since $\lambda$ is fixed and $\kappa$ a priori depends on $\etaa$). We integrate the $p$-th power of the above relation with respect to $z$:
\begin{eqnarray}
&& \int_\Omega |\abil_\etaa(u,g^z_h-g^z)|^p \, dz
\nonumber
\\
& \leq &
C_\etaa \, h^\lambda \left( \int_\Omega \| \tau_z^{-1} \nabla u \|^p_{L^p(\Omega)} \, dz + \int_\Omega \| \tau_z^{-1} u \|^p_{L^p(\Omega)} \, dz \right)
\nonumber
\\
& = &
C_\etaa \, h^\lambda \left( \int_\Omega |\nabla u(x)|^p \left[ \int_\Omega \tau_z^{-p}(x) dz \right] dx + \int_\Omega |u(x)|^p \left[ \int_\Omega \tau_z^{-p}(x) dz \right] dx \right).
\nonumber
\\
&& \label{eq:pth-power}
\end{eqnarray}
Let us bound $\dis \int_\Omega \tau_z^{-p}(x) dz$. We write, using again~\eqref{eq:util} with $\theta=d+\lambda$, that
$$
\forall x \in \Omega, \quad
\int_\Omega \tau_z^{-p}(x) dz
=
\int_\Omega \weight_z^{-(d+\lambda)}
\leq
\frac{C}{h^\lambda}.
$$
We hence infer from~\eqref{eq:pth-power} that
\begin{equation}
\int_\Omega |\abil_\etaa(u,g^z_h-g^z)|^p \, dz
\leq
C_\etaa \left( \int_\Omega |\nabla u|^p + \int_\Omega |u|^p \right)
\leq
C_\etaa \, \| u \|^p_{W^{1,p}(\Omega)}.
\label{eq:bound_second}
\end{equation}

\subsubsection{Bound on the first term of the right hand side of~\eqref{eq:split}}
\label{sec:22bis}

We denote
$$
F(z)
=
\int_\Omega \delta^z(x) \ \nu \cdot \nabla u(x) \, dx
$$
the first term of the right-hand side of~\eqref{eq:split}. Recalling that $\delta^z$ is supported in $K^z$, and using the H\"older inequality, we have, for any $z \in \Omega$, 
$$
\left| F(z) \right|
\leq
\| \nabla u \|_{L^p(K^z)} \| \delta^z \|_{L^q_x(K^z)}.
$$
We compute that
$$
\| \delta^z \|^q_{L^q_x(K^z)} 
\leq
\| \delta^z \|^q_{L^\infty_x(K^z)} \int_{K^z} dx
\leq
C h^{-qd} \, h^d,
$$
thus $\dis \| \delta^z \|_{L^q_x(K^z)} \leq C h^{-d/p}$, hence
$$
\left| F(z) \right|
\leq
C \| \nabla u \|_{L^p(K^z)} \, h^{-d/p}.
$$
We integrate the $p$-th power of the above estimate with respect to $z$:
\begin{eqnarray}
\int_\Omega \left| F(z) \right|^p dz
&\leq& 
C h^{-d} \ \int_\Omega dz \int_{K^z} |\nabla u(x) |^p \, dx
\nonumber
\\
&=&
C h^{-d} \ \sum_{T \in \mathcal{T}_h} \int_T dz \int_{K^z} |\nabla u(x) |^p \, dx
\nonumber
\\
&=&
C h^{-d} \ \sum_{T \in \mathcal{T}_h} \int_T dz \int_T |\nabla u(x) |^p \, dx
\nonumber
\\
&=&
C h^{-d} \ \sum_{T \in \mathcal{T}_h} h^d \int_T |\nabla u(x) |^p \, dx
\nonumber
\\
&=&
C \int_\Omega |\nabla u(x) |^p \, dx,
\label{eq:bound_premier}
\end{eqnarray}
where we have used that $K^z = T$ when $z \in T$.

\subsubsection{Proof of~\eqref{eq:main1}}
\label{sec:22ter}

The right-hand side of~\eqref{eq:split} is the sum of two functions of $z$, the $L^p$ norm of which is bounded from above (up to a multiplicative constant independent of $h$) by $\| u \|_{W^{1,p}(\Omega)}$, in view of~\eqref{eq:bound_second} and~\eqref{eq:bound_premier}. We thus get
$$
\| \nabla u_h \|_{L^p(\Omega)}
\leq
C_\etaa \, \| u \|_{W^{1,p}(\Omega)}
$$
for any $h < h_0(\etaa)$ (this restriction comes from the fact that, in Lemma~\ref{lem:bound_M}, we work in the regime $\kappa h \leq 1$ for some $\kappa$ that depends on $\etaa$), where $C_\etaa$ and $h_0(\etaa)$ a priori depend on $\etaa$. This concludes the proof of~\eqref{eq:main1}.

\subsection{Proof of~\eqref{eq:main2}}

To obtain a bound on $u-u_h$, we introduce the interpolant $I_h u \in \Sigma_h$. We then note that $u-I_h u \in H^1(\Omega)$, $u_h - I_h u \in \Sigma_h$ and that, in view of~\eqref{eq:galerkin_orth}, we have, for any $v \in \Sigma_h$,  
$$
\abil_\etaa\Big((u-I_h u) - (u_h - I_h u),v\Big) = 0.
$$
Furthermore, we observe that $u - I_h u \in W^{1,p}(\Omega)$. We are thus in position to write~\eqref{eq:main1}, that is
$$
\| \nabla (u_h-I_h u) \|_{L^p(\Omega)}
\leq
C_\etaa \, \| \nabla (u - I_h u) \|_{L^p(\Omega)}.
$$
Thus, we get that
$$
\| \nabla (u-u_h) \|_{L^p(\Omega)}
\leq
\| \nabla (u-I_h u) \|_{L^p(\Omega)} + \| \nabla (I_h u-u_h) \|_{L^p(\Omega)}
\leq
C_\etaa \, \| \nabla (u - I_h u) \|_{L^p(\Omega)}
$$
and we conclude using an approximation result (see e.g.~\cite[eq. (1.5)]{rannacher-scott}), stating that, for any $u \in W^{2,p}(\Omega)$, we have $\| u - I_h u \|_{W^{1,p}(\Omega)} \leq C h \| u \|_{W^{2,p}(\Omega)}$. This yields~\eqref{eq:main2}.

\section{Proof of Lemma~\ref{lem:bound_M}}
\label{sec:proof_lem4}

The proof of Lemma~\ref{lem:bound_M} relies on four technical results, that we state below and prove in Appendix~\ref{sec:proofs}. We next turn here to the proof of Lemma~\ref{lem:bound_M}.

\begin{proposition}
\label{prop:831}
Assume that $b \in (L^\infty(\Omega))^d$ and that $\kappa$ and $h$ are such that $\kappa h \leq 1$. Let $w \in H^1(\Omega)$ and $w_h \in \Sigma_h$ such that
\begin{equation}
\label{eq:galerkin_orth_bis}
\forall v \in \Sigma_h, \quad \abil_\etaa(v,w-w_h) = 0.
\end{equation}
We set $e=w-w_h$. Then there exists $C > 0$, independent of $\etaa$, $\kappa$ and $h$, such that
\begin{multline*}
\int_\Omega \weight_z^{d+\lambda} \, |\nabla e|^2 
\leq 
C \int_\Omega \weight_z^{d+\lambda-2} \ |e|^2 
\\
+ C \int_\Omega \weight_z^{d+\lambda}  \, |\nabla (w-I_h w)|^2 + \weight_z^{d+\lambda-2} \, |w-I_h w|^2.
\end{multline*}
\end{proposition}

\medskip

\begin{lemma}
\label{lem:837}
Assume that~\eqref{eq:regul} holds for any $q$ such that $2d/(d+2) < q < \infty$. Assume also that $2 \leq d \leq 3$ and that $\dis 0 < \lambda < 4-d$. Let $\zeta>0$ and $x_0 \in \Omega$, and set
\begin{equation}
\label{eq:def_sigma0}
\weight(x) = \sqrt{|x-x_0|^2 + \zeta^2}.
\end{equation}
Let $f \in H^1(\Omega)$. The solution $v \in H^1(\Omega)$ to the problem
\begin{equation}
  \label{eq:def_v}
\forall \phi \in H^1(\Omega), \qquad \abil_\etaa(v,\phi) = \int_\Omega f \, \phi
\end{equation}
satisfies
$$
\int_\Omega \weight^{-d-\lambda} \, \left| \nabla^2 v \right|^2 \leq C_\etaa \, \zeta^{-2} \int_\Omega \weight^{4-d-\lambda} \, \left( |f|^2 + |\nabla f |^2 \right),
$$
where $C_\etaa$ is independent of $x_0$ and $\zeta$ (but a priori depends on $\etaa$), and where $\sigma_1$ is the invariant measure defined by~\eqref{invariant_measure_droniou}--\eqref{invariant_measure_droniou-int-1}. In addition, $C_\etaa \leq C/\etaa^2$ for some $C$ independent of $\etaa$.
\end{lemma}

\begin{remark}
Using the same arguments as for the proof of Lemma~\ref{lem:837}, it is also possible to show that
$$
\int_\Omega \weight^{-d-\lambda} \, \left| \nabla^2 \left( v - \sigma_1 \fint_\Omega v \right) \right|^2 \leq C \zeta^{-2} \int_\Omega \weight^{4-d-\lambda} \, \left( |f|^2 + |\nabla f |^2 \right),
$$
where $C$ is independent of $\etaa$, $x_0$ and $\zeta$.
\end{remark}

\medskip

\begin{proposition}
\label{prop:835}
Assume that $b \in (L^\infty(\Omega))^d$ and that~\eqref{eq:regul} holds for any $q$ such that $2d/(d+2) < q < \infty$. Assume also that $2 \leq d \leq 3$ and that $\dis 0 < \lambda < 4-d$. 

Let $w \in H^1(\Omega)$ and $w_h \in \Sigma_h$ such that the Galerkin orthogonality~\eqref{eq:galerkin_orth_bis} holds. We set $e=w-w_h$. Then, for any $\eps>0$ small enough, there exists $\kappa_1(\eps,\etaa) \geq 1$ (which a priori depends on $\eps$ and $\etaa$) such that, for any $\kappa \geq \kappa_1(\eps,\etaa)$ and any $h$ such that $\kappa h \leq 1$, we have
$$
\int_\Omega \weight_z^{d+\lambda-2} \, |e|^2 
\leq 
8 \eps \int_\Omega \weight_z^{d+\lambda} \, |\nabla e|^2,
$$
where $\weight_z$ is defined by~\eqref{eq:def_sigma}.
\end{proposition}
The proof of Proposition~\ref{prop:835} shows that it is sufficient to take $\dis \kappa_1(\eps,\etaa) \geq \frac{C}{\eps \etaa}$ for some $C$ independent of $\eps$ and $\etaa$.

\medskip

\begin{lemma}
\label{lem:8311}
Assume that $b \in (L^\infty(\Omega))^d$, that~\eqref{eq:regul} holds for any $q$ such that $2d/(d+2) < q < \infty$, and that~\eqref{eq:regul_adjoint} holds. Assume also that $2 \leq d \leq 3$ and that $\dis 0 < \lambda < 4-d$. Consider $\weight$ defined by~\eqref{eq:def_sigma0} and assume that $\zeta \leq 1$. 
   
Let $f \in H^1_0(\Omega)$ and consider the solution $v \in H^1(\Omega)$ of the adjoint problem
\begin{equation}
\label{eq:eq_v}
\forall \phi \in H^1(\Omega), \qquad \abil_\etaa(\phi,v) = \int_\Omega (\nu \cdot \nabla f) \ \phi.
\end{equation}
Then $v$ satisfies
$$
\int_\Omega \weight^{d+\lambda} \, |\nabla^2 v |^2 \leq C \int_\Omega \weight^{d+\lambda} \, |\nabla f |^2 + C_\etaa \, \zeta^{-2} \int_\Omega \weight^{d+\lambda} \, f^2,
$$
where $C$ is independent of $\etaa$, $x_0$ and $\zeta$, and where $C_\etaa$ is independent of $x_0$ and $\zeta$ (and a priori depends on $\etaa$).
\end{lemma}
The proof of Lemma~\ref{lem:8311} shows that $C_\etaa \leq C/\etaa^2$ for some $C$ independent of $\etaa$, $x_0$ and $\zeta$.

\medskip

Thanks to the above results, we are in position to prove Lemma~\ref{lem:bound_M}.

\begin{proof}[Proof of Lemma~\ref{lem:bound_M}]

From the assumptions of Lemma~\ref{lem:bound_M}, we know that~\eqref{hyp_H0_debut}--\eqref{hyp_H_11_num}--\eqref{hyp_H_22_num} hold. We can thus take $\lambda$ such that $0 < \lambda < 1 \leq 4-d$, and all the assumptions of Proposition~\ref{prop:831}, Lemma~\ref{lem:837}, Proposition~\ref{prop:835} and Lemma~\ref{lem:8311} are fulfilled.

Let $g^z$ and $g^z_h$ be the solutions to~\eqref{eq:def_gz} and~\eqref{eq:def_gzh}. They satisfy the Galerkin orthogonality~\eqref{eq:galerkin_orth_bis}, namely $\abil_\etaa(v,g^z-g^z_h) = 0$ for any $v \in \Sigma_h$. Let $\eps > 0$ be small enough, as in Proposition~\ref{prop:835}. Combining Propositions~\ref{prop:831} and~\ref{prop:835}, we obtain that there exists $\kappa_1(\eps,\etaa) \geq 1$ (a priori depending on $\eps$ and $\etaa$) such that, for any $\kappa \geq \kappa_1(\eps,\etaa)$ and any $h$ such that $\kappa h \leq 1$,
\begin{align*}
& \int_\Omega \weight_z^{d+\lambda}|\nabla (g^z-g^z_h)|^2 + \weight_z^{d+\lambda-2}(g^z-g^z_h)^2
\\
&\leq
(C+1) \int_\Omega \weight_z^{d+\lambda-2} (g^z-g^z_h)^2
\\
& \qquad
+
C\left( \int_\Omega \weight_z^{d+\lambda-2} (g^z-I_h g^z)^2 + \int_\Omega \weight_z^{d+\lambda}|\nabla (g^z-I_h g^z)|^2\right)
\\
&\leq
8 (C+1) \eps \int_\Omega \weight_z^{d+\lambda} |\nabla (g^z-g^z_h)|^2
\\
& \qquad
+
C\left( \int_\Omega \weight_z^{d+\lambda-2} (g^z-I_h g^z)^2 + \int_\Omega \weight_z^{d+\lambda}|\nabla (g^z-I_h g^z)|^2\right)
\end{align*}
where $C$ is independent of $\eps$, $\kappa$ and $\etaa$. We pick $\eps$ such that $8 (C+1) \eps \leq 1/2$ and we obtain that
\begin{multline*}
\int_\Omega \weight_z^{d+\lambda}|\nabla (g^z-g^z_h)|^2 + \weight_z^{d+\lambda-2}(g^z-g^z_h)^2
\\
\leq
C\left( \int_\Omega \weight_z^{d+\lambda-2} (g^z-I_h g^z)^2 + \int_\Omega \weight_z^{d+\lambda}|\nabla (g^z-I_h g^z)|^2\right).
\end{multline*}
We have simply written that the error $g^z-g^z_h$ is bounded by the best approximation error. However, this is not a trivial estimate, as all errors are weighted.

\medskip

We next proceed as follows, successively using~\eqref{eq:inter1} and the fact that $h \leq \kappa h \leq \weight_z$:
\begin{multline}
\int_\Omega \weight_z^{d+\lambda}|\nabla (g^z-g^z_h)|^2 + \weight_z^{d+\lambda-2}(g^z-g^z_h)^2 
\\
\leq 
C \int_\Omega \big( \weight_z^{d+\lambda-2} h^4+\weight_z^{d+\lambda} h^2 \big) \ |\nabla^2 g^z|^2
\leq 
C h^2 \int_\Omega \weight_z^{d+\lambda} |\nabla^2 g^z|^2.
\label{eq:francois9}
\end{multline}
For the final estimate, we have used the regularity assumption~\eqref{eq:regul_adjoint} with $q=2$. Applying Lemma~\ref{lem:8311} with $f \equiv \delta^z \in C^\infty_0(\Omega)$ and $\weight \equiv \weight_z$, we obtain
\begin{align*}
\int_\Omega \weight_z^{d+\lambda} \, |\nabla^2 g^z|^2 
&\leq 
C \int_\Omega \weight_z^{d+\lambda} |\nabla \delta^z|^2 + \frac{C_\etaa}{\kappa^2 h^2} \int_\Omega \weight_z^{d+\lambda} (\delta^z)^2
\\
&\leq C \int_{K^z} \weight_z^{d+\lambda} |\nabla \delta^z|^2 + \frac{C_\etaa}{\kappa^2 h^2} \int_{K^z} \weight_z^{d+\lambda} (\delta^z)^2
\\
&\leq C \, h^{-d-2} \ \| \weight_z \|^{d+\lambda}_{L^\infty(K^z)} \left( 1 + \frac{C_\etaa}{\kappa^2} \right),
\end{align*}
where we have used that $\| \delta^z \|_{L^\infty(K^z)} \leq C h^{-d}$ and $\| \nabla \delta^z \|_{L^\infty(K^z)} \leq C h^{-d-1}$. Using that $\| \weight_z \|^2_{L^\infty(K^z)} \leq C h^2 (1+\kappa^2)$, we get that
\begin{equation}
\int_\Omega \weight_z^{d+\lambda} \, |\nabla^2 g^z|^2
\leq C_{\kappa,\lambda,\etaa} \, h^{\lambda -2},
\label{eq:francois10}
\end{equation}
where $C_{\kappa,\lambda,\etaa}$ depends on $\kappa$, $\lambda$ and $\etaa$ but not on $h$ (more precisely, since $1 \leq \kappa$, one can take $\dis C_{\kappa,\lambda,\etaa} = C \, \kappa^{d+\lambda} \, \etaa^{-2}$). Introduce
$$
\mathcal{M}_{h,\lambda}(z) = \sqrt{ \int_\Omega \weight_z^{d+\lambda} \Big( |g^z - g^z_h|^2 + |\nabla (g^z - g^z_h) |^2 \Big) },
$$
so that $\dis M_{h,\lambda} = \sup_z \mathcal{M}_{h,\lambda}(z)$. Using that $\weight_z$ is bounded (this is a consequence of the regime $\kappa h \leq 1$), we write, collecting~\eqref{eq:francois9} and~\eqref{eq:francois10}, that
$$
\mathcal{M}^2_{h,\lambda}(z)
\leq 
C \int_\Omega \weight_z^{d+\lambda}|\nabla (g^z-g^z_h)|^2 + \weight_z^{d+\lambda-2}(g^z-g^z_h)^2
\leq 
C_{\kappa,\lambda,\etaa} \, h^{\lambda}.
$$
Taking the supremum over $z$ yields the claimed bound on $M_{h,\lambda}$ and thus concludes the proof of Lemma~\ref{lem:bound_M}.
\end{proof}

\section{Technical proofs}
\label{sec:proofs}

We collect in this Appendix the proofs of Proposition~\ref{prop:831}, Lemma~\ref{lem:837}, Proposition~\ref{prop:835} and Lemma~\ref{lem:8311}.

\subsection{Proof of Proposition~\ref{prop:831}}

We set $e=w-w_h$, $\widetilde{e}=I_h w - w_h$ and $\psi=\weight_z^{d+\lambda} \, \widetilde{e}$. We have
\begin{align*}
&\int_\Omega \weight_z^{d+\lambda} \, |\nabla e|^2 + \etaa \int_\Omega \weight_z^{d+\lambda} e^2
\\
&= \int_\Omega \nabla (\weight_z^{d+\lambda}e) \cdot \nabla e - \int_\Omega e \nabla (\weight_z^{d+\lambda}) \cdot \nabla e + \etaa \int_\Omega \weight_z^{d+\lambda} e^2
\\
&= \abil_\etaa(\weight_z^{d+\lambda}e,e) - \int_\Omega (b\cdot\nabla e) \ \weight_z^{d+\lambda} e - \int_\Omega e \nabla (\weight_z^{d+\lambda}) \cdot \nabla e
\\
&= \abil_\etaa\big(\weight_z^{d+\lambda}(w-I_h w+\widetilde{e}),e\big) - \int_\Omega (b\cdot\nabla e) \, \weight_z^{d+\lambda} e - \int_\Omega e \nabla (\weight_z^{d+\lambda}) \cdot \nabla e
\\
&= \abil_\etaa\big(\weight_z^{d+\lambda}(w-I_h w),e\big)+ \abil_\etaa(\psi,e) - \int_\Omega (b\cdot\nabla e) \, \weight_z^{d+\lambda} e - \int_\Omega e \nabla (\weight_z^{d+\lambda}) \cdot \nabla e.
\end{align*}
Using the Galerkin orthogonality~\eqref{eq:galerkin_orth_bis} and the fact that $\dis \etaa \int_\Omega \weight_z^{d+\lambda} e^2 > 0$, and next the estimate~\eqref{eq:bound_sigma2}, we get
\begin{align*}
\int_\Omega \weight_z^{d+\lambda} \, |\nabla e|^2 
&\leq \abil_\etaa\big(\weight_z^{d+\lambda}(w-I_h w),e\big)+ \abil_\etaa(\psi-I_h \psi,e) + \|b\|_{L^\infty} \int_\Omega |\nabla e| \, \weight_z^{d+\lambda} \, |e| 
\\
& \qquad + \int_\Omega |\nabla e| \ |\nabla (\weight_z^{d+\lambda})| \ |e|
\\
&\leq \abil_\etaa\big(\weight_z^{d+\lambda}(w-I_h w),e\big)+ \abil_\etaa(\psi-I_h \psi,e) 
\\
& \qquad + (C + \|b\|_{L^\infty} \|\weight_z \|_{L^\infty} )\int_\Omega \weight_z^{d+\lambda-1} \, |\nabla e| \, |e|.
\end{align*}
Since we work in the regime $\kappa h \leq 1$, we have that, for any $z \in \Omega$,
\begin{equation}
\label{eq:bound_chi}
\|\weight_z \|_{L^\infty} \leq C,
\end{equation}
where $C$ only depends on $\Omega$. We deduce from the above estimate that
\begin{align}
\int_\Omega \weight_z^{d+\lambda} \, |\nabla e|^2
& \leq
\abil_\etaa\big(\weight_z^{d+\lambda}(w-I_h w),e\big) + \abil_\etaa(\psi-I_h \psi,e)
+ C \int_\Omega \weight_z^{d+\lambda-1} |\nabla e| \, |e|
\nonumber
\\
& \leq
\abil_\etaa\big(\weight_z^{d+\lambda}(w-I_h w),e\big)+ \abil_\etaa(\psi-I_h \psi,e) 
\nonumber
\\
& \qquad + \frac{1}{4} \int_\Omega \weight_z^{d+\lambda} \, |\nabla e|^2 + C \int_\Omega \weight_z^{d+\lambda-2} \, |e|^2. 
\label{eq:francois5}
\end{align}
We successively estimate the first two terms of~\eqref{eq:francois5}. For the first term, we write, using estimate~\eqref{eq:bound_sigma2}, the Cauchy Schwarz inequality and the Young inequality,
\begin{align}
&\left|\abil_\etaa\big(\weight_z^{d+\lambda}(w-I_h w),e\big)\right|
\nonumber
\\
&\leq C\int_\Omega |\nabla e| \Big( \weight_z^{d+\lambda} \, |\nabla (w-I_h w)| + \weight_z^{d+\lambda-1} \, |w-I_h w| \Big)
\nonumber
\\
&
\qquad + \| b \|_{L^\infty(\Omega)} \int_\Omega \weight_z^{d+\lambda} \ |w-I_h w| \ |\nabla e| + \etaa \int_\Omega \weight_z^{d+\lambda} \ |w-I_h w| \ |e|
\nonumber
\\
&\leq C \int_\Omega |\nabla e| \, \Big( \weight_z^{d+\lambda} \, |\nabla (w-I_h w)| + \weight_z^{d+\lambda-1} \, |w-I_h w| \Big) 
\nonumber
\\
&
\qquad + C \int_\Omega \weight_z^{d+\lambda-2} \ |w-I_h w| \ |e|
\nonumber
\\
&\leq 
C \left( \int_\Omega \weight_z^{d+\lambda} \, |\nabla e|^2 \right)^{1/2} 
\left(\int_\Omega \weight_z^{d+\lambda} |\nabla (w-I_h w)|^2 + \weight_z^{d+\lambda-2} |w-I_h w|^2 \right)^{1/2} 
\nonumber
\\
&
\qquad + C \left( \int_\Omega \weight_z^{d+\lambda-2} \ |w-I_h w|^2 \right)^{1/2} \left( \int_\Omega \weight_z^{d+\lambda-2} \ |e|^2 \right)^{1/2}
\nonumber
\\
&
\leq \frac{1}{4}\int_\Omega \weight_z^{d+\lambda} |\nabla e|^2 
+ C \int_\Omega \weight_z^{d+\lambda} |\nabla (w-I_h w)|^2 + \weight_z^{d+\lambda-2} |w-I_h w|^2
\nonumber
\\
& \qquad + C \int_\Omega \weight_z^{d+\lambda-2} |e|^2.
\label{eq:titi1}
\end{align}
Estimating the second term of~\eqref{eq:francois5} is done in a similar fashion:
\begin{align}
&|\abil_\etaa(\psi-I_h \psi,e)|
\nonumber
\\
&\leq 
\int_\Omega |\nabla e| \, |\nabla (\psi-I_h \psi)|
+ \| b \|_{L^\infty(\Omega)} \int_\Omega |\psi-I_h \psi| \, |\nabla e|
+ \etaa \int_\Omega |\psi-I_h \psi| \, |e|
\nonumber
\\
&\leq C \int_\Omega |\nabla e| \Big( |\nabla (\psi-I_h \psi)| + |\psi-I_h \psi| \Big) + \int_\Omega |\psi-I_h \psi| \, |e|
\nonumber
\\
&\leq C \left(\int_\Omega \weight_z^{d+\lambda} |\nabla e|^2 \right)^{1/2} 
\left(\int_\Omega \weight_z^{-d-\lambda} \Big( |\nabla (\psi-I_h \psi)|^2+|\psi-I_h \psi|^2 \Big) \right)^{1/2}
\nonumber
\\
&\qquad + \left(\int_\Omega \weight_z^{d+\lambda} |e|^2 \right)^{1/2} 
\left(\int_\Omega \weight_z^{-d-\lambda} |\psi-I_h \psi|^2 \right)^{1/2}
\nonumber
\\
&
\leq \frac{1}{4}\int_\Omega \weight_z^{d+\lambda} |\nabla e|^2
+ C \int_\Omega \weight_z^{-d-\lambda} \Big( |\nabla (\psi-I_h \psi)|^2+|\psi-I_h \psi|^2 \Big)
\nonumber
\\
&\qquad + C \int_\Omega \weight_z^{d+\lambda} |e|^2.
\label{eq:titi2bis}
\end{align}
Collecting~\eqref{eq:francois5}, \eqref{eq:titi1} and~\eqref{eq:titi2bis}, we obtain
\begin{align}
\frac{1}{4} \int_\Omega \weight_z^{d+\lambda} \, |\nabla e|^2 
&\leq 
C \int_\Omega \weight_z^{d+\lambda-2} \, |e|^2 
\nonumber
\\
& + C \int_\Omega \weight_z^{d+\lambda} |\nabla (w-I_h w)|^2 + \weight_z^{d+\lambda-2} |w-I_h w|^2
\nonumber
\\
&+ C \int_\Omega \weight_z^{-d-\lambda} \Big( |\nabla (\psi-I_h \psi)|^2 +|\psi-I_h \psi|^2 \Big).
\label{eq:francois6}
\end{align}
We now bound the last term of~\eqref{eq:francois6} by finite element estimation (note that $\psi \in H^1(\Omega)$ and $\psi_{|T} \in H^2(T)$ for any $T \in \mathcal{T}_h$, since $\weight_z$ belongs to $C^\infty(\Omega)$ and $\widetilde{e} \in \Sigma_h$; we are thus in position to use~\eqref{eq:inter1}):
\begin{eqnarray*}
&& \int_\Omega \weight_z^{-d-\lambda} \Big( |\nabla (\psi-I_h \psi)|^2 +|\psi-I_h \psi|^2 \Big)
\\
& \leq & 
C h^2 \sum_{T \in \mathcal{T}_h} \int_T \weight_z^{-d-\lambda} \, \left|\nabla^2 (\weight_z^{d+\lambda} \widetilde{e}) \right|^2 \quad \text{[estimate~\eqref{eq:inter1} and def. of $\psi$]}
\\
&\leq& 
C h^2 \int_\Omega \weight_z^{-d-\lambda} \Big( |\weight_z^{d+\lambda-2} \, \widetilde{e}|^2 + |\weight_z^{d+\lambda-1}|^2 \, | \nabla \widetilde{e}|^2 \Big),
\end{eqnarray*}
where, in the last line, we have used~\eqref{eq:bound_sigma2} and the fact that $\widetilde{e}$ is piecewise affine. We next use the inverse inequality~\eqref{eq:inter3} and the fact that $\weight_z^{-2} \leq h^{-2}$:
\begin{multline*}
\int_\Omega \weight_z^{-d-\lambda} \Big( |\nabla (\psi-I_h \psi)|^2 +|\psi-I_h \psi|^2\Big)
\\
\leq
C h^2 \int_\Omega \weight_z^{d+\lambda-4} \ |\widetilde{e}|^2
+
C \int_\Omega \weight_z^{d+\lambda-2} \, |\widetilde{e}|^2
\leq 
C \int_\Omega \weight_z^{d+\lambda-2} \, |\widetilde{e}|^2.
\end{multline*}
Since $\widetilde{e} = I_h w - w_h = I_h w - w + e$, we get
\begin{multline*}
\int_\Omega \weight_z^{-d-\lambda} \Big( |\nabla (\psi-I_h \psi)|^2 +|\psi-I_h \psi|^2 \Big)
\\
\leq
C \int_\Omega \weight_z^{d+\lambda-2} \, |e|^2 + C \int_\Omega \weight_z^{d+\lambda-2} \, (I_h w - w)^2.
\end{multline*}
Inserting this estimate in~\eqref{eq:francois6}, we obtain
\begin{multline*}
\frac{1}{4} \int_\Omega \weight_z^{d+\lambda} \, |\nabla e|^2 
\leq 
C \int_\Omega \weight_z^{d+\lambda-2} \, |e|^2 
\\
+ C \int_\Omega \weight_z^{d+\lambda} |\nabla (w-I_h w)|^2 + \weight_z^{d+\lambda-2} |w-I_h w|^2.
\end{multline*}
This concludes the proof of Proposition~\ref{prop:831}.

\subsection{Proof of Lemma~\ref{lem:837}}

We choose some $s$ such that
\begin{equation}
\label{eq:inter_s}
1 \leq \frac{2d}{d+2} < s < 2
\end{equation} 
and let $s^\star = sd/(d-s)$ (note that $s < 2 \leq d$). From the Sobolev injections, we know that there exists $C_s$ such that
\begin{equation}
\label{eq:poinc}
\forall g \in W^{1,s}(\Omega), \quad \| g \|_{L^{s^\star}(\Omega)} \leq C_s \| g \|_{W^{1,s}(\Omega)}.
\end{equation}

The function $f$ in the statement of Lemma~\ref{lem:837} belongs to $H^1(\Omega)$. Since $s<2$, we see that $f \in W^{1,s}(\Omega) \subset L^{s^\star}(\Omega)$. Since $\dis s > \frac{2d}{d+2}$, we have $\dis s^\star > 2 > \frac{2d}{d+2}$. 
We can thus use the regularity assumption~\eqref{eq:regul} for $s^\star$, from which we deduce that $v \in W^{2,s^\star}(\Omega)$. We set $q = s^\star/2 > 1$ and write a H\"older inequality with exponents $q$ and $q'$:
%
%
%
%
\begin{eqnarray}
\int_\Omega \weight^{-d-\lambda} \, \left| \nabla^2 v \right|^2 
&\leq&
\left( \int_\Omega \weight^{-(d+\lambda)q'} \right)^{1/q'} \left\| \nabla^2 v \right\|^2_{L^{2q}(\Omega)}
\nonumber
\\
&=&
\left( \int_\Omega \weight^{-(d+\lambda)q'} \right)^{1/q'} \left\| \nabla^2 v \right\|^2_{L^{s^\star}(\Omega)}
\nonumber
\\
&\leq & 
C \left( \frac{1}{\zeta^{q'(d+\lambda)-d}} \right)^{1/q'} \left\| v \right\|^2_{W^{2,s^\star}(\Omega)} \quad \text{[eq.~\eqref{eq:util}]}
\label{eq:maison2}
\end{eqnarray}
In view of~\eqref{eq:regul}, we have
$$
\left\| v \right\|_{W^{2,s^\star}(\Omega)}
\leq
\left\| v - \sigma_1 \fint_\Omega v \right\|_{W^{2,s^\star}(\Omega)} + C \left| \fint_\Omega v \right|
\leq
C \left\| f \right\|_{L^{s^\star}(\Omega)} + C \left| \fint_\Omega v \right|.
$$
Taking $\phi \equiv 1$ in~\eqref{eq:def_v}, we see that $\dis \etaa \fint_\Omega v = \fint_\Omega f$. Since $s^\star \geq 1$, we get
$$
\left\| v \right\|_{W^{2,s^\star}(\Omega)}
\leq
C_\etaa \left\| f \right\|_{L^{s^\star}(\Omega)}
\qquad \text{(with $C_\etaa \leq C/\etaa$).}
$$
Inserting this estimate in~\eqref{eq:maison2}, and next using~\eqref{eq:poinc} for the function $f \in W^{1,s}(\Omega)$, we deduce that
\begin{equation}
\int_\Omega \weight^{-d-\lambda} \, \left| \nabla^2 v \right|^2 
\leq
\frac{C_\etaa^2}{\zeta^{\lambda+d/q}} \ \| f \|^2_{L^{s^\star}(\Omega)}
\leq
\frac{C_\etaa^2}{\zeta^{\lambda+d/q}} \ \| f \|^2_{W^{1,s}(\Omega)}.
\label{eq:francois11}
\end{equation}
We now define $\overline{q} = 2/s$. Since $s<2$, we have $\overline{q} > 1$ and we can use the H\"older inequality with exponents $\overline{q}$ and $\overline{q}'$ to bound from above $\| \nabla f \|^s_{L^s(\Omega)}$ (and likewise for $\| f \|^s_{L^s(\Omega)}$):
\begin{eqnarray}
\| \nabla f \|^s_{L^s(\Omega)} 
&=&
\int_\Omega \weight^{-(4-d-\lambda)s/2} \ \ \weight^{(4-d-\lambda)s/2} \, |\nabla f|^s
\nonumber
\\
& \leq &
\left( \int_\Omega \weight^{-(4-d-\lambda)s\overline{q}'/2} \right)^{1/\overline{q}'} \ \left( \int_\Omega \weight^{(4-d-\lambda)\overline{q}s/2} \, |\nabla f|^{\overline{q}s} \right)^{1/\overline{q}} 
\nonumber
\\
& = &
\left( \int_\Omega \weight^{-(4-d-\lambda)s/(2-s)} \right)^{(2-s)/2} \ \left( \int_\Omega \weight^{4-d-\lambda} \, |\nabla f|^2 \right)^{1/\overline{q}}.
\label{eq:hoho}
\end{eqnarray}
We now observe that $\dis \frac{2d}{4-\lambda} < 2$ since $\lambda < 4-d$. Consequently, we can pick a real number $s$ satisfying~\eqref{eq:inter_s} and $\dis s > \frac{2d}{4-\lambda}$. This implies that $(4-d-\lambda)s/(2-s) > d$. In~\eqref{eq:hoho}, we are thus in position to use~\eqref{eq:util} with $\theta = (4-d-\lambda)s/(2-s)$. We thus obtain
$$
\| \nabla f \|^s_{L^s(\Omega)} 
\leq
C \zeta^{(-4+d+\lambda)s/2 + d(2-s)/2} \ \left( \int_\Omega \weight^{4-d-\lambda} \, |\nabla f|^2 \right)^{1/\overline{q}}
$$
and likewise for $\| f \|^s_{L^s(\Omega)}$. Inserting these estimates in~\eqref{eq:francois11}, we deduce that
$$
\int_\Omega \weight^{-d-\lambda} \, \left| \nabla^2 v \right|^2 
\leq
\frac{C_\etaa^2}{\zeta^{\lambda+d/q}} \zeta^{(-4+d+\lambda) + d(2-s)/s} \ \int_\Omega \weight^{4-d-\lambda} \, \left( |f|^2 + |\nabla f|^2 \right).
$$
We have $1/q = 2/s^\star = 2/s - 2/d$, so that $\lambda+d/q = \lambda + 2d/s - 2$ while $(-4+d+\lambda) + d(2-s)/s = -4+\lambda  + 2d/s$. We then obtain 
$$
\int_\Omega \weight^{-d-\lambda} \, \left| \nabla^2 v \right|^2
\leq
C_\etaa^2 \, {\zeta^{-2}} \ \int_\Omega \weight^{4-d-\lambda} \, \left( |f|^2 + |\nabla f|^2 \right),
$$
which concludes the proof of Lemma~\ref{lem:837}.

\subsection{Proof of Proposition~\ref{prop:835}}

Consider the problem~\eqref{eq:pb} for the right-hand side $\dis f = \weight_z^{d+\lambda-2} \, e$, which is indeed in $L^2(\Omega)$. We denote $v \in H^1(\Omega)$ its solution, and thus have
$$
\forall \phi \in H^1(\Omega), \quad \abil_\etaa(v,\phi) = \int_\Omega \weight_z^{d+\lambda-2} \, e \, \phi.
$$
Taking $e$ as a test function in the above problem, we get, using the Cauchy Schwarz inequality,
\begin{eqnarray*}
&& \int_\Omega \weight_z^{d+\lambda-2} \, |e|^2
\\
&=&
\abil_\etaa(v,e)
\\
&=&
\abil_\etaa(v-I_h v,e) \quad \text{[Galerkin orthogonality~\eqref{eq:galerkin_orth_bis}]}
\\
& \leq & 
C \left[ \int_\Omega \weight_z^{d+\lambda} \left( e^2 + |\nabla e|^2 \right) \right]^{1/2} 
\left[ \int_\Omega \weight_z^{-d-\lambda} \Big( |\nabla (v-I_h v)|^2+|v-I_h v|^2 \Big) \right]^{1/2}.
\end{eqnarray*}
Let $\eps > 0$. Using the Young inequality, we deduce that
\begin{multline}
\int_\Omega \weight_z^{d+\lambda-2} \, |e|^2
\\
\leq
\eps \int_\Omega \weight_z^{d+\lambda} \left( e^2 + |\nabla e|^2 \right)
+ \frac{C^2}{4\eps} \int_\Omega \weight_z^{-d-\lambda} \Big( |\nabla (v-I_h v)|^2+|v-I_h v|^2 \Big).
\label{eq:ahah1}
\end{multline}
We now use the regularity assumption~\eqref{eq:regul} for the problem~\eqref{eq:pb} with the right hand side $f$ defined above, which states that $v$ belongs to $H^2(\Omega)$ (note indeed that $\dis \frac{2d}{d+2} < 2$). We are thus in position to use the finite element estimate~\eqref{eq:inter1}. Inserting~\eqref{eq:inter1} in~\eqref{eq:ahah1}, we get
\begin{equation}
\label{eq:adele}
\int_\Omega \weight_z^{d+\lambda-2} \, e^2
\leq
\eps \int_\Omega \weight_z^{d+\lambda} \left( e^2 + |\nabla e|^2 \right)
+ \frac{C h^2}{\eps} \int_\Omega \weight_z^{-d-\lambda} |\nabla^2 v|^2.
\end{equation}
We now use Lemma~\ref{lem:837}, with $\weight = \weight_z$, noting that the right-hand side $\dis f = \weight_z^{d+\lambda-2} \, e$ is in $H^1(\Omega)$. We thus deduce from~\eqref{eq:adele}, successively using Lemma~\ref{lem:837} and estimate~\eqref{eq:bound_sigma2}, that
\begin{eqnarray*}
&& \int_\Omega \weight_z^{d+\lambda-2} \, e^2
\\
&\leq& 
\eps \int_\Omega \weight_z^{d+\lambda} \left( e^2 + |\nabla e|^2 \right)
+ \frac{C h^2}{\eps} \frac{C_\etaa}{\kappa^2 h^2} \int_\Omega \weight_z^{4-d-\lambda} \left( |\weight_z^{d+\lambda-2} \, e |^2 + |\nabla (\weight_z^{d+\lambda-2} \, e) |^2 \right)
\\
&\leq& 
\eps \int_\Omega \weight_z^{d+\lambda} \left( e^2 + |\nabla e|^2 \right)
+\frac{C_\etaa}{\eps \kappa^2} \left( \int_\Omega \weight_z^{d+\lambda} |\nabla e |^2 + \int_\Omega \weight_z^{d+\lambda-2} e^2 \right),
\end{eqnarray*}
where $C_\etaa$ only depends on $\etaa$ (and satisfies $C_\etaa \leq C/\etaa^2$). For any fixed $\eps$, we take $\kappa_1(\eps,\etaa) \geq 1$ such that, when $\kappa \geq \kappa_1(\eps,\etaa)$, we have $\dis \frac{C_\etaa}{\eps \kappa^2} \leq \min(\eps,1/2)$. We thus deduce that, for any $\kappa \geq \kappa_1(\eps,\etaa)$,
\begin{eqnarray*}
\frac{1}{2} \int_\Omega \weight_z^{d+\lambda-2} \, e^2
&\leq&
\eps \int_\Omega \weight_z^{d+\lambda} e^2 
+
2 \eps\int_\Omega \weight_z^{d+\lambda} |\nabla e|^2
\\
&\leq&
\eps \|\weight_z \|_{L^\infty}^2 \int_\Omega \weight_z^{d+\lambda-2} e^2 
+
2 \eps\int_\Omega \weight_z^{d+\lambda} |\nabla e|^2.
\end{eqnarray*}
Since we work in the regime $\kappa h \leq 1$, we are in position to use~\eqref{eq:bound_chi}, and thus $\eps \|\weight_z \|_{L^\infty}^2 \leq C \eps$ for a constant $C$ that only depends on $\Omega$. Taking $\eps$ small enough (namely such that $C \eps \leq 1/4$), we get the claimed bound. This concludes the proof of Proposition~\ref{prop:835}.

\subsection{Proof of Lemma~\ref{lem:8311}}

Since~\eqref{eq:regul_adjoint} holds and $\nu \cdot \nabla f \in L^2(\Omega)$, we have that $v \in H^2(\Omega)$. Expanding the expression $\nabla^2 \left( \weight^{(d+\lambda)/2} \, v \right)$, we find (using~\eqref{eq:bound_sigma2} for $\weight$ rather than $\weight_z$) that
\begin{equation}
\weight^{d+\lambda} \, |\nabla^2 v|^2 \leq \left| \nabla^2 \left(\weight^{(d+\lambda)/2} v \right) \right|^2 + C \left( \weight^{d+\lambda-2} \, |\nabla v|^2 + \weight^{d+\lambda-4} \, v^2 \right).
\label{eq:integ}
\end{equation}
We now identify an equation satisfied by $\weight^{(d+\lambda)/2} v$, which will be useful to estimate its second derivatives. For any $\phi \in H^1(\Omega)$, we have
\begin{align*}
\abil_\etaa\left(\phi,\weight^{(d+\lambda)/2} v\right) 
&= 
\int_\Omega \nabla \phi \cdot \nabla \left(\weight^{(d+\lambda)/2} v \right) + \phi \, b \cdot \nabla \left(\weight^{(d+\lambda)/2} \, v \right) + \etaa \, \phi \, \weight^{(d+\lambda)/2} v 
\\
&=
\abil_\etaa\left(\weight^{(d+\lambda)/2}\phi,v\right) + \int_\Omega \nabla \phi \cdot \nabla \left(\weight^{(d+\lambda)/2} v \right) + \phi \, b \cdot \nabla \left(\weight^{(d+\lambda)/2} \, v \right) 
\\
& \qquad -\int_\Omega \nabla \left(\weight^{(d+\lambda)/2} \phi\right) \cdot \nabla v - \int_\Omega \weight^{(d+\lambda)/2} \phi \, b \cdot \nabla v
\\
&=
\abil_\etaa\left(\weight^{(d+\lambda)/2}\phi,v\right) + \int_\Omega v \nabla \phi \cdot \nabla \left(\weight^{(d+\lambda)/2}\right) 
\\
& \qquad - \int_\Omega \phi \nabla \left(\weight^{(d+\lambda)/2}\right) \cdot \nabla v + \int_\Omega \phi \, v \, b \cdot \nabla \left(\weight^{(d+\lambda)/2} \right)
\\
&= 
\int_\Omega (\nu\cdot\nabla f) \, \weight^{(d+\lambda)/2} \phi + \int_\Omega v \nabla \phi \cdot \nabla \left(\weight^{(d+\lambda)/2}\right)
\\
& \qquad - \int_\Omega \phi \nabla \left(\weight^{(d+\lambda)/2}\right) \cdot \nabla v + \int_\Omega \phi \, v \, b \cdot \nabla \left(\weight^{(d+\lambda)/2}\right)
\\
&= \int_\Omega F \, \phi + \int_{\partial \Omega} G \, \phi,
\end{align*}
where 
$$
F=\weight^{(d+\lambda)/2} \, (\nu\cdot\nabla f) - \text{div} \left[ v \nabla \left(\weight^{(d+\lambda)/2}\right)\right] - \nabla \left( \weight^{(d+\lambda)/2} \right) \cdot \nabla v + v \, b \cdot \nabla \left( \weight^{(d+\lambda)/2} \right)
$$
and
$$
G = v \, n \cdot  \nabla \left( \weight^{(d+\lambda)/2} \right).
$$
Let $\zeta = \weight^{(d+\lambda)/2} v$. We see that $\zeta \in H^1(\Omega)$ and is such that, for any $\phi \in H^1(\Omega)$,
$$
\abil_\etaa(\phi,\zeta) = \int_\Omega F \, \phi + \int_{\partial \Omega} G \, \phi.
$$
Due to the presence of $G$, we cannot directly use the regularity result~\eqref{eq:regul_adjoint}. We are instead going to use Lemma~\ref{theorem_girault}, which states a regularity result for (non-homogeneous) Neumann problems. We write that $\zeta$ satisfies
$$
-\Delta \zeta + b \cdot \nabla \zeta + \etaa \zeta = F \quad \text{in $\Omega$}, 
\qquad
\nabla \zeta\cdot n = G \quad\text{on $\partial\Omega$},
$$
that we recast in the form
\begin{equation}
\label{eq:zeta}
-\Delta \zeta = \widetilde{F} \quad \text{in $\Omega$}, 
\qquad
\nabla \zeta\cdot n = G \quad\text{on $\partial\Omega$},
\end{equation}
with $\widetilde{F} = F - b \cdot \nabla \zeta - \etaa \zeta$, that is
\begin{multline*}
\widetilde{F}
=
\weight^{(d+\lambda)/2} \, (\nu\cdot\nabla f) - \text{div} \left[ v \nabla \left(\weight^{(d+\lambda)/2}\right)\right] \\ - \nabla \left( \weight^{(d+\lambda)/2} \right) \cdot \nabla v - \weight^{(d+\lambda)/2} b \cdot \nabla v - \etaa \weight^{(d+\lambda)/2} v.
\end{multline*}
We wish to use Lemma~\ref{theorem_girault} for the problem~\eqref{eq:zeta}. Since $f \in H^1(\Omega)$, $\weight \in C^\infty$, $v \in H^1(\Omega)$ and $b \in (L^\infty(\Omega))^d$, we see that $\widetilde{F} \in L^2(\Omega)$. We have $v \in H^1(\Omega)$ thus $G \in H^{1/2}(\partial \Omega)$. We are thus in position to use Lemma~\ref{theorem_girault} with $p=2$ on~\eqref{eq:zeta} (see also~\cite[Theorem 3.12 and Remark 3.13]{EG}), which implies that
\begin{equation}
\label{eq:hoho2}
\|\nabla^2 \zeta \|_{L^2(\Omega)}
=
\left\|\nabla^2 \left( \weight^{(d+\lambda)/2} v \right) \right\|_{L^2(\Omega)} 
\leq C \left( \left\| \widetilde{F} \right\|_{L^2(\Omega)} + \| G \|_{H^{1/2}(\partial \Omega)} \right)
\end{equation}
where $C$ is of course independent of $\etaa$. We integrate~\eqref{eq:integ} and use~\eqref{eq:hoho2}:
$$
\int_\Omega \weight^{d+\lambda}|\nabla^2 v|^2
\leq 
C \left( \left\| \widetilde{F} \right\|^2_{L^2(\Omega)} + \|G\|^2_{H^{1/2}(\partial \Omega)} + \int_\Omega \weight^{d+\lambda-2} \, |\nabla v|^2+\int_\Omega \weight^{d+\lambda-4} \, v^2 \right).
$$
Since we work in the regime $\zeta \leq 1$, we have that $\|\weight \|_{L^\infty} \leq C$ for some $C$ that only depends on $\Omega$. Using in addition the bounds~\eqref{eq:bound_sigma2}, we deduce from the above estimate that there exists $C$ independent of $\etaa$, $x_0$ and $\zeta$ such that
\begin{equation}
\int_\Omega \weight^{d+\lambda}|\nabla^2 v|^2
\leq
C \|G\|^2_{H^1(\Omega)} + C \int_\Omega \weight^{d+\lambda} \, (\nu \cdot \nabla f)^2 + \weight^{d+\lambda-2} \, |\nabla v|^2 + \weight^{d+\lambda-4} \, v^2.
\label{eq:francois7}
\end{equation}
For the first term above, we see that $|G| \leq C \, |v| \, |\weight|^{(d+\lambda)/2-1}$, thus $\dis \| G\|^2_{L^2(\Omega)} \leq C \int_\Omega \weight^{d+\lambda-2} \, v^2 \leq C \int_\Omega \weight^{d+\lambda-4} \, v^2$. In addition, we have
\begin{eqnarray*}
|\nabla G| 
&\leq &
|\nabla v| \, \left| \nabla \left(\weight^{(d+\lambda)/2}\right) \right| + |v| \, |\nabla n| \, \left| \nabla \left(\weight^{(d+\lambda)/2} \right) \right| + |v| \, \left| \nabla^2 \left( \weight^{(d+\lambda)/2} \right) \right|
\\
&\leq &
C \left( |\nabla v| \ \weight^{(d+\lambda)/2-1} + |v| \ \weight^{(d+\lambda)/2-1} + |v| \ \weight^{(d+\lambda)/2-2} \right)
\end{eqnarray*}
and thus
$$
\| \nabla G\|^2_{L^2(\Omega)} \leq C \int_\Omega \weight^{d+\lambda-2} \, |\nabla v|^2 + \weight^{d+\lambda-4} \, v^2. 
$$
We hence deduce from~\eqref{eq:francois7} that
\begin{equation}
\int_\Omega \weight^{d+\lambda}|\nabla^2 v|^2
\leq 
C \left( \int_\Omega \weight^{d+\lambda} \, (\nu \cdot \nabla f)^2 + \weight^{d+\lambda-4} \, v^2 + \weight^{d+\lambda-2} \, |\nabla v|^2 \right).
\label{term_3}
\end{equation}
We are now left with bounding the two last terms in~\eqref{term_3} in terms of $f$. We start with the last term, and write
\begin{align*}
& \int_\Omega \weight^{d+\lambda-2} \, |\nabla v|^2
\\
&=
\int_\Omega \nabla \left( \weight^{d+\lambda-2} v \right) \cdot \nabla v - \int_\Omega v \, \nabla \left( \weight^{d+\lambda-2} \right) \cdot \nabla v
\\
&=
\abil_\etaa\left(\weight^{d+\lambda-2} v,v\right) - \int_\Omega \weight^{d+\lambda-2} v \, b \cdot \nabla v - \etaa \int_\Omega \weight^{d+\lambda-2} v^2 - \int_\Omega v \nabla \left(\weight^{d+\lambda-2}\right) \cdot \nabla v.
\end{align*}
Using~\eqref{eq:eq_v}, we see that
$$
\abil_\etaa\left(\weight^{d+\lambda-2} v,v\right) 
= 
\int_\Omega (\nu\cdot\nabla f) \, \weight^{d+\lambda-2} v.
$$
We thus get that
\begin{multline*}
\int_\Omega \weight^{d+\lambda-2} |\nabla v|^2
\\
=
\int_\Omega (\nu\cdot\nabla f) \, \weight^{d+\lambda-2} v - \int_\Omega \weight^{d+\lambda-2} v \, b \cdot \nabla v - \etaa \int_\Omega \weight^{d+\lambda-2} v^2 - \int_\Omega v \nabla \left(\weight^{d+\lambda-2}\right) \cdot \nabla v.
\end{multline*}
We next proceed using the Young inequality and~\eqref{eq:bound_sigma2}:
\begin{eqnarray*}
&& \int_\Omega \weight^{d+\lambda-2} |\nabla v|^2
\\
& \leq &
\frac{1}{2} \int_\Omega \weight^{d+\lambda} |\nu \cdot \nabla f|^2 + \frac{1}{2} \int_\Omega \weight^{d+\lambda-4} v^2
+
\frac{1}{4} \int_\Omega \weight^{d+\lambda-2} |\nabla v|^2 + \int_\Omega \weight^{d+\lambda-2} |b|^2 v^2
\\
&& \qquad 
+
\frac{1}{4} \int_\Omega \weight^{d+\lambda-2} |\nabla v|^2 + C \int_\Omega \weight^{d+\lambda-4} v^2,
\end{eqnarray*}
which implies that
\begin{equation}
\int_\Omega \weight^{d+\lambda-2} |\nabla v|^2
\leq
C \int_\Omega \weight^{d+\lambda} |\nabla f|^2 + C \int_\Omega \weight^{d+\lambda-4} v^2.
\label{term_2}
\end{equation}
We now turn to the second term of~\eqref{term_3}. We pick some $\dis P' > \max\left(\frac{d}{2},\frac{d}{4-d-\lambda}\right)$ (note that $4-d-\lambda>0$) and write the H\"older's inequality with $P'$ and its conjugate exponent $P$ (note that $P'>d/2\geq 1$):
\begin{align}
\int_\Omega \weight^{d+\lambda-4} v^2
& \leq \left(\int_\Omega \weight^{(d+\lambda-4)P'} \right)^{1/P'} \left(\int_\Omega v^{2P}\right)^{1/P}
\nonumber
\\
& \leq C \zeta^{(d+\lambda-4) + d/P'} \left(\int_\Omega v^{2P}\right)^{1/P},
\label{eq:bidule3}
\end{align}
where we have used~\eqref{eq:util} with $\theta = P'(4-d-\lambda)$, which is indeed larger than $d$. Note that the last factor of~\eqref{eq:bidule3} is finite, as we have $v \in H^2(\Omega) \subset L^\infty(\Omega)$ (recall that $d \leq 3$).

We next use a duality argument to bound $\| v \|_{L^{2P}(\Omega)}$ in terms of $f$. Let $w \in H^1(\Omega)$ solve~\eqref{eq:pb} with a right-hand side equal to $\text{sign}(v) \, |v|^{2P-1}$:
\begin{equation}
\label{eq:bidule5}
\text{For any $\phi \in H^1(\Omega)$, \quad $\abil_\etaa(w,\phi) = \int_\Omega \text{sign}(v) \, |v|^{2P-1} \, \phi$}.
\end{equation}
Taking $v$ as test function in~\eqref{eq:bidule5}, we get
\begin{align*}
\| v \|_{L^{2P}}^{2P} 
&=
\int_\Omega \left( \text{sign}(v) \, |v|^{2P-1} \right) v
\\
&=
\abil_\etaa(w,v) \quad \text{[def. of $w$]}
\\
&= \int_\Omega (\nu\cdot \nabla f) \, w \quad \text{[def.~\eqref{eq:eq_v} of $v$]}
\\
&= - \int_\Omega f \, (\nu \cdot \nabla w) \quad \text{[int. by parts and $f \in H^1_0(\Omega)$]}
\\
&\leq 
\|f\|_{L^r(\Omega)} \| \nabla w \|_{L^{r'}(\Omega)}
\end{align*}
with $\dis r=\frac{2Pd}{2P+d}$. Note that $r>1$ since $\dis P>1\geq \frac{d}{2(d-1)}$. We have that $\dis r'=\left(1-\frac{1}{2P}-\frac{1}{d}\right)^{-1}$, and we note that $W^{1,2P/(2P-1)}(\Omega) \subset L^{r'}(\Omega)$. Using that Sobolev injection, we deduce from above that
\begin{equation}
\label{eq:bidule4}  
\| v \|_{L^{2P}}^{2P}
\leq C \| f \|_{L^r(\Omega)} \| \nabla w \|_{W^{1,2P/(2P-1)}(\Omega)}.
\end{equation}
We now bound from above $\| \nabla w \|_{W^{1,2P/(2P-1)}(\Omega)}$ using the regularity~\eqref{eq:regul}, which indeed holds since $2P/(2P-1) > 2d/(d+2)$ (this condition is equivalent to the condition $P' > d/2$, which we have enforced when choosing $P'$). We thus write
\begin{align*}
\| \nabla w \|_{W^{1,2P/(2P-1)}(\Omega)}
& \leq
\| w \|_{W^{2,2P/(2P-1)}(\Omega)}
\\
& \leq
\left\| w - \sigma_1 \fint_\Omega w \right\|_{W^{2,2P/(2P-1)}(\Omega)} + C \left| \fint_\Omega w \right|
\\
& \leq C \left\| \text{sign}(v) |v|^{2P-1} \right\|_{L^{2P/(2P-1)}(\Omega)} + C \left| \fint_\Omega w \right|.
\end{align*}
Taking $\phi \equiv 1$ in~\eqref{eq:bidule5}, we obtain $\dis \etaa \fint_\Omega w = \fint_\Omega \text{sign}(v) \, |v|^{2P-1}$, and we thus deduce from above that
$$
\| \nabla w \|_{W^{1,2P/(2P-1)}(\Omega)} \leq C_\etaa \left\| \text{sign}(v) |v|^{2P-1} \right\|_{L^{2P/(2P-1)}(\Omega)} \qquad \text{(with $C_\etaa \leq C/\etaa$).}
$$
Inserting this estimate in~\eqref{eq:bidule4}, we get
\begin{align*}
\| v \|_{L^{2P}}^{2P}
&\leq
C_\etaa \| f \|_{L^r(\Omega)} \left\| \text{sign}(v) |v|^{2P-1} \right\|_{L^{2P/(2P-1)}(\Omega)}
\\
& =
C_\etaa \|f\|_{L^r(\Omega)} \| v \|^{2P-1}_{L^{2P}(\Omega)}.
\end{align*}
Since $v \in L^{2P}(\Omega)$, we deduce that $\dis \| v \|_{L^{2P}(\Omega)} \leq C_\etaa \|f\|_{L^r(\Omega)}$. Inserting this in~\eqref{eq:bidule3}, we obtain
\begin{align*}
\int_\Omega \weight^{d+\lambda-4} v^2
& \leq 
C_\etaa^2 \ \zeta^{2d(1-1/r)+\lambda -2} \ \|f\|^2_{L^r(\Omega)} \quad \text{[Writing $P'$ in terms of $r$]}
\\
& \leq C_\etaa^2 \ \zeta^{2d(1-1/r)+\lambda -2} \left(\int_\Omega \weight^{d+\lambda} f^2\right) \left(\int_\Omega \weight^{-(d+\lambda)r/(2-r)}\right)^{(2-r)/r}
\end{align*}
where we have eventually used a H\"older inequality with $\overline{q}=2/r$. Note that $\overline{q} > 1$ (that is, $r<2$) as a consequence of the fact that $P'>d/2$. We next see that $\dis r > 1 > \frac{2d}{2d+\lambda}$, which implies that $\dis (d+\lambda)r/(2-r) > d$, so we are in position to use~\eqref{eq:util}, which yields
\begin{align}
\int_\Omega \weight^{d+\lambda-4} v^2
&\leq C_\etaa^2 \ \zeta^{2d(1-1/r)+\lambda -2} \left(\int_\Omega \weight^{d+\lambda} f^2\right) \zeta^{-(d+\lambda)+d(2-r)/r}
\nonumber
\\
&\leq C_\etaa^2 \ \zeta^{-2} \int_\Omega \weight^{d+\lambda}f^2. 
\label{term_1}
\end{align}
Collecting~\eqref{term_3}, \eqref{term_2} and~\eqref{term_1} yields the desired estimate and concludes the proof of Lemma~\ref{lem:8311}.

\bibliographystyle{amsplain}
\bibliography{biblio-inv-measure}
\end{document}